\documentclass[a4paper, 10pt, notitlepage]{article}

\usepackage{amsthm, amsmath, amssymb, latexsym}
\usepackage{mathrsfs,cite}

\theoremstyle{plain}
\newtheorem{theorem}{Theorem}[section]	
\newtheorem{lemma}[theorem]{Lemma}	
\newtheorem{corollary}[theorem]{Corollary}
\newtheorem{proposition}[theorem]{Proposition}

\theoremstyle{definition}
\newtheorem{definition}{Definition}
\newtheorem{example}{Example}

\theoremstyle{remark}
\newtheorem{remark}{Remark}

\DeclareMathOperator{\dom}{dom}
\DeclareMathOperator{\cl}{cl}
\DeclareMathOperator{\argmin}{arg\,min}

\author{M.V. Dolgopolik}
\title{A unifying theory of exactness of linear penalty functions II: parametric penalty functions}

\begin{document}

\maketitle

\begin{abstract}
In this article we develop a general theory of exact parametric penalty functions for constrained optimization
problems. The main advantage of the method of parametric penalty functions is the fact that a parametric penalty
function can be both smooth and exact unlike the standard (i.e. non-parametric) exact penalty functions that are always
nonsmooth. We obtain several necessary and/or sufficient conditions for the exactness of parametric penalty functions,
and for the zero duality gap property to hold true for these functions. We also prove some convergence results for 
the method of parametric penalty functions, and derive necessary and sufficient conditions for a parametric penalty
function to not have any stationary points outside the set of feasible points of the constrained optimization
problem under consideration. In the second part of the paper, we apply the general theory of exact parametric penalty
functions to a class of parametric penalty functions introduced by Huyer and Neumaier, and to smoothing
approximations of nonsmooth exact penalty functions. The general approach adopted in this article allowed us to unify
and significantly sharpen many existing results on parametric penalty functions.
\end{abstract}

\section{Introduction}

One of the main approaches to finding a local or global minimum of a constrained optimization problem consists of the
reduction of this problem to a sequence of unconstrained optimization problems or to a single unconstrained
optimization problem whose locally (or globally) optimal solutions coincide with locally (or globally) optimal solutions
of the original problem. In turn, one of the main methods of the reduction of a constrained optimization problem to a
single unconstrained optimization problem is the exact penalty method. The exact penalty method was proposed by Eremin
\cite{Eremin} and Zangwill \cite{Zangwill} in the 1960s, and, later on, became a standard tool of constrained
optimization \cite{Bertsekas,EvansGouldTolle,HanMangasarian,DiPilloGrippo,DiPilloGrippo2,DiPillo,ExactBarrierFunc,
RubinovYang,WuBaiYang,Demyanov,Zaslavski,Dolgopolik}.

The exact penalty method allows one to replace a constrained optimization problem by an equivalent unconstrained
optimization problem with necessarily nonsmooth objective function. However, the nonsmoothness of exact penalty
functions makes them less attractive (especially for practitioners who are often not familiar with efficient methods
for solving nonsmooth optimization problems) than other methods of constrained optimization. There exist two approaches
to overcome this difficulty within the theory of exact penalty functions. The first one is based on the use of smoothing
approximations of nonsmooth exact penalty functions
\cite{Pinar,WuBaiYang,MengDangYang,Liu,LiuzziLucidi,MengLiYang,XuMengSunShen,Lian,XuMengSunHuangShen}, while the second
approach was developed by Huyer and Neumaier in \cite{HuyerNeumaier}. Throughout this article, we refer to the penalty
function proposed in \cite{HuyerNeumaier} as a \textit{singular exact penalty function}, since one achieves smoothness
of this exact penalty function via the introduction of a singular term into the definition of this function. Recently,
singular exact penalty functions have attracted a lot of attention of researchers
\cite{Bingzhuang,WangMaZhou,Dolgopolik_OptLet,Dolgopolik_OptLet2}, and were successfully applied to various
constrained optimization problems \cite{YuTeoZhang,LiYu,JianLin,MaLiYiu, YuTeoBai,LinWuYu,LinLoxton}. It should be noted
that the main feature of both smoothing approximations of exact penalty functions and singular exact penalty functions
is the fact that they depend on some additional parameters apart from the penalty parameter. Thus, both smoothing
approximations and singular exact penalty functions are \textit{parametric penalty functions}.

The main goal of this article is to develop a general theory of exact \textit{parametric} penalty functions that allows
one not only to unify and generalize existing results on parametric penalty functions, but also to better
understand capabilities as well as limitations of the method of exact \textit{parametric} penalty functions. Apart from
the general theory of exactness of parametric penalty functions, which can be viewed as an extension of the main
results on exactness of standard (i.e. non-parametric) penalty function \cite{Dolgopolik} to the parametric case, we
also present some other results closely related to the theory of parametric penalty functions. Namely, we obtain
necessary and sufficient conditions for a parametric penalty function to not have any stationary (i.e. critical) points
outside the set of feasible points of the initial constrained optimization problem. We also study the zero duality gap
property, and obtain several results on convergence of the method of parametric penalty functions. 

Note that parametric penalty functions can be viewed as a particular case of separation functions that are studied
within the image space analysis (see~\cite{Giannessi_book,Mastroeni,LiFengZhang,ZhuLi,XuLi} and references
therein). However, surprisingly, no results of this paper (even the necessary and sufficient condition for the zero
duality gap property) can be derived from the general results on separation functions known to the author (see also
Remark~\ref{Rmrk_ZeroDualityGap_ImageSpaceAnalysis} below).

The paper is organised as follows. In Section~\ref{Sect_Preliminaries} we present some preliminary material that
is used throughout the article. A general theory of exact parametric penalty functions is developed in
Section~\ref{Sect_ParametricPenaltyFunctions}. In Section~\ref{Section_SingPenFunc} we apply this theory to singular
penalty functions. The main results of this section sharpen and/or generalize many existing results on singular penalty
functions. Finally, in Section~\ref{Section_SmoothingPenFunc} we develop a theory of approximations of exact penalty
functions within the framework of the theory of exact parametric penalty functions. The theory of approximations of
exact penalty functions presented in Section~\ref{Section_SmoothingPenFunc} unifies and generalizes most of existing
results on smoothing approximations of nonsmooth exact penalty functions.

\section{Preliminaries}
\label{Sect_Preliminaries}

In this section, we recall the notions of the rate of steepest descent/ascent of a function defined on a metric
space (see, e.g., \cite{Demyanov,DemyanovRSD}), and derive some calculus rules for the rate of steepest
descent/ascent that will be used throughout the article. 

Let $(X, d)$ be a metric space, and $A \subset X$ be a nonempty set. Denote, as usual, 
$\overline{\mathbb{R}} = \mathbb{R} \cup \{ + \infty \} \cup \{ - \infty \}$ and $\mathbb{R}_+ = [0, + \infty)$.
For any function $f \colon C \to \overline{\mathbb{R}}$ with $C \subset X$ being a nonempty set, denote 
$\dom f = \{ x \in C \mid |f(x)| < + \infty \}$.

\begin{definition}
Let $U \subset X$ be an open set, and $f \colon U \to \overline{\mathbb{R}}$ be a given function. For any $x \in \dom f
\cap A$ the quantity
$$
  f^{\downarrow}_A(x) = \liminf_{y \to x, \: y \in A} \frac{f(y) - f(x)}{d(y, x)}
$$
is called \textit{the rate of steepest descent} of the function $f$ with respect to the set $A$ at the point $x$, while
the quantity
$$
  f^{\uparrow}_A(x) = \limsup_{y \to x, \: y \in A} \frac{f(y) - f(x)}{d(y, x)}
$$
is called \textit{the rate of steepest ascent} of the function $f$ with respect to the set $A$ at the point $x$. If $x$
is an isolated point of the set $A$, then by definition $f^{\downarrow}_A(x) = + \infty$ and 
$f^{\uparrow}_A(x) = - \infty$. If $A = X$, then the quantity $f^{\downarrow}_X(x)$ is called \textit{the rate of
steepest descent} of $f$ at $x$, and is denoted by $f^{\downarrow}(x)$, and the similar notation is used for \textit{the
rate of steepest ascent} of $f$ at $x$. 
\end{definition}

\begin{remark}
It should be noted that that the notion of the rate of steepest descent is very similar (but not identical) to the
notion of the strong slope $|\nabla f|$ of a function $f$ defined on a metric space (cf., e.g., \cite{Ioffe,Aze}).
Namely, it is easy to see that $f^{\downarrow}(x) < 0$ iff $|\nabla f| (x) > 0$, and in this case 
$f^{\downarrow}(x) = - |\nabla f|(x)$. However, it should be noted that $|\nabla f|(x) \ge 0$ for any function $f$,
while it is easy to provide an example of a function $f$ such that $f^{\downarrow}(x) > 0$.
\end{remark}

Although there is no elaborate calculus for the rate of steepest descent/ascent, in some cases one can easily compute
(or estimate) these quantities. Below, we present several results of this type (see also \cite{Aze}). If the proof of a
statement is omitted, then it follows directly from definitions.

\begin{lemma} \label{Lemma_SimpleCases}
Let $f \colon X \to \mathbb{R}$ be a given function. The following statements hold true:
\begin{enumerate}
\item{if $X$ is a normed space, and $f$ is Fr\'echet differentiable at a point $x \in X$, then 
$f^{\downarrow}(x) = - \| f'(x) \|_{X^*}$ and $f^{\uparrow}(x) = \| f'(x) \|_{X^*}$, where $\| \cdot \|_{X^*}$ is the
norm in the topological dual space $X^*$ of $X$; furthermore, if $x$ is a limit point of the set $A$, then
$|f^{\downarrow}_A (x)| \le \| f'(x) \|_{X^*}$ and $|f^{\uparrow}_A(x)| \le \| f'(x) \|_{X^*}$;
}

\item{if $f$ is calm at a limit point $x$ of the set $A$ such that $f(x) = 0$, then for any 
$\alpha > 0$ one has $(f^{1 + \alpha})^{\downarrow}_A(x) = (f^{1 + \alpha})^{\uparrow}_A(x) = 0$; furthermore, if $X$ is
a normed space, then the function $f^{1 + \alpha}$ is Fr\'echet differentiable at the point $x$ and 
$(f^{1 + \alpha})'(x) = 0$;
}

\item{if $f(x) = d(x, x_0) + c$ in a neighbourhood of a limit point $x_0$ of the set $A$ with $c \in \mathbb{R}$, then 
$f^{\downarrow}_A(x_0) = f^{\uparrow}_A(x_0) = 1$.
}
\end{enumerate}
\end{lemma}

\begin{lemma}	\label{Lemma_SumRule}
Let $f, g \colon X \to \overline{\mathbb{R}}$ be given function. Then for any $x \in \dom f \cap \dom g$ such that the
sum $f + g$ is correctly defined in a neighbourhood of $x$, and the sums $f^{\downarrow}_A (x) + g^{\downarrow}_A (x)$
and $f^{\uparrow}_A (x) + g^{\uparrow}_A (x)$ are correctly defined one has
$$
  (f + g)^{\downarrow}_A (x) \ge f^{\downarrow}_A(x) + g^{\downarrow}_A (x), \quad
  (f + g)^{\uparrow}_A (x) \le f^{\uparrow}_A(x) + f^{\uparrow}_A (x).
$$
\end{lemma}

\begin{lemma}	\label{Lmm_SDD_SumEstimate}
Let $f, g \colon X \to \overline{\mathbb{R}}$ be given functions. Then for any $x \in \dom f \cap \dom g$ such that 
the sum $f + g$ is correctly defined in a neighbourhood of $x$, and the sum $f^{\uparrow}_A (x) + g^{\downarrow}_A (x)$
is correctly defined one has $(f + g)^{\downarrow}_A (x) \le f^{\uparrow}_A (x) + g^{\downarrow}_A (x)$.
\end{lemma}

\begin{lemma} \label{Lemma_RSD_Product}
Let a function $f \colon (a, b) \to \mathbb{R}$ be differentiable at a point $t \in (a, b)$ such that $f(t) > 0$. Then
for any function $g \colon (a, b) \to \overline{\mathbb{R}}$ such that $t \in \dom g$ one has
\begin{gather} \label{RSD_Product}
  (f \cdot g)^{\downarrow}(t) = \inf\big\{ f(t) g^{\downarrow}_{[t, b)}(t) + f'(t) g(t),
  f(t) g^{\downarrow}_{(a, t]} (t) - f'(t) g(t) \big\}, \\ 
  (f \cdot g)^{\uparrow}(t) = \sup\big\{ f(t) g^{\uparrow}_{[t, b)}(t) + f'(t) g(t),
  f(t) g^{\uparrow}_{(a, t]} (t) - f'(t) g(t) \big\}. \label{RSA_Product}
\end{gather}
In particular, the following inequalities hold true
$$
  (f \cdot g)^{\downarrow}(t) \ge - |f'(t) g(t)| + f(t) g^{\downarrow}(t), \quad
  (f \cdot g)^{\uparrow}(t) \le |f'(t) g(t)| + f(t) g^{\uparrow}(t).
$$
\end{lemma}

\begin{proof}
Let us prove the validity of~(\ref{RSD_Product}). Equality (\ref{RSA_Product}) is proved in a similar way.

It is easy to verify that
\begin{equation} \label{RSD_MinLeftRight}
  (f \cdot g)^{\downarrow}(t) = \inf\big\{ (f \cdot g)^{\downarrow}_{[t, b)}(t), 
  (f \cdot g)^{\downarrow}_{(a, t]}(t) \big\}.
\end{equation}
From the fact that the function $f$ is differentiable at the point $t$ it follows that $f$ is continuous at the point
$t$, and for any sufficiently small $\Delta t \in \mathbb{R}$ one has
\begin{multline*}
  f(t + \Delta t) g(t + \Delta t) - f(t) g(t) 
  = f(t + \Delta t) (g(t + \Delta t) - g(t)) + (f(t + \Delta t) - f(t)) g(t) = \\
  = f(t + \Delta t) (g(t + \Delta t) - g(t)) + f'(t) g(t) \Delta t + o(\Delta t) g(t),
\end{multline*}
where $o(\Delta t) / \Delta t \to 0$ as $\Delta t \to 0$. Therefore taking into account the facts that $f(t) > 0$, and
$f$ is continuous at the point $t$ one gets that
$$
  (f \cdot g)^{\downarrow}_{[t, b)}(t) = 
  \liminf_{\Delta t \to +0} \frac{f(t + \Delta t) g(t + \Delta t) - f(t) g(t)}{\Delta t}
  = f(t) g^{\downarrow}_{[t, b)}(t) + f'(t) g(t)
$$
and
$$
  (f \cdot g)^{\downarrow}_{(a, t]}(t) = 
  \liminf_{\Delta t \to -0} \frac{f(t + \Delta t) g(t + \Delta t) - f(t) g(t)}{|\Delta t|} = 
  f(t) g^{\downarrow}_{(a, t]}(t) - f'(t) g(t).
$$
Hence with the use of (\ref{RSD_MinLeftRight}) one obtains the desired result.
\end{proof}

\begin{lemma}	\label{Lemma_Superpos}
Let $f \colon [0, + \infty] \to [0, + \infty]$ be a non-decreasing function, and 
$g \colon X \to [0, + \infty]$ be a given function. Let $x \in \dom g$ be such that 
$g(x) > 0$, $- \infty < g^{\downarrow}_A (x) < 0$, the function $f$ is continuously differentiable at the point $g(x)$,
and $f'(g(x)) > 0$. Then 
\begin{equation} \label{Compos_RSD}
  \big[ f(g(\cdot)) \big]^{\downarrow}_A (x) = f'(g(x)) g^{\downarrow}_A (x) < 0.
\end{equation}
Similary, if $0 < g^{\uparrow}_A (x) < + \infty$, then
\begin{equation} \label{Compos_RSA}
  \big[ f(g(\cdot)) \big]^{\uparrow}_A (x) = f'(g(x)) g^{\uparrow}_A (x) > 0
\end{equation}
\end{lemma}

\begin{proof}
Let us prove the validity of~(\ref{Compos_RSD}). Equality~(\ref{Compos_RSA}) is proved in a similar way.

From the definition of limit inferior it follows that there exists a sequence $\{ x_n \} \subset A$ converging to $x$,
and such that
$$
  \lim_{n \to \infty} \frac{g(x_n) - g(x)}{d(x_n, x)} = g^{\downarrow}_A (x).
$$
Observe that $g(x_n) \to g(x)$ as $n \to \infty$, since $g^{\downarrow}_A(x)$ is finite. Therefore applying the
fact that the function $f$ is differentiable at the point $g(x)$ one gets that there exists a function 
$\omega \colon \mathbb{R} \to \mathbb{R}$ such that $\omega(0) = 0$, $\omega(t) \to 0$ as $t \to 0$ and
$$
  f(g(x_n)) - f(g(x)) = f'(g(x)) \big( g(x_n) - g(x) \big) + \omega\big( g(x_n) - g(x) \big) |g(x_n) - g(x)|
$$
for all $n \in \mathbb{N}$. Dividing both sides of the last equality by $d(x_n, x)$, and passing to the limit
as $n \to \infty$ one obtains that
\begin{equation} \label{CompositionRSD_UpperEstimate}
  \liminf_{n \to \infty} \frac{f(g(x_n)) - f(g(x))}{d(x_n, x)} = f'(g(x)) g^{\downarrow}_A (x) < 0.
\end{equation}
Hence $[f(g(\cdot))]^{\downarrow}_A(x) \le f'(g(x)) g^{\downarrow}_A (x) < 0$.

Let, now, $\{ x_n \} \subset A$ be a sequence converging to $x$, and such that
$$
  \lim_{n \to \infty} \frac{f(g(x_n)) - f(g(x))}{d(x_n, x)} = \big[ f(g(\cdot)) \big]^{\downarrow}_A(x).
$$
Since $[f(g(\cdot))]^{\downarrow}_A(x) < 0$, without loss of generality one can suppose that 
$f(g(x_n)) < f(g(x))$ for all $n \in \mathbb{N}$. Therefore taking into account the fact that the function $f$ is
non-decreasing one obtains that $g(x_n) < g(x)$ for any $n \in \mathbb{N}$. Note that $g(x_n) \to g(x)$ as 
$n \to \infty$, since otherwise
$$
  g^{\downarrow}_A(x) \le \liminf_{n \to \infty} \frac{g(x_n) - g(x)}{d(x_n, x)} = - \infty,
$$
which contradicts the assumption of the lemma. Consequently, by the mean value theorem for any sufficiently large 
$n \in \mathbb{N}$ there exists $\theta_n \in [g(x_n), g(x)]$ such that 
$f(g(x_n)) - f(g(x)) = f'(\theta_n) (g(x_n) - g(x))$. Applying the fact that $f$ is continuously differentiable at the
point $g(x)$ one gets that for any $\varepsilon > 0$ there exists $n_0 \in \mathbb{N}$ such that 
$0 < f'(\theta_n) < f'(g(x)) + \varepsilon$ for any $n \ge n_0$. Hence for any $n \ge n_0$ one has
$$
  f(g(x_n)) - f(g(x)) = f'(\theta_n) \big( g(x_n) - g(x) \big) \ge (f'(g(x)) + \varepsilon) \big( g(x_n) - g(x) \big).
$$
Dividing the last inequality by $d(x_n, x)$, and passing to the limit inferior as $n \to \infty$ one obtains that
\begin{multline*}
  \big[ f(g(\cdot)) \big]^{\downarrow}_A(x) = \lim_{n \to \infty} \frac{f(g(x_n)) - f(g(x))}{d(x_n, x)} \ge \\
  \ge (f'(g(x)) + \varepsilon) \liminf_{n \to \infty} \frac{g(x_n) - g(x)}{d(x_n, x)} \ge
  (f'(g(x)) + \varepsilon) g^{\downarrow}_A(x)
\end{multline*}
for any $\varepsilon > 0$. Hence and from (\ref{CompositionRSD_UpperEstimate}) one obtains the desired result.	 
\end{proof}

Let $(X, d_X)$ and $(Y, d_Y)$ be metric spaces. Hereinafter, we suppose that the Cartesian product $X \times Y$ is
endowed with the metric 
$$
  d\big( (x_1, y_1), (x_2, y_2) \big) = d_X(x_1, x_2) + d_Y(y_1, y_2).
$$
It is readily seen that the following result hold true.

\begin{lemma}	\label{Lemma_PartialRSD}
Let $A_X \subset X$ and $A_Y \subset Y$ be nonempty sets, and $A = A_X \times A_Y$. Then for any
function $f \colon X \times Y \to \overline{\mathbb{R}}$ and for any $(x, y) \in \dom f \cap A$ one has
\begin{gather*}
  f^{\downarrow}_A (x, y) \le 
  \inf\big\{ f(\cdot, y)^{\downarrow}_{A_X}(x), f(x, \cdot)^{\downarrow}_{A_Y}(y) \big\}, \\
  f^{\uparrow}_A (x) \ge 
  \sup\big\{ f(\cdot, y)^{\uparrow}_{A_X}(x), f(x, \cdot)^{\uparrow}_{A_Y}(y) \big\}.
\end{gather*}
\end{lemma}

The lemma above furnishes the upper estimate of the rate of steepest descent $f^{\downarrow}_A (x, y)$ via the
``partial'' rates of steepest descent $f(\cdot, y)^{\downarrow}_{A_X}(x)$ and $f(x, \cdot)^{\downarrow}_{A_Y}(y)$. In
the case when $X$ and $Y$ are normed spaces, and the function $f$ is Fr\'echet differentiable, one can obtain the lower
estimate of the rate of steepest descent $f^{\downarrow}_A (x, y)$ via the ``partial'' rates of steepest descent.

\begin{theorem} \label{Thrm_RSDLowerEstimViaPartialRSD}
Let $X$ and $Y$ be normed spaces, $A_X \subset X$ and $A_Y \subset Y$ be nonempty sets, and $A = A_X \times A_Y$. Let
also $U_x \subset X$ be a neighbourhood of a point $x \in A_X$, and $U_y \subset Y$ be a neighbourhood of a point 
$y \in A_Y$. Suppose that a function $f \colon U_x \times U_y \to \mathbb{R}$ is Fr\'echet differentiable at $(x, y)$,
and $g \colon U_x \to \mathbb{R}$ and $h \colon U_y \to \mathbb{R}$ are given functions. Then the function 
$F(x, y) = f(x, y) + g(x) + h(y)$ is defined in a neighbourhood of the point $(x, y)$ and
\begin{equation} \label{RSDLowerEstimViaPartialRSD}
  F^{\downarrow}_A(x, y) \ge 2 \cdot \inf\big\{ 0, f(\cdot, y)^{\downarrow}_{A_X}(x) + g^{\downarrow}_{A_X}(x),
  f(x, \cdot)^{\downarrow}_{A_Y}(y) + h^{\downarrow}_{A_Y}(y) \big\}.
\end{equation}
\end{theorem}

\begin{proof}
If $(x, y)$ is an isolated point of the set $A$, then $F^{\downarrow}_A(x, y) = + \infty$, and inequality
(\ref{RSDLowerEstimViaPartialRSD}) is valid. Therefore one can suppose that $(x, y)$ is a limit point of the set $A$.

By the definition of rate of steepest descent there exists a sequence $\{ (x_n, y_n) \} \subset A$ converging to the
point $(x, y)$, and such that
\begin{equation} \label{RSDDefInProductSpace}
  \lim_{n \to \infty} \frac{F(x_n, y_n) - F(x, y)}{d(x_n, x)_X + d(y_n, y)_Y} = F^{\downarrow}_A(x, y).
\end{equation}
Note that if $x_n = x$ for all but finitely many $n \in \mathbb{N}$, then taking into account
Lemma~\ref{Lemma_SumRule} one obtains that 
$$
  F^{\downarrow}_A(x, y) = \lim_{n \to \infty} \frac{F(x, y_n) - F(x, y)}{d(y_n, y)_Y} \ge 
  F(x, \cdot)^{\downarrow}_{A_Y}(y) \ge f(x, \cdot)^{\downarrow}_{A_Y}(y) + h^{\downarrow}_{A_Y}(y),
$$
which implies the required result. Similarly, if $y_n = y$ for all but finitely many $n \in \mathbb{N}$, then applying
Lemma~\ref{Lemma_SumRule} again one gets that
$$
  F^{\downarrow}_A(x, y) = F(\cdot, y)^{\downarrow}_{A_X}(x) \ge 
  f(\cdot, y)^{\downarrow}_{A_X}(x) + g^{\downarrow}_{A_X}(x),
$$
and inequality (\ref{RSDLowerEstimViaPartialRSD}) holds true. Thus, replacing, if necessary, the sequence 
$\{ (x_n, y_n) \}$ by its subsequence, one can suppose that $x_n \ne x$ and $y_n \ne y$ for all $n \in \mathbb{N}$.

Taking into account the fact that the function $f$ is Fr\'echet differentiable at the point $(x, y)$ one obtains that
\begin{equation} \label{DefOfFrechetDifferetiability}
  \frac{1}{\alpha_n} (f(x_n, y_n) - f(x, y)) = 
  \frac{1}{\alpha_n} \frac{\partial f}{\partial x}(x, y)[x_n - x] + 
  \frac{1}{\alpha_n} \frac{\partial f}{\partial y}(x, y)[y_n - y] + c_n,
\end{equation}
where $\alpha_n = \| x_n - x \|_X + \| y_n - y \|_Y$ and $c_n \to 0$ as $n \to \infty$. Similarly, one has
$$
  f(x_n, y) - f(x, y) = \frac{\partial f}{\partial x}(x_n - x) + o(\|x_n - x\|_X),
$$
where $o(\| x_n - x \|_X) / \| x_n - x \|_X \to 0$ as $n \to \infty$. Note that
$$
  \frac{1}{\alpha_n} |o(\| x_n - x \|_X)| = 
  \frac{\| x_n - x \|_X}{\alpha_n} \frac{|o(\| x_n - x \|_X)|}{\| x_n - x \|_X} \le
  \frac{|o(\| x_n - x \|_X)|}{\| x_n - x \|_X},
$$
which implies that $o(\| x_n - x \|_X) / \alpha_n \to 0$ as $n \to \infty$. Consequently, one gets that
$$
  \frac{1}{\alpha_n} \frac{\partial f}{\partial x}(x, y)[x_n - x] = \frac{1}{\alpha_n} ( f(x_n, y) - f(x, y) ) + c^x_n
  \quad \forall n \in \mathbb{N},
$$
where $c_n^x \to 0$ as $n \to \infty$. Replacing $x$ with $y$ one obtains that
$$
  \frac{1}{\alpha_n} \frac{\partial f}{\partial y}(x, y)[y_n - y] = \frac{1}{\alpha_n} ( f(x, y_n) - f(x, y) ) + c^y_n
  \quad \forall n \in \mathbb{N},
$$
where $c_n^y \to 0$ as $n \to \infty$. Hence and from (\ref{DefOfFrechetDifferetiability}) it follows that
$$
  \frac{1}{\alpha_n} (f(x_n, y_n) - f(x, y)) = 
  \frac{1}{\alpha_n} ( f(x_n, y) - f(x, y) ) + \frac{1}{\alpha_n} ( f(x, y_n) - f(x, y) ) + \sigma_n,
$$
where $\sigma_n \to 0$ as $n \to \infty$. Therefore
\begin{multline} \label{RSDreductionToPartialRSDInequal}
  F^{\downarrow}_A(x) = \lim_{n \to \infty} \frac{F(x_n, y_n) - F(x, y)}{d(x_n, x)_X + d(y_n, y)_Y} = 
  \lim_{n \to \infty} \Bigg( \frac{f(x_n, y) + g(x_n) - f(x, y) - g(x)}{\alpha_n} + \\
  + \frac{f(x, y_n) + h(y_n) - f(x, y) - h(y)}{\alpha_n} \Bigg) \ge
  \liminf_{n \to \infty} \frac{f(x_n, y) + g(x_n) - f(x, y) - g(x)}{\alpha_n} + \\
  + \liminf_{n \to \infty} \frac{f(x, y_n) + h(y_n) - f(x, y) - h(y)}{\alpha_n}.
\end{multline}
Replacing, if necessary, the sequence $\{ (x_n, y_n) \}$ by its subsequence one can suppose that
$$
  \liminf_{n \to \infty} \frac{f(x_n, y) + g(x_n) - f(x, y) - g(x)}{\| x_n - x \|_X} = 
  \lim_{n \to \infty} \frac{f(x_n, y) + g(x_n) - f(x, y) - g(x)}{\| x_n - x \|_X}
$$
(note that equality (\ref{RSDDefInProductSpace}) remains valid if one replaces the sequence $\{ (x_n, y_n) \}$ by its
subsequence). Replacing again, if necessary, the obtained sequence by its subsequence one can also suppose that
$$
  \liminf_{n \to \infty} \frac{f(x, y_n) + h(y_n) - f(x, y) - h(y)}{\| y_n - y \|_Y} = 
  \lim_{n \to \infty} \frac{f(x, y_n) + h(y_n) - f(x, y) - h(y)}{\| y_n - y \|_Y}.
$$
If
$$
  \lim_{n \to \infty} \frac{f(x_n, y) + g(x_n) - f(x, y) - g(x)}{\| x_n - x \|_X} > 0,
$$
then there exists $n_0 \in \mathbb{N}$ such that $f(x_n, y) + g(x_n) - f(x, y) - g(x) > 0$ for all $n \ge n_0$.
Hence $(f(x_n, y) + g(x_n) - f(x, y) - g(x)) / \alpha_n > 0$ for all $n \ge n_0$, and
\begin{equation} \label{FirstCase_PartialRSD}
  \liminf_{n \to \infty} \frac{f(x_n, y) + g(x_n) - f(x, y) - g(x)}{\alpha_n} \ge 0.
\end{equation}
If
$$
  \lim_{n \to \infty} \frac{f(x_n, y) + g(x_n) - f(x, y) - g(x)}{\| x_n - x \|_X} = 0,
$$
then for any $\varepsilon > 0$ there exists $n_0 \in \mathbb{N}$ such that for any $n \ge n_0$ one has
$$
  \frac{|f(x_n, y) + g(x_n) - f(x, y) - g(x)|}{\alpha_n} \le
  \frac{|f(x_n, y) + g(x_n) - f(x, y) - g(x)|}{\| x_n - x \|_X} < \varepsilon.
$$
Therefore
\begin{equation} 
  \lim_{n \to \infty} \frac{f(x_n, y) + g(x_n) - f(x, y) - g(x)}{\alpha_n} = 0.
\end{equation}
Finally, if
$$
  \lim_{n \to \infty} \frac{f(x_n, y) + g(x_n) - f(x, y) - g(x)}{\| x_n - x \|_X} < 0,
$$
then there exists $n_0 \in \mathbb{N}$ such that $f(x_n, y) + g(x_n) - f(x, y) - g(x) < 0$ for all $n \ge n_0$.
Consequently, for any $n \ge n_0$ one has
$$
  0 > \frac{f(x_n, y) + g(x_n) - f(x, y) - g(x)}{\alpha_n} \ge 
  \frac{f(x_n, y) + g(x_n) - f(x, y) - g(x)}{\| x_n - x \|_X},
$$
which yields
\begin{multline} \label{ThirdCase_PartialRSD}
  0 \ge \liminf_{n \to \infty} \frac{f(x_n, y) + g(x_n) - f(x, y) - g(x)}{\alpha_n} \ge \\
  \ge \lim_{n \to \infty}  \frac{f(x_n, y) + g(x_n) - f(x, y) - g(x)}{\| x_n - x \|_X} \ge 
  \big[ f(\cdot, y) + g(\cdot) \big]^{\downarrow}_{A_X}(x).
\end{multline}
Combining (\ref{FirstCase_PartialRSD})--(\ref{ThirdCase_PartialRSD}), and applying Lemma~\ref{Lemma_SumRule} one
gets that
$$
  \liminf_{n \to \infty} \frac{f(x_n, y) + g(x_n) - f(x, y) - g(x)}{\alpha_n} \ge 
  \inf\big\{ 0, f(\cdot, y)^{\downarrow}_{A_X}(x) + g^{\downarrow}_{A_X}(x) \big\}.
$$
Arguing in a similar way one can show that
$$
  \liminf_{n \to \infty} \frac{f(x, y_n) + h(y_n) - f(x, y) - h(y)}{\alpha_n} \ge
  \inf\big\{ 0, f(x, \cdot)^{\downarrow}_{A_Y}(y) + h^{\downarrow}_{A_Y}(y) \big\}.
$$
Hence and from (\ref{RSDreductionToPartialRSDInequal}) one obtains the desired result.	 
\end{proof}

The rates of steepest descent and ascent can be used to express optimality conditions. Namely, it is easy to
verify that the following result holds true.

\begin{lemma}	\label{Lemma_NessOptCond}
Let $f \colon X \to \overline{\mathbb{R}}$ be a given function, and let $x^* \in A \cap \dom f$ be a point of local
minimum $($resp.~maximum$)$ of $f$ on $A$. Then $f^{\downarrow}_A(x^*) \ge 0$
$(\text{resp. }f^{\uparrow}(A)(x^*) \le 0)$. Furthermore, if a point $x \in A \cap \dom f$ satisfies the inequality
$f^{\downarrow}_A(x) > 0$ $(\text{resp. }f^{\uparrow}(A)(x^*) < 0)$, then $x$ is a point of strict local minimum 
$($resp.~maximum$)$ of the function $f$ on the set $A$.
\end{lemma}

For any function $f \colon X \to \overline{\mathbb{R}}$ a point $x \in A \cap \dom f$ such that 
$f^{\downarrow}_A (x) \ge 0$ is called an \textit{inf-stationary} (or \textit{lower semistationary}, see
\cite{Giannessi}) point of the function $f$ with respect to the set $A$. 

We will also need the following approximate Fermat's rule in terms of rate of steepest descent, which is a simple
corollary to the Ekeland variational principle (cf.~\cite{UderzoCalm, Aze}).

\begin{lemma}[Approximate Fermat's rule]
Let $X$ be a complete metric space, $A \subseteq X$ be closed, and 
$f \colon X \to \mathbb{R} \cup \{ + \infty \}$ be proper, l.s.c. and bounded below on $A$. Let also 
$\varepsilon > 0$ and $x_{\varepsilon} \in A$ be such that
$$
  f(x_{\varepsilon}) \le \inf_{x \in A} f(x) + \varepsilon.
$$
Then for any $r > 0$ there exists $y \in A$ such that $f(y) \le f(x_{\varepsilon})$, 
$d(y, x_{\varepsilon}) \le r$ and $f^{\downarrow}_A(y) \ge - \varepsilon / r$.
\end{lemma}

\section{Parametric penalty functions}
\label{Sect_ParametricPenaltyFunctions}

In this section, we develop a general theory of exact parametric penalty functions. We obtain several sufficient
conditions for the exactness of parametric penalty functions, and study the concept of feasibility-preserving
penalty function. We also derive necessary and sufficient conditions for the zero duality gap property for a parametric
penalty function to hold true, and study some properties of minimizing sequences of a parametric penalty function.

\subsection{Exact penalty functions}
\label{Subsect_ExactPenFunc}

Let $X$ be a metric space, $M, A \subset X$ be nonempty sets, and $f \colon X \to \mathbb{R} \cup \{ + \infty \}$ be a
given function. Hereinafter, we study the following optimization problem
$$
  \min f(x) \quad \text{subject to} \quad x \in M, \quad x \in A.	\eqno{(\mathcal{P})}
$$
Denote by $\Omega = M \cap A$ the set of feasible points of this problem. We suppose that $\dom f \ne \emptyset$, and
$f$ attains a global minimum on the set $\Omega$.

Choose a function $\phi \colon X \to [0, +\infty]$ such that $\phi(x) = 0$ iff $x \in M$. The function
$g_{\lambda}(x) = f(x) + \lambda \phi(x)$ is called \textit{a penalty function} for the problem $(\mathcal{P})$,
where $\lambda \ge 0$ is a penalty parameter. The main goal of the theory of exact penalty functions is to study a
relation between the problem $(\mathcal{P})$ and the penalized problem
\begin{equation} \label{PenProb}
  \min g_{\lambda}(x) \quad \text{subject to} \quad x \in A.
\end{equation}
Note that only the constraint $x \in M$ is penalized, while the constraint $x \in A$ is taken into account explicitly.
Usually, the set $A$ corresponds to ``simple'' constraints, e.g. bound or linear constraints.

Let $x^* \in \dom f$ be a locally optimal solution of the problem $(\mathcal{P})$. Recall that the penalty function
$g_{\lambda}$ is said to be \textit{exact} at $x^*$, if there exists $\lambda^* \ge 0$ such that for any 
$\lambda \ge \lambda^*$ the point $x^*$ is a locally optimal solution of the penalized problem (\ref{PenProb}).

Observe that if the penalty function $g_{\lambda}$ is exact at a locally optimal solution $x^* \in \dom f$ of the
problem $(\mathcal{P})$, then the penalty term $\phi$ must be nonsmooth at $x^*$ in the general case. Indeed, suppose
that $X$ is a normed space, and the function $\phi$ is Fr\'echet differentiable at $x^*$. Note that $x^* \in M$ due to
the fact that $x^*$ is a locally optimal solution of the problem $(\mathcal{P})$. Therefore 
$\phi(x^*) = 0$, which implies that $x^*$ is a point of global minimum of $\phi$, since $\phi$ is nonnegative. Hence
$\phi'(x^*) = 0$. Applying Lemmas~\ref{Lemma_SimpleCases}, \ref{Lmm_SDD_SumEstimate} and \ref{Lemma_NessOptCond} one
obtains
$$
  f^{\downarrow}_A (x^*) = f^{\downarrow}_A (x^*) + \lambda \phi^{\uparrow}_A(x^*) \ge
  ( g_{\lambda} )^{\downarrow}_A(x^*) \ge 0.
$$
Thus, $f^{\downarrow}_A (x^*) \ge 0$, i.e. $x^*$ is an inf-stationary point of the problem
$$
  \min f(x) \quad \text{subject to} \quad x \in A,
$$
which is normally not the case, since otherwise the constraint $x \in M$ is somewhat redundant. As a result, one gets
that in order to construct an exact penalty function one needs to choose a nonsmooth penalty term $\phi$.
Alternatively, one can consider a \textit{parametric} penalty function.

Let $P$ be a metric space of parameters, and $p_0 \in P$ be fixed. Choose a function 
$\varphi \colon X \times P \to [0, + \infty]$ such that $\varphi(x, p) = 0$ iff $p = p_0$ and $x \in M$. Introduce
the following extended penalized problem
\begin{equation} \label{ParamPenProb}
  \min_{x, p} F_{\lambda}(x, p) \quad \text{subject to} \quad (x, p) \in A \times P,
\end{equation}
where $F_{\lambda}(x, p) = f(x) + \lambda \varphi(x, p)$ is \textit{a parametric penalty function} for the problem
$(\mathcal{P})$, and $\lambda \ge 0$ is a penalty parameter. Note that $F_{\lambda}(x, p_0) = f(x)$ for any $x \in M$,
and $F_{\lambda}(x, p) > f(x)$ otherwise. Observe also that $\Omega = \{ x \in A \mid \varphi(x, p_0) = 0 \}$.

\begin{definition}
Let $x^* \in \dom f$ be a locally optimal solution of the problem $(\mathcal{P})$. The penalty function $F_{\lambda}$ is
said to be exact at $x^*$ if there exists $\lambda_0 \ge 0$ such that the pair $(x^*, p_0)$ is a locally optimal
solution
of the problem (\ref{ParamPenProb}) with $\lambda = \lambda_0$. The greatest lower bound of all such $\lambda_0$ is
referred to as \textit{the exact penalty parameter} of the penalty function $F_{\lambda}$ at $x^*$, and is denoted by
$\lambda^*(x^*, f, \varphi)$ or simply by $\lambda^*(x^*)$, if $f$ and $\varphi$ are fixed (some authors refer to
$\lambda^*(x^*)$ as \textit{the least exact penalty parameter}, see~\cite{RubinovYang}).
\end{definition}

\begin{definition}
The penalty function $F_{\lambda}$ is called (\textit{globally}) \textit{exact} if there exists $\lambda_0 \ge 0$ such
that the function $F_{\lambda_0}$ attains a global minimum on the set $A \times P$, and a pair $(x^*, p^*)$ is a
globally optimal solution of the problem (\ref{ParamPenProb}) with $\lambda = \lambda_0$ iff $p^* = p_0$, and $x^*$ is a
globally optimal solution of the problem $(\mathcal{P})$. The greatest lower bound of all such $\lambda_0$ is referred
to as \textit{the exact penalty parameter} of the penalty function $F_{\lambda}$, and is denoted by $\lambda^*(f,
\varphi)$.
\end{definition}

Note that the definitions of local and global exact penalty parameters are meaningful and correct. Namely, if
$F_{\lambda}$ is exact at a locally optimal solution $x^*$ of the problem $(\mathcal{P})$, then for any 
$\lambda > \lambda^*(x^*)$ the pair $(x^*, p_0)$ is a locally optimal solution of the problem (\ref{ParamPenProb}).
Similarly, if $F_{\lambda}$ is exact, then for any $\lambda > \lambda^*(f, \varphi)$ the penalty function $F_{\lambda}$
attains a global minimum on the set $A \times P$, and a pair $(x^*, p^*)$ is a globally optimal solution of the problem
(\ref{ParamPenProb}) iff $p^* = p_0$, and $x^*$ is a globally optimal solution of the problem $(\mathcal{P})$. The
validity of these statements follows directly from the fact that $F_{\lambda}(x, p)$ is non-decreasing with respect to
$\lambda$ for any $(x, p) \in X \times P$, and $F_{\lambda}(x, p)$ is strictly increasing in $\lambda$ if both $f(x)$
and $\varphi(x, p)$ are finite and either $p \ne p_0$, $x \in X$ or $p = p_0$, $x \notin M$.

The following result is a direct corollary to the definition of exact parametric penalty function.

\begin{proposition} \label{Prp_ExactLimInfPenFunc}
Let $F_{\lambda}(x, p)$ be exact at a locally optimal solution $x^* \in \dom f$ of the problem $(\mathcal{P})$. Then the
penalty functions 
$$
  g_{\lambda}(x) = f(x) + \lambda \varphi(x, p_0), \quad 
  h_{\lambda}(x) = f(x) + \lambda \liminf_{p \to p_0} \varphi(x, p)
$$
are exact at $x^*$ as well. Furthermore, if $F_{\lambda}$ is (globally) exact, then so are $g_{\lambda}$ and
$h_{\lambda}$ (provided $\liminf_{p \to p_0} \varphi(x, p) > 0$ for any $x \notin M$).
\end{proposition}

Thus, in the general case, for the (local or global) exactness of the parametric penalty function 
$F_{\lambda}(x, p)$, it is necessary that the functions $x \to \varphi(x, p_0)$ and
$x \to \liminf_{p \to p_0} \varphi(x, p)$ are nonsmooth. In particular, if $X$ is a normed space, and the function
$\varphi$ has the form $\varphi(x, p) = \alpha(p) \phi(x) + \beta(p)$ for some functions 
$\alpha, \beta \colon P \to [0, + \infty]$ and $\phi \colon X \to \mathbb{R}_+$ such that $\phi$ is smooth and
$\beta(p) \to 0$ as $p \to p_0$, then for the exactness of the penalty function $F_{\lambda}$ it is necessary
that $\alpha(p) \to + \infty$ as $p \to p_0$ (cf. the parametric penalty function from
Section~\ref{Section_SingPenFunc}).

\begin{remark}
{(i) A different approach to the definition of parametric penalty functions, in which a parameter is
included into the objective function, was considered in \cite{Dolgopolik_OptLet}.
}

\noindent{(ii) Note that the parametric penalty function $F_{\lambda}$ is nothing but a standard penalty function for
the extended optimization problem of the form
$$
  \min_{(x, p)} \widehat{f}(x, p) \quad \text{subject to} \quad (x, p) \in M \times \{ p_0 \}, \quad 
  (x, p) \in A \times P,
$$
where $\widehat{f}(x, p) = f(x)$ for all $(x, p) \in X \times P$. Therefore one can apply the existing results on
exactness of non-parametric penalty functions \cite{Dolgopolik} to a penalty function for the extended problem in order
to obtain sufficient and/or necessary conditions for the exactness of the parametric penalty function $F_{\lambda}$.
However, it should be noted that in most cases this approach leads to weaker results than a direct study of the
exactness of parametric penalty functions presented below, since the parameter $p$ is treated in the same way as
the variable $x$, and the particular structure of the constraint $(x, p) \in M \times \{ p_0 \}$ is not taken into
account within this approach (cf., e.g., Theorem~\ref{Th_GlobExFiniteDim} below and its non-parametric counterpart
\cite{Dolgopolik}, Theorem~3.17).
}
\end{remark}

\subsection{Global exactness}

Observe that if the penalty function $F_{\lambda}$ is exact, then it is exact at every globally optimal solution of the
problem $(\mathcal{P})$. Our aim is to prove that under some additional assumptions the converse statement holds true.
However, at first, we describe an alternative approach to the definition of global exactness of a penalty function that
is not based on properties of global minimizers of this function.

Denote by $f^* = \inf_{x \in \Omega} f(x)$ the optimal value of the problem $(\mathcal{P})$. The proposition below
follows directly from the fact that the penalty function $F_{\lambda}(x, p)$ is strictly increasing in $\lambda$ for
any $(x, p)$ such that both $f(x)$ and $\varphi(x, p)$ are finite and either $p \ne p_0$, $x \in X$ or 
$p = p_0$, $x \notin M$.

\begin{proposition} \label{Prp_EquivDefExPen}
The penalty function $F_{\lambda}$ is exact if and only if there exists $\lambda_0 \ge 0$ such that
$$
  \inf_{(x, p) \in A \times P} F_{\lambda_0}(x, p) \ge f^*.
$$
Moreover, the greatest lower bound of all such $\lambda_0$ is equal to the exact penalty parameter 
$\lambda^*(f, \varphi)$.
\end{proposition}

\begin{corollary} \label{Crlr_ReductionToStandExPen}
The penalty function $F_{\lambda}$ is exact if and only if the penalty function 
$g_{\lambda}(x) = f(x) + \lambda \inf_{p \in P} \varphi(x, p)$ is exact. Furthermore, the exact penalty parameters of
these functions coincide.
\end{corollary}

\begin{remark}
The corollary above describes a general approach to the study of exact parametric penalty functions. Namely,
Corollary~\ref{Crlr_ReductionToStandExPen} allows one to study the standard penalty function $g_{\lambda}$ instead of
the parametric penalty function $F_{\lambda}$ in order to obtain sufficient (and maybe necessary) conditions
for $F_{\lambda}$ to be exact. One simply has to compute the penalty function $g_{\lambda}$ (or to find its lower and
upper estimates), and then apply well-developed methods of the theory of exact linear penalty function~\cite{Dolgopolik}
to the penalty function $g_{\lambda}$ (or its estimates). For the application of this approach to the study of singular
penalty functions see~\cite{Dolgopolik_OptLet2}. However, it should be noted that sometimes it might be difficult
to compute the penalty function $g_{\lambda}$ or to find its sharp estimates, which makes the described approach
inapplicable in some cases.
\end{remark}

Let us show that the infimum in Proposition~\ref{Prp_EquivDefExPen} can be taken over a significantly smaller set than
$A \times P$. We need the following definition in order to conveniently formulate this result.

\begin{definition}
Let $C \subset A \times P$ be a nonempty set. The penalty function $F_{\lambda}$ is said to be \textit{exact} on the set
$C$ if there exists $\lambda_0 \ge 0$ such that
$$
  F_{\lambda_0}(x, p) \ge f^* \qquad \forall (x, p) \in C.
$$
The greatest lower bounded of all $\lambda_0 \ge 0$ for which the inequality above holds true is denoted by
$\lambda^*(C)$ and is referred to as \textit{the exact penalty parameter} of the penalty function $F_{\lambda}$ on the
set $C$. In the case when $C = A_0 \times P$ for some $A_0 \subset A$, we simply say the the penalty function
$F_{\lambda}$ is exact on the set $A_0$, and denote $\lambda^*(A_0) := \lambda^*(A_0 \times P)$.
\end{definition}

Note that if the penalty function $F_{\lambda}$ is exact on a set $C \subset A \times P$, and there exists a globally
optimal solution $x^*$ of the problem $(\mathcal{P})$ such that $(x^*, p_0) \in C$, then for any 
$\lambda > \lambda^*(C)$ the function $F_{\lambda}$ attains a global minimum on the set $C$, and a pair 
$(\overline{x}, \overline{p})$ is a point of global minimum of $F_{\lambda}$ on $C$ iff $\overline{p} = p_0$, and
$\overline{x}$ is a globally optimal solution of the problem $(\mathcal{P})$. Note also that 
$\lambda^*(C) \le \lambda^*(f, \varphi)$ for any $C \subseteq A \times P$.

Being inspired by the ideas of prof. V.F. Demyanov \cite{Demyanov}, introduce the set 
$$
  \Omega_{\delta} = \big\{ (x, p) \in A \times P \mid \varphi(x, p) < \delta \big\} \quad \forall \delta > 0.
$$
Note that 
\begin{equation} \label{OmegaDeltaShrinking}
  \bigcap_{\delta > 0} \Omega_{\delta} = \Omega \times \{ p_0 \}
\end{equation}
and for any $\delta_1 > \delta_2$ one has $\Omega_{\delta_2} \subseteq \Omega_{\delta_1}$, i.e. the set
$\Omega_{\delta}$ shrinks into the set $\Omega \times \{ p_0 \}$ as $\delta \to +0$.

\begin{lemma} \label{Lmm_Reduction}
Let there exist $\lambda_0 \ge 0$ such that the function $F_{\lambda_0}$ is bounded below on $A \times P$.
Then for any $\delta > 0$ one has
$$
  F_{\lambda}(x, p) \ge f^* \qquad \forall (x, p) \notin \Omega_{\delta}
  \qquad \forall \lambda \ge \lambda_0 + \frac{f^* - c}{\delta},
$$
where $c = \inf\{ F_{\lambda_0}(x, p) \mid (x, p) \in A \times P \}$, i.e. the penalty function $F_{\lambda}$ is exact
on the set $(A \times P) \setminus \Omega_{\delta}$ and 
$\lambda^*((A \times P) \setminus \Omega_{\delta}) \le \lambda_0 + (f^* - c) / \delta$.
\end{lemma}

\begin{proof}
Fix $\delta > 0$. Let $(x, p) \notin \Omega_{\delta}$, i.e. $\varphi(x, p) \ge \delta$. Then
$$
  F_{\lambda}(x, p) = f(x) + \lambda \varphi(x, p) = f(x) + \lambda_0 \varphi(x, p) + 
  (\lambda - \lambda_0) \varphi(x, p) \ge F_{\lambda_0}(x, p) + (\lambda - \lambda_0) \delta.
$$
Therefore for any $\lambda \ge \overline{\lambda}$, where
$$
  \overline{\lambda} = \lambda_0 + \frac{f^* - c}{\delta}, \qquad 
  c = \inf_{(x, p) \in A \times P} F_{\lambda_0}(x, p),
$$
one has $F_{\lambda}(x, p) \ge c + (\lambda - \lambda_0) \delta \ge f^*$ for all $(x, p) \notin \Omega_{\delta}$, that
completes the proof.	 
\end{proof}

\begin{proposition} \label{Prp_Reduction}
The penalty function $F_{\lambda}$ is exact if and only if $F_{\lambda}$ is bounded below on $A \times P$ for some
$\lambda \ge 0$, and there exist $\delta > 0$ such that the penalty function $F_{\lambda}$ is exact on
$\Omega_{\delta}$. Furthermore, one has
$$
  \lambda^*(f, \varphi) \le \overline{\lambda} := 
  \max\left\{ \lambda^*(\Omega_{\delta}), \lambda_0 + \frac{f^* - c}{\delta} \right\}, \quad 
  c = \inf_{(x, p) \in A \times P} F_{\lambda_0}(x, p)
$$
for any $\lambda_0 \ge 0$ such that $F_{\lambda_0}$ is bounded below on $A \times P$.
\end{proposition}

\begin{proof}
If the penalty function $F_{\lambda}$ is exact, then applying Proposition~\ref{Prp_EquivDefExPen} one gets that
the function $F_{\lambda}$ is exact on the set $\Omega_{\delta}$ for any $\delta > 0$, and is bounded below on 
$A \times P$ for any $\lambda \ge \lambda^*(f, \varphi)$.

Suppose now that $F_{\lambda}$ is bounded below on $A \times P$, and there exists $\delta > 0$ such that
the penalty function $F_{\lambda}$ is exact on $\Omega_{\delta}$. Applying~Lemma~\ref{Lmm_Reduction} and the fact that 
$F_{\lambda}$ is exact on $\Omega_{\delta}$ one obtains that $F_{\lambda}(x, p) \ge f^*$ for all $(x, p) \in A \times P$
and $\lambda \ge \overline{\lambda}$. Consequently, the penalty function $F_{\lambda}$ is exact by
Proposition~\ref{Prp_EquivDefExPen}.	 
\end{proof}

\begin{remark} \label{Rmrk_BarrierTermPenFunc}
One can easily verify that if the penalty function $F_{\lambda}$ is exact on $\Omega_{\delta}$ for some $\delta > 0$,
then the penalty function $\Psi_{\lambda}(x) = f(x) + \lambda \psi_{\delta}(\varphi(x, p))$ is exact and
$\lambda^*(f, \psi_{\delta} \circ \varphi) \le \delta \cdot \lambda^*(\Omega_{\delta})$, where 
$\psi(t) = t / ( \delta - t )$, if $0 \le t < \delta$, and $\psi(t) = + \infty$, otherwise. Conversely, if the penalty
function $\Psi_{\lambda}$ is exact, then for any $\theta \in (0, \delta)$ the penalty function $F_{\lambda}$ is exact on
$\Omega_{\theta}$ and $\lambda^*(\Omega_{\theta}) \le \lambda^*(f, \psi_{\delta} \circ \varphi) / (\delta - \theta)$.
\end{remark}

Recall that $\Omega$ is the set of feasible points of the problem $(\mathcal{P})$. As it was mentioned above, the
set $\Omega_{\delta}$ shrinks into the set $\Omega \times \{ p_0 \}$ as $\delta \to +0$
(see~(\ref{OmegaDeltaShrinking})). Furthermore, it is easy to see that if the function $\varphi$ is lower semicontinuous
on $A \times P$, and the set $A$ is closed, then
$$
  \limsup_{\delta \to +0} \Omega_{\delta} = \Omega \times \{ p_0 \},
$$
where $\limsup$ is the outer limit. Thus, the set $\Omega_{\delta}$ can be considered as an outer approximation of the
set of feasible solutions of the problem $(\mathcal{P})$. Consequently, the proposition above can be imprecisely
formulated as follows: the penalty function $F_{\lambda}$ is exact if and only if it is bounded below for
sufficiently large $\lambda \ge 0$, and $F_{\lambda}$ is exact on the outer approximation $\Omega_{\delta}$ of the set
$\Omega$ with arbitrarily small $\delta > 0$. However, note that the set $\Omega_{\delta}$ is defined via the penalty
term $\varphi$.

Utilizing the previous proposition one can obtain a general sufficient condition for the global exactness of 
the parametric penalty function $F_{\lambda}$.

\begin{theorem}	\label{Th_GlobalExGeneralCase}
Let $X$ and $P$ be complete metric spaces, $A$ be closed, $f$ be l.s.c. on $A$, and $\varphi$ be l.s.c. on 
$A \times P$. Suppose also that there exist $\delta > 0$ and $\lambda_0 \ge 0$ such that
\begin{enumerate}
\item{the function $f$ is Lipschitz continuous on the projection of the set 
$$
  C(\delta, \lambda_0) = \big\{ (x, p) \in \Omega_{\delta} \mid F_{\lambda_0}(x, p) < f^* \big\}
$$
onto the space $X$, and upper semicontinuous (u.s.c.) at every point of this projection;
} 

\item{there exists $a > 0$ such that $\varphi^{\downarrow}_{A \times P}(x, p) \le - a$ for any 
$(x, p) \in C(\delta, \lambda_0) \cap \dom \varphi$.
\label{Assumpt_MetricRegDiffCond}
} 
\end{enumerate}
Then the penalty function $F_{\lambda}$ is exact if and only if it is bounded below on $A \times P$ for
some $\lambda \ge 0$.
\end{theorem}

\begin{proof}
If $F_{\lambda}$ is exact, then, obviously, it is bounded below on $A \times P$ for any 
$\lambda \ge \lambda^*(f, \varphi)$. Let us prove the converse statement. Note that without loss of generality one can
suppose that $F_{\lambda_0}$ is bounded below on $A \times P$.

Arguing by reductio ad absurdum, suppose that the function $F_{\lambda}$ is not exact. Then for any $\lambda \ge 0$
there exists $(x_{\lambda}, p_{\lambda}) \in A \times P$ such that $F_{\lambda}(x_{\lambda}, p_{\lambda}) < f^*$. 
The function $F_{\lambda}$ is l.s.c. due to the lower semicontinuity of the functions $f$ and $\varphi$. Applying
approximate Fermat's rule to the restriction of $F_{\lambda}$ to the set $A \times P$ one obtains that for any 
$\lambda \ge 0$ there exists $(y_{\lambda}, q_{\lambda}) \in A \times P$ such that
\begin{equation} \label{ApproxFermatRule}
  F_{\lambda}(y_{\lambda}, q_{\lambda}) \le F_{\lambda}(x_{\lambda}, p_{\lambda}) < f^*, \quad
  \big( F_{\lambda} \big)^{\downarrow}_{A \times P} (y_{\lambda}, q_{\lambda}) \ge -1.
\end{equation}
From Lemma~\ref{Lmm_Reduction} it follows that
$$
  F_{\lambda}(x, p) \ge f^* \quad \forall (x, p) \notin \Omega_{\delta} \quad \forall
  \lambda \ge \overline{\lambda} := \lambda_0 + \frac{f^* - \inf_{x \in A \times P} F_{\lambda_0}(x, p)}{\delta},
$$
which yields $\varphi(y_{\lambda}, q_{\lambda}) < \delta$ for any $\lambda \ge \overline{\lambda}$. Hence
$(y_{\lambda}, q_{\lambda}) \in C(\delta, \lambda_0)$ for all $\lambda \ge \overline{\lambda}$.

Due to assumption~\ref{Assumpt_MetricRegDiffCond} one has
$\varphi^{\downarrow}_{A \times P} (y_{\lambda}, q_{\lambda}) \le - a$ for any $\lambda \ge \overline{\lambda}$.
Therefore by the definition of rate of steepest descent for all $\lambda \ge \overline{\lambda}$ 
there exists a sequence $\{ (y_{\lambda}^{(n)}, q_{\lambda}^{(n)}) \} \subset A \times P$ converging to 
$(y_{\lambda}, q_{\lambda})$ such that
\begin{equation} \label{RateOfStDesc_Def}
  \frac{\varphi( y_{\lambda}^{(n)}, q_{\lambda}^{(n)} ) - \varphi(y_{\lambda}, q_{\lambda})}{
  d( y_{\lambda}^{(n)}, y_{\lambda} )_X + d( q_{\lambda}^{(n)}, q_{\lambda} )_P} \le -
  \frac{a}{2} \quad \forall n \in \mathbb{N}.
\end{equation}
Hence, in particular, $\varphi(y_{\lambda}^{(n)}, q_{\lambda}^{(n)}) < \varphi(y_{\lambda}, q_{\lambda}) < \delta$.
Taking into account the facts that $F_{\lambda}(y_{\lambda}, q_{\lambda}) < f^*$, and $f$ is u.s.c. at $y_{\lambda}$,
by virtue of the fact that $(y_{\lambda}, q_{\lambda}) \in C(\delta, \lambda_0)$, one obtains that
$F_{\lambda}(y_{\lambda}^{(n)}, q_{\lambda}^{(n)}) < f^*$ for any sufficiently large $n$. Therefore without loss of
generality one can suppose that $\{ (y_{\lambda}^{(n)}, q_{\lambda}^{(n)}) \} \subset C(\delta, \lambda_0)$ for any
$\lambda \ge \overline{\lambda}$. Consequently, applying (\ref{RateOfStDesc_Def}), and the fact that $f$ is Lipschitz
continuous on the projection of $C(\delta, \lambda_0)$ onto $X$ one gets that there exists $L > 0$ such that
\begin{multline*}
  F_{\lambda}(y_{\lambda}^{(n)}, q_{\lambda}^{(n)}) - F_{\lambda}(y_{\lambda}, q_{\lambda})
  \le f(y_{\lambda}^{(n)}) - f(y_{\lambda}) + \lambda 
  \big( \varphi(y_{\lambda}^{(n)}, q_{\lambda}^{(n)}) - \varphi(y_{\lambda}, q_{\lambda}) \big) \le \\
  \le \left( L - \frac{a}{2} \lambda \right) 
  \big( d( y_{\lambda}^{(n)}, y_{\lambda} )_X + d( q_{\lambda}^{(n)}, q_{\lambda} )_P \big)
\end{multline*}
for any $n \in \mathbb{N}$ and $\lambda \ge \lambda^*$. Hence dividing by 
$d( y_{\lambda}^{(n)}, y_{\lambda} )_X + d( q_{\lambda}^{(n)}, q_{\lambda} )_P$ and  passing to the limit inferior as 
$n \to \infty$ one obtains that
$$
  \big( F_{\lambda} \big)^{\downarrow}_{A \times P} (y_{\lambda}, q_{\lambda}) \le L - \frac{a}{2} \lambda 
  \quad \forall \lambda \ge \overline{\lambda},
$$
which contradicts (\ref{ApproxFermatRule}).	 
\end{proof}

\begin{remark}
{(i)~It is not difficult to verify that the theorem above holds true if the set $A \cap \dom f$, considered as a metric
subspace of $X$, is complete, while the space $X$ itself can be incomplete.
}

{\noindent(ii)~Note that in the theorem above it is sufficient to suppose that the function $f$ is Lipschitz continuous
on the set $\{ x \in X \mid f(x) < f^* \}$. Furthermore, if $X$ is a normed space, and for some $\lambda \ge 0$ one has
$\inf_{p \in P} F_{\lambda}(x, p) \to + \infty$ as $\| x \| \to \infty$ (in particular, one can suppose that 
the function $f$ is coercive), then it is sufficient to suppose that $f$ is Lipschitz continuous on any bounded subset
of the space $X$.
}

{\noindent(iii)~It should be noted that the fact that a function $g \colon X \to \mathbb{R}$ is Lipschitz continuous on
a set $C \subset X$ does not imply that $g$ is u.s.c. on $C$. For example, let $X = \mathbb{R}$, $C = [0, 1]$, 
$g(x) = 1$ for any $x \in (- \infty, 0) \cup (1, + \infty)$ and $g(x) = 0$ for all $x \in [0, 1]$. Then $g$ is
Lipschitz continuous on the set $C$ with the Lipschitz constant $L = 0$, since for any $x, y \in C$ one has 
$|g(x) - g(y)| = 0$. However, $g$ is not u.s.c. on the set $C$, since it is not u.s.c. at the points $x = 1 \in C$ and
$x = 0 \in C$. Therefore, in the previous theorem, the assumption that the function $f$ is u.s.c. at every point of the
projection of $C(\delta, \lambda_0)$ onto $X$ is not redundant.
}

{\noindent(iv)~Note, finally, that Theorem~\ref{Th_GlobalExGeneralCase} was inspired by Theorem~3.4.1 from
\cite{Demyanov}. Furthermore, Theorem~\ref{Th_GlobalExGeneralCase} sharpens Theorem~3.4.1 from \cite{Demyanov} (since
we do not assume that $F_{\lambda}(x, p)$ attains a global minimum for any sufficiently large $\lambda \ge 0$) and
extends it to the parametric case.
}
\end{remark}

Clearly, the most restrictive assumption of the previous theorem is assumption~\ref{Assumpt_MetricRegDiffCond}.
Typically, in order to check this assumption one needs to verify that a constraint qualification condition in the
problem $(\mathcal{P})$ holds at every point of the projection of $C(\delta, \lambda_0)$ onto $X$ that might have a very
complicated structure. However, let us show that under natural assumptions in the case when $X$ is a finite dimensional
normed space, the global exactness of the penalty function $F_{\lambda}$ is completely defined by the behaviour of this
function near globally optimal solutions of the problem $(\mathcal{P})$. We call this result \textit{the localization
principle}, since it allows one to reduce the study of the \textit{global} exactness of a penalty function to
a \textit{local} analysis of the behaviour of this function near globally optimal solutions of the original problem.

\begin{theorem}[Localization Principle] \label{Th_GlobExFiniteDim}
Let $X$ be a finite dimensional normed space, $A$ be a closed set, $f$ be l.s.c. on $\Omega$, and $\varphi$ be l.s.c. on
$A \times \{ p_0 \}$. Suppose that there exists a non-decreasing function 
$\omega \colon \mathbb{R}_+ \to \mathbb{R}_+$ such that $\omega(t) = 0$ iff $t = 0$, and
\begin{equation} \label{PenTermGen_GrowthInP}
  \varphi(x, p) \ge \omega(d(p, p_0)) \quad \forall (x, p) \in A \times P.
\end{equation}
Let also one of the following assumptions be valid:
\begin{enumerate}
\item{there exists $\delta > 0$ such that the projection of the set $\Omega_{\delta}$ onto the space $X$ is bounded;
}

\item{there exists $\mu > 0$ such that the set $\{ x \in A \mid \inf_{p \in P} F_{\mu}(x, p) < f^* \}$ is
bounded (in particular, one can suppose that $A$ is compact or $f$ is coercive).
\label{Assump_BoundLevelSet}}
\end{enumerate}
Then the penalty function $F_{\lambda}$ is exact if and only if $F_{\lambda}$ is exact at every globally optimal
solution
of the problem $(\mathcal{P})$, and there exists $\lambda_0 \ge 0$ such that $F_{\lambda_0}$ is bounded below on 
$A \times P$.
\end{theorem}

\begin{proof}
If $F_{\lambda}$ is exact, then it is obviously exact at every globally optimal solution of the problem $(\mathcal{P})$,
and for any $\lambda \ge \lambda^*(f, \varphi)$ the function $F_{\lambda}$ is bounded below. Let us prove the
converse statement.

Observe that if $F_{\lambda_0}$ is bounded below on $A \times P$ for some $\lambda_0 \ge 0$, then by virtue of 
Lemma~\ref{Lmm_Reduction} one gets that for any $\delta > 0$ there exists $\mu \ge \lambda_0$ such that
$$
  \big\{ (x, p) \in A \times P \mid F_{\mu}(x, p) < f^* \big\} \subset \Omega_{\delta}.
$$
Therefore without loss of generality one can suppose that assumption~\ref{Assump_BoundLevelSet} is valid.

Our aim is to show that $F_{\lambda}$ is exact on $\Omega_{\delta}$ for sufficiently small $\delta > 0$. Then with 
the use of Proposition~\ref{Prp_Reduction} one concludes that $F_{\lambda}$ is exact.

From the fact that the set $\{ x \in A \mid \inf_{p \in P} F_{\mu}(x, p) < f^* \}$ is bounded it follows that
$\lambda \ge \mu$ one has
\begin{equation} \label{ExactOutsideBall}
  F_{\lambda}(x, p) \ge f^* \quad \forall p \in P \quad \forall x \in A \colon \| x \| > R
\end{equation}
for sufficiently large $R > 0$.

Denote by $\Omega^* \subset \Omega$ the set of all globally optimal solutions of the problem $(\mathcal{P})$ that belong
to the ball $B(0, R) = \{ x \in X \mid \| x \| \le R \}$. Observe that the function $\varphi(\cdot, p_0)$ is l.s.c. on
$A$ by virtue of the fact that $\varphi$ is l.s.c. on $A \times \{ p_0 \}$. Hence taking into account the facts that
$\varphi$ is nonnegative, $\Omega = \{ x \in A \mid \varphi(x, p_0) = 0 \}$, and $A$ is closed, one obtains that
$\Omega$ is closed as well. Therefore by virtue of the lower semicontinuity of the function $f$ on $\Omega$ one gets
that $\Omega^*$ is closed and, consequently, compact due to the fact that $X$ is a finite dimensional normed space.

By the assumption of the theorem the penalty function $F_{\lambda}$ is exact at every globally optimal solution of 
the problem $(\mathcal{P})$. Therefore for any $x^* \in \Omega^*$ there exist $\lambda(x^*) \ge 0$ and $r(x^*) > 0$ such
that for any $\lambda \ge \lambda(x^*)$ one has
\begin{equation} \label{ExactAtOptSol}
  F_{\lambda}(x, p) \ge f(x^*) = f^* \quad 
  \forall (x, p) \in \Big( U\big( x^*, r(x^*) \big) \cap A \Big) \times U\big( p_0, r(x^*) \big),
\end{equation}
where $U( p_0, r(x^*) ) = \{ p \in P \mid d(p_0, p) < r(x^*) \}$. Applying the compactness of the set
$\Omega^*$ one obtains that there exist $x_1^*, \ldots, x_m^* \in \Omega^*$ such that
\begin{equation} \label{FiniteOpenCover}
  \Omega^* \subset \bigcup_{k = 1}^m U\left( x_k^*, \frac{r(x_k^*)}{2} \right).
\end{equation}
Denote
$$
  \lambda_0 = \max_{k \in \{ 1, \ldots, m \}} \lambda(x_k^*), \quad
  r_0 = \min_{k \in \{ 1, \ldots, m \}} \frac{r(x_k^*)}{2}, \quad 
  U = \bigcup_{x^* \in \Omega^*} \big( U(x^*, r_0) \cap A \big).
$$
Observe that $\Omega^* \subset U$, and the set $U$ is open in $A$ (hereafter, we suppose that the set $A$ is endowed
with
the induced metric). 

Let $x \in U$ be arbitrary. By definition there exists $x^* \in \Omega^*$ such that $x \in U(x^*, r_0)$. Applying
\eqref{FiniteOpenCover} one obtains that there exists $k \in \{ 1, \ldots, m \}$ such that
$x^* \in U(x_k^*, r(x_k^*)/2)$. Hence and from the definition of $r_0$ it follows that $x \in U(x_k^*, r(x_k^*))$,
which with the use of \eqref{ExactAtOptSol} and the definition of $\lambda_0$ implies that
$F_{\lambda}(x, p) \ge f^*$ for all $p \in U(p_0, r_0)$ and $\lambda \ge \lambda_0$. Thus, one has
\begin{equation} \label{ExactNearOptSol}
  F_{\lambda}(x, p) \ge f^* \quad \forall (x, p) \in U \times U( p_0, r_0 ) \quad
  \forall \lambda \ge \lambda_0.
\end{equation}
Denote $C = (\Omega \setminus U) \cap B(0, R)$. Clearly, for any $x \in C$ one has $f(x) > f^*$. Applying the lower
semicontinuity of the function $f$ on $\Omega$ one gets that for any $x \in C$ there exists $\tau(x) > 0$ such that 
$f(y) > f^*$ for any $y \in U(x, \tau(x))$. Denote
$$
  V = \bigcup_{x \in C} \big( U(x, \tau(x)) \cap A \big).
$$
Note that $V$ is open in $A$, and
\begin{equation} \label{ExactNotInOptSol}
  F_{\lambda}(x, p) \ge f(x) > f^* \quad \forall (x, p) \in V \times P \quad \forall \lambda \ge 0,
\end{equation}
i.e. $F_{\lambda}$ is exact on the set $V$ and $\lambda^*(V) = 0$.

From the facts that the sets $U$ and $V$ are open in $A$, and $A$ is closed it follows that the set 
$K = ( B(0, R) \cap A ) \setminus (U \cup V)$ is closed in $X$, and, consequently, compact. Furthermore, by the
definitions of $U$ and $V$ one has $\Omega \cap B(0, R) \subset U \cup V$, i.e. the sets $K$ and $\Omega$ are disjoint.
Therefore for any $x \in K$ one has $\varphi(x, p_0) > 0$. By the assumption of the theorem $\varphi$ is l.s.c. on 
$A \times \{ p_0 \}$. Hence for any $x \in K$ there exists $s(x) > 0$ such that
$$
  \varphi(y, p) > \frac{\varphi(x, p_0)}{2} \quad \forall (x, p) \in U(x, s(x)) \times U(p_0, s(x)).
$$
Applying the compactness of the set $K$ one gets that there exist $x_1, \ldots, x_l \in K$ such that
$$
  K \subset \bigcup_{k = 1}^l U(x_k, s(x_k)),
$$
which yields $\varphi(x, p) \ge \delta_0$ for all $(x, p) \in K \times U( p_0, s_0 )$, where
$$
  \delta_0 = \min_{k \in \{ 1, \ldots, l \}} \frac{\varphi(x_k, p_0)}{2}, \qquad
  s_0 = \min_{k \in \{ 1, \ldots, l \}} s(x_k).
$$
On the other hand, if $p \in P$ is such that $d(p, p_0) \ge s_0$, then $\varphi(x, p) \ge \omega(s_0) > 0$.
Denote $\delta = \min\{ \delta_0, \omega(\min\{ r_0, s_0 \}) \} > 0$. Then $\varphi(x, p) \ge \delta$ for 
all $(x, p) \in K \times P$.

Observe that
$$
  \Omega_{\delta} \cap ( B(0, R) \times P ) \subseteq (U \cup V) \times U(p_0, r_0).
$$
Indeed, if $(x, p) \in \Omega_{\delta} \cap ( B(0, R) \times P )$, then 
$x \in (A \cap B(0, R)) \setminus K = U \cup V$, and $d(p, p_0) < r_0$, since otherwise 
$\varphi(x, p) \ge \omega(r_0) \ge \delta$, which is impossible. Therefore taking into account (\ref{ExactOutsideBall}),
(\ref{ExactNearOptSol}) and (\ref{ExactNotInOptSol}) one obtains that
$$
  F_{\lambda}(x, p) \ge f^* \quad \forall (x, p) \in \Omega_{\delta}
  \quad \forall \lambda \ge \max\{ \mu, \lambda_0 \}.
$$
Thus, the penalty function $F_{\lambda}$ is exact on the set $\Omega_{\delta}$.	 
\end{proof}

\begin{remark}
The theorem above extends Theorem~3.17 from \cite{Dolgopolik} to the case of parametric penalty functions. It
should be noted that although Theorem~3.17 from \cite{Dolgopolik} is correct, its proof is valid only in the case 
$A = X$ due to a small mistake. However, the theorem above provides a correct proof of Theorem~3.17 from
\cite{Dolgopolik}, since it contains Theorem~3.17 as a particular case.
\end{remark}

Note that by Proposition~\ref{Prp_EquivDefExPen} the boundedness of the set 
$\{ x \in A \mid \inf_{p \in P} F_{\mu}(x, p) < f^* \}$ for some $\mu \ge 0$ is also \textit{necessary} for the penalty
function $F_{\lambda}$ to be exact. Thus, the previous theorem can be reformulated as follows.

\begin{theorem}[Localization Principle]
Let $X$ be a finite dimensional normed space, $A$ be a closed set, $f$ be l.s.c. on $\Omega$, $\varphi$ be l.s.c. on
$A \times \{ p_0 \}$, and let inequality \eqref{PenTermGen_GrowthInP} hold true. Then the penalty function
$F_{\lambda}$ is globally exact if and only if $F_{\lambda}$ is exact at every globally optimal solution of the problem
$(\mathcal{P})$, and there exists $\mu \ge 0$ such that the set 
$\{ x \in A \mid \inf_{p \in P} F_{\mu}(x, p) < f^* \}$ is either bounded or empty, and the function $F_{\mu}$ is
bounded below on $A \times P$.
\end{theorem}

Arguing in the same way as in the proof of Theorem~\ref{Th_GlobExFiniteDim} one can obtain a complete characterization
of the exactness of the penalty function $F_{\lambda}$ on bounded subsets of a finite dimensional space.

\begin{theorem}[Localization Principle]
Let $X$ be a finite dimensional normed space, $A$ be a closed set, $f$ be l.s.c. on $\Omega$,
$\varphi$ be l.s.c. on $A \times \{ p_0 \}$, and the penalty function $F_{\lambda_0}$ be bounded below on
$A \times P$ for some $\lambda_0 \ge 0$. Suppose that there exists a non-decreasing function 
$\omega \colon \mathbb{R}_+ \to \mathbb{R}_+$ such that $\omega(t) = 0$ iff $t = 0$, and
$\varphi(x, p) \ge \omega(d(p, p_0))$ for all $(x, p) \in A \times P$. Then for the penalty function $F_{\lambda}$ to be
exact on any bounded subset of the set $A$ it is necessary and sufficient that $F_{\lambda}$ is exact at every globally
optimal solution of the problem $(\mathcal{P})$.
\end{theorem}

In the case, when $X$ is a finite dimensional normed space, and $P$ is a closed subset of a finite dimensional normed
space, one can provide a different characterization of the exactness of a parametric penalty function (cf.~Theorem~3.10
in \cite{Dolgopolik}).

\begin{definition}
Let $X$ be a normed space, and $P$ be a subset of a normed space $Y$. The parametric penalty function
$F_{\lambda}$ is said to be \textit{non-degenerate} if there exist $\lambda_0 \ge 0$ and $R > 0$ such that for any
$\lambda \ge \lambda_0$ the penalty function $F_{\lambda}$ attains a global minimum on the set $A \times P$ and there
exists $(x_{\lambda}, p_{\lambda}) \in \argmin_{(x, p) \in A \times P} F_{\lambda}(x, p)$ such that
$\| x_{\lambda} \|_X \le R$ and $\| p_{\lambda} \|_Y \le R$.
\end{definition}

Roughly speaking, the non-degeneracy of the penalty function $F_{\lambda}$ means that $F_{\lambda}$ attains a global
minimum on $A \times P$ for any sufficiently large $\lambda \ge 0$, and the norm of global minimizers of this function
on $A \times P$ cannot increase unboundedly as $\lambda \to +\infty$. In other words, the non-degeneracy condition does
not allow points of global minimum of the penalty function $F_{\lambda}$ to escape to infinity as $\lambda \to +\infty$.

\begin{theorem}[Localization Principle] \label{Thrm_ExactnessNessSuffCond_NonDegeneracy}
Let $X$ be a finite dimensional normed space, $P$ be a closed subset of a finite dimensional normed space $Y$, and $A$
be closed. Suppose also that $f$ is l.s.c. on $A$, and $\varphi$ is l.s.c. on $A \times P$. Then the parametric penalty
function $F_{\lambda}$ is exact if and only if $F_{\lambda}$ is non-degenerate, and exact at every globally optimal
solution of the problem $(\mathcal{P})$.
\end{theorem}

\begin{proof}
Let $F_{\lambda}$ be exact, and let $x^*$ be a globally optimal solution of the problem $(\mathcal{P})$. Then for any 
$\lambda > \lambda^*(f, \varphi)$ the pair $(x^*, p_0)$ is a point of global minimum of $F_{\lambda}$ on the set 
$A \times P$. Therefore $F_{\lambda}$ is exact at $x^*$, which implies that $F_{\lambda}$ is exact at every globally
optimal solution of the problem $(\mathcal{P})$. Moreover, $F_{\lambda}$ is non-degenerate with 
$\lambda_0 = \lambda^*(f, \varphi)$ and $R = \max\{ \|x^*\|_X, \|p_0\|_Y \}$.

Suppose, now, that $F_{\lambda}$ is non-degenerate and exact at every globally optimal solution of the problem
$(\mathcal{P})$. Then there exist $\lambda_0 \ge 0$ and $R > 0$ such that for any $\lambda \ge \lambda_0$ there exists 
$(x_{\lambda}, p_{\lambda}) \in \argmin_{(x, p) \in A \times P} F_{\lambda}(x, p)$ for which 
$\| x_{\lambda} \|_X \le R$ and $\| p_{\lambda} \|_Y \le R$.

Choose an increasing unbounded sequence $\{ \lambda_n \} \subset [\lambda_0, + \infty)$. The corresponding sequence 
$\{ (x_{\lambda_n}, p_{\lambda_n}) \} \subset A \times P$ is bounded in the finite dimensional normed space 
$X \times Y$. Therefore, without loss of generality, one can suppose that this sequence converges to some 
$(x^*, p^*) \in X \times Y$. Recall that the sets $A$ and $P$ are closed. Therefore $(x^*, p^*) \in A \times P$. Let us
show that $x^*$ is a globally optimal solution of the problem $(\mathcal{P})$ and $p^* = p_0$.

Indeed, from Theorem~\ref{Thrm_MinimizingSequences} below it follows that $\varphi(x_{\lambda_n}, p_{\lambda_n}) \to 0$
as $n \to \infty$. Hence taking into account the fact that the function $\varphi$ is l.s.c. on $A \times P$ one obtains
that $x^* \in \Omega$ and $p = p_0$. 

From the facts that $(x_{\lambda_n}, p_{\lambda_n})$ is a point of global minimum of the penalty function $F_{\lambda}$
on the set $A \times P$, $F_{\lambda}(x, p_0) = f(x)$ for all $x \in \Omega$, and the function $\varphi$ is nonnegative
it follows that $f(x_{\lambda_n}) \le f^*$. Consequently, applying the lower semicontinuity of the function $f$ on the
set $A$ and the fact that $x^* \in \Omega$ one gets that $f(x^*) = f^*$. Thus, $x^*$ is a globally optimal solution of
the problem $(\mathcal{P})$. Therefore, $F_{\lambda}$ is exact at $x^*$. 

Fix an arbitrary $\mu > \lambda^*(x^*)$. Then there exists $r > 0$ such that
$$
  F_{\mu}(x, p) \ge F_{\mu}(x^*, p_0) = f^* \quad 
  \forall (x, p) \in \Big( U(x^*, r) \cap A \Big) \times U(p_0, r).
$$
Hence and from the fact that the function $F_{\lambda}$ is non-decreasing in $\lambda$ it follows that for any
$\lambda \ge \mu$ one has
\begin{equation} \label{NonDeg_ExactNearGlobMin}
  F_{\lambda}(x, p) \ge F_{\lambda}(x^*, p_0) = f^* \quad 
  \forall (x, p) \in \Big( U(x^*, r) \cap A \Big) \times U(p_0, r).
\end{equation}
Applying the facts that $\{ \lambda_n \}$ is an increasing unbounded sequence, and 
the sequence $\{ (x_{\lambda_n}, p_{\lambda_n}) \}$ converges to the point $(x^*, p_0)$ one obtains that there exists
$n_0 \in \mathbb{N}$ such that for any $n \ge n_0$ one has $\lambda_n \ge \mu$ and 
$(x_{\lambda_n}, p_{\lambda_n}) \in U(x^*, r) \times U(p_0, r)$. Hence and from (\ref{NonDeg_ExactNearGlobMin}) one
gets that
$$
  F_{\lambda_n}(x_{\lambda_n}, p_{\lambda_n}) \ge f^* \quad \forall n \ge n_0.
$$
Recall that $(x_{\lambda_n}, p_{\lambda_n})$ is a point of global minimum of the function $F_{\lambda_n}$ on the set 
$A \times P$. Therefore for any $n \ge n_0$ one has $F_{\lambda_n}(x, p) \ge f^*$ for all $(x, p) \in A \times P$,
which, with the use of Proposition~\ref{Prp_EquivDefExPen}, implies that the penalty function function $F_{\lambda}$ is
exact.	 
\end{proof}

\subsection{Feasibility-preserving parametric penalty functions}

In the previous subsection, we developed a direct approach to the study of exact parametric penalty functions, i.e. we
obtained several sufficient conditions for a parametric penalty function to be exact. However, there is an indirect
approach to the study of exact penalty functions. In this approach, one is concerned with a useful property of a
penalty function that is different from exactness, while the proof of the fact that a penalty function is exact comes as
a by-product. This property is the absence of stationary (critical) points of a penalty function not belonging
to the set of feasible solutions of the initial problem. 

The conditions under which a penalty function does not have any stationary points outside the set of feasible
solutions of the original problem are very important for applications, and they have been studied by different
researchers (see, e.g., \cite{DiPilloGrippo, DiPilloGrippo2, ExactBarrierFunc, Demyanov, Ye}). It should also be noted
that such conditions were the main tool for the study of singular exact penalty functions
\cite{HuyerNeumaier,Bingzhuang,WangMaZhou}. Despite all attention and importance, the property of a penalty function to
not have any infeasible stationary points has never been named.

Recall that for any function $g \colon X \to \overline{\mathbb{R}}$ and any nonempty set $K \subset X$ a point 
$x \in \dom g \cap K$ is called an inf-stationary point of the function $g$ with respect to the set $K$ iff
$g^{\downarrow}_K(x) \ge 0$.

\begin{definition}
Let $C \subseteq A$ be a nonempty set. The parametric penalty function $F_{\lambda}$ is said to be
\textit{feasiblity-preserving} on the set $C$ if there exists $\lambda_0 \ge 0$ such that for any 
$\lambda \ge \lambda_0$ there are no inf-stationary points of the function $F_{\lambda}$ with respect to 
the set $A \times P$ belonging to the set $(C \times P) \setminus (\Omega \times \{ p_0 \})$. In other words, the
penalty function $F_{\lambda}$ is feasibility-preserving on the set $C$ if for any sufficiently large $\lambda \ge 0$
and for any $(x, p) \in (C \times P) \cap \dom F_{\lambda}$ the inequality 
$(F_{\lambda})^{\downarrow}_{A \times P} (x, p) \ge 0$ implies that
$x \in \Omega$ and $p = p_0$. The greatest lower bound of all such $\lambda_0$ is denoted by $\lambda_{fp}(C)$, and
is referred to as \textit{the parameter of feasibility preservation}.
\end{definition}

\begin{remark}
{(i) Note that a penalty function $F_{\lambda}$ might not have any infeasible inf-stationary points with respect to the
set $A \times P$ for some $\lambda \ge 0$, and nevertheless be \textit{non}-feasibility-preserving. In particular, it is
not difficult to provide an example of a penalty function that does not have any infeasible inf-stationary points for
any $\lambda \in [0, \lambda_0)$, and have infeasible inf-stationary points for any $\lambda > \lambda_0$ for some 
$\lambda_0 < + \infty$ (it is sufficient to choose functions $f$ and $\varphi$ such that $f^{\downarrow}(x) < 0$ and
$\varphi^{\downarrow}(x, p) > 0$ for some infeasible $x$ and for all $p \in P$). Such pathological examples are
inconsistent with the general theory of (exact) penalty functions, since in general one expects that the greater is
the penalty parameter, the better a penalty function approximates the original constrained optimization problem. That is
why we excluded such pathological cases from further consideration by requiring that a feasibility-preserving penalty
function must not have any infeasible inf-stationary points for any $\lambda$ greater than some $\lambda_0 \ge 0$.
}

{\noindent(ii) Let $F_{\lambda}$ be feasibility-preserving on a set $C \subseteq A$. From the definition it follows
that for any $\lambda > \lambda_{fp}(C)$ and $x \in C \setminus A$ the inequality
$(F_{\lambda})^{\downarrow}_{A \times P}(x, p) < 0$ is satisfied \emph{for all} $p \ne p_0$. In some cases, it
can be useful to consider only a subset $K \subset P$, and to require that the inequality 
$(F_{\lambda})^{\downarrow}_{A \times P}(x, p) < 0$ is satisfied for all $p \in K \setminus \{ p_0 \}$.
(moreover, the set $K$ can depend on $x$). In other words, it can be convenient to consider a more general defintion of 
feasibility-preserving parametric penalty function in which the set $C \subseteq A$ is replaced by a set 
$C \subseteq A \times P$. The interested reader can extend the results below to this more general case.
}
\end{remark}

Feasibility-preserving penalty functions have a very important advantage over other penalty functions, which makes them
more appealing for applications. Namely, any standard minimization algorithm, when applied to a feasibility-preserving
penalty function $F_{\lambda}$ with $\lambda > \lambda_{fp}(A)$, converges to a feasible point of the original problem
(at least, in theory), since such algorithms, normally, cannot converge to a non-stationary point. 

In addition to its value from the computational point view, the concept of feasibility preservation can be applied to
the study of exact penalty functions.

\begin{proposition} \label{Prp_FeasibPreservImpliesExactness}
Let a penalty function $F_{\lambda}$ be feasibility-preserving on the set $A$, and suppose that there exists 
$\lambda_0 \ge 0$ such that for any $\lambda \ge \lambda_0$ the function $F_{\lambda}$ attains a global minimum on the
set $A \times P$. Then $F_{\lambda}$ is exact, and $\lambda^*(f, \varphi) \le \max\{ \lambda_{fp}(A), \lambda_0 \}$.
\end{proposition}

\begin{proof}
Choose $\lambda > \max\{ \lambda_{fp}(A), \lambda_0 \}$, and let $(x^*, p^*)$ be a point of global minimum of
$F_{\lambda}$ on the set $A \times P$ that exists due to our assumption. Clearly, the pair $(x^*, p^*)$ is an
inf-stationary point of $F_{\lambda}$. Therefore $p^* = p_0$ and $x^* \in \Omega$ by virtue of the fact that
$F_{\lambda}$ is feasibility-preserving. Hence, as it is easy to see, $x^*$ is a globally optimal solution of
the problem $(\mathcal{P})$ (recall that $F_{\lambda}(x, p_0) = f(x)$ for any $x \in \Omega$) and 
$\min_{(x, p) \in A} F_{\lambda}(x, p) = F_{\lambda}(x^*, p^*) = f^*$, which implies that $F_{\lambda}$ is exact, and 
$\lambda^*(f, \varphi) \le \max\{ \lambda_{fp}(A), \lambda_0 \}$ by virtue of Proposition~\ref{Prp_EquivDefExPen}.
\end{proof}

It should be noted that the assumption that the penalty function $F_{\lambda}$ attains a global minimum on the set 
$A \times P$ for any sufficiently large $\lambda$ is indispensable for the validity of the proposition above. Namely,
one can construct a parametric penalty function $F_{\lambda}$ that is feasibility-preserving on the set $A$, but is not
exact, which implies that it does not attain a global minimum on the set $A \times P$ due of the previous
proposition. The example of such penalty function is taken from \cite{LianZhang}.

\begin{example}
Let $X = \mathbb{R}^n$, and the problem $(\mathcal{P})$ have the form
\begin{equation} \label{ProblemInExampleOfIncorPenFunc}
  \min f(x) \quad \text{subject to} \quad F(x) = 0, \quad x \in [u, v],
\end{equation}
where the functions $f \colon \mathbb{R}^n \to \mathbb{R}$ and $F \colon \mathbb{R}^n \to \mathbb{R}^m$ are
continuously differentiable, $u, v \in \mathbb{R}^n$, and 
$[u, v] = \{ x \in \mathbb{R}^n \mid u_i \le x_i \le v_i, \: i \in \{1, \ldots, n \} \}$.

Let $P = \mathbb{R}_+$ and $p_0 = 0$. Fix $w \in \mathbb{R}^m$, and for any $p > 0$ denote 
$\Delta(x, p) = \| F(x) - p w \|^2$, where $\| \cdot \|$ is the Euclidean norm. Finally, choose $a > 0$, and define
\begin{equation} \label{IncorrectPenTerm}
  \varphi(x, p) = \begin{cases}
    0, & \text{if } p = 0, x \in \Omega, \\
    \dfrac{1}{2(p + 1)} \dfrac{\Delta(x, p)}{1 - a \Delta(x, p)} + p, & \text{if } p > 0, \Delta(x, p) < 1/a, \\
    + \infty, & \text{otherwise}.
  \end{cases}
\end{equation}
Note that $\varphi(x, p) = 0$ iff $p = 0$ and $x \in \Omega$. 

It was proved in \cite{LianZhang} that under some additional assumptions the parametric penalty function
$F_{\lambda}(x, p) = f(x) + \lambda \varphi(x, p)$ with the penalty term (\ref{IncorrectPenTerm}) is
feasibility-preserving on the set $A = [u, v]$. However, let us show that this penalty function is not exact.

Indeed, for any $x \in \mathbb{R}^n$ one has 
$g_{\lambda}(x) = \liminf_{p \to +0} F_{\lambda}(x, p) = f(x) + \lambda \phi(x)$, where
$$
  \phi(x) = \begin{cases}
    \dfrac{1}{2}\dfrac{\| F(x) \|^2}{1 - a \| F(x) \|^2}, & \text{if } \| F(x) \|^2 < 1/a, \\
    + \infty, & \text{if } \| F(x) \|^2 \ge 1/a.
  \end{cases}
$$
By Proposition~\ref{Prp_ExactLimInfPenFunc}, if the parametric penalty function $F_{\lambda}$ is
exact, then the penalty function $g_{\lambda}$ is exact as well. However, note that the penalty term $\phi$ is
continuously differentiable for any $x$ such that $\| F(x) \|^2 < 1/a$. Therefore the penalty function $g_{\lambda}$
cannot be exact in the general case. 

Indeed, let $n = m = 1$, $f(x) = x$, $F(x) = x - 1$, $u = 0$ and $v = 2$. The only feasible point of this problem is
$x^* = 1$. If the penalty function $g_{\lambda}$ is exact, then $x^*$ is a point of global minimum of the function
$g_{\lambda}$ on the set $[u, v] = [0, 2]$ for any sufficiently large $\lambda \ge 0$, which implies that
$g'_{\lambda}(x^*) = 0$ for any sufficiently large $\lambda$. On the other hand, 
$g'_{\lambda}(x^*) = f'(x^*) = 1$. Consequently, the penalty function $g_{\lambda}$ is not exact, which, with the use of
Proposition~\ref{Prp_ExactLimInfPenFunc}, implies that the penalty function $F_{\lambda}$ is not exact as well. In spite
of this fact, one can show that the penalty function $F_{\lambda}$ is feasibility-preserving (see \cite{LianZhang}).
Therefore one concludes that the penalty function $F_{\lambda}$ does not attain a global minimum on the set 
$A \times P = [u, v] \times \mathbb{R}_+$ for any sufficiently large $\lambda$ by virtue of
Proposition~\ref{Prp_FeasibPreservImpliesExactness}. Thus, the claim from \cite{LianZhang} that the penalty function 
$F_{\lambda}(x, p) = f(x) + \lambda \varphi(x, p)$ with the penalty term $\varphi$ defined by (\ref{IncorrectPenTerm})
is an exact penalty function for the problem~(\ref{ProblemInExampleOfIncorPenFunc}) (see~Theorem~2.2 in
\cite{LianZhang}) is not correct.
\end{example}

Let us obtain general sufficient conditions for a penalty function to be feasibility-preserving. In order to understand 
a natural way to formulate these conditions, we, at first, derive simple necessary conditions for a penalty function to
be feasibility-preserving.

For any set $C \subset A$ denote 
$$
  (C \times P)_{\inf} = \Big( \big( C \times P \big) \cap \dom \varphi \Big) \setminus 
  \Big( \Omega \times \{ p_0 \} \Big) 
$$
(throughout this section we suppose that the penalty term $\varphi$ is fixed). Thus, the set $(C \times P)_{\inf}$
consists of all those points $(x, p) \in (C \times P) \cap \dom \varphi$ that are not ``feasible'' for the
problem $(\mathcal{P})$.

\begin{proposition}
Let $C \subset A$ be a nonempty set. Suppose that the penalty function $F_{\lambda}$ is feasibility-preserving on $C$
for any function $f$ that is Lipschitz continuous on $X$. Then 
$$
  \varphi^{\downarrow}_{A \times P} (x, p) < 0 \quad 
  \forall (x, p) \in (C \times P)_{\inf}.
$$
\end{proposition}

\begin{proof}
Arguing by reductio ad absurdum, suppose that there exists $(y, q) \in (C \times P)_{\inf}$ such that 
$\varphi^{\downarrow}_{A \times P} (y, q) \ge 0$. Define $f(x) \equiv 0$. Then applying Lemma~\ref{Lemma_SumRule}
one obtains that
$$
  (F_{\lambda})^{\downarrow}_{A \times P} (y, q) \ge f^{\downarrow}_{A \times P} (x_0) + 
  \lambda \varphi^{\downarrow}_{A \times P} (y, q) \ge 0 + \lambda \cdot 0 = 0,
$$
which contradicts the assumption that $F_{\lambda}$ is feasibility-preserving on $C$ for any Lipschitz continuous
function $f$.	 
\end{proof} 

Under an additional assumption on the penalty function $F_{\lambda}$, that is satisfied, in particular, when $P$ is
a one-point set (i.e. when $F_{\lambda}$ is a non-parametric penalty function) or when the penalty term $\varphi$ is
Fr\'echet differentiable at all points $(x, p) \in (A \times P)_{\inf}$
(see~Theorem~\ref{Thrm_RSDLowerEstimViaPartialRSD}), one can obtain a stronger necessary condition for a penalty
function to be feasibility-preserving.

\begin{theorem} \label{Thrm_NessCondFeasPres}
Let $\Omega$ be closed, and $C \subset A$ be a nonempty set. Suppose that the penalty function $F_{\lambda}$ is
feasibility-preserving on the set $C$ for any function $f$ that is Lipschitz continuous on $X$, and there exists
$\kappa > 0$ such that for any $(x, p) \in (C \times P)_{\inf}$ one has
\begin{equation} \label{RSD_SumEstimate}
  (F_{\lambda})^{\downarrow}_{A \times P}(x, p) \ge 
  \kappa \min\big\{ 0, f^{\downarrow}_A(x) + \lambda \varphi(\cdot, p)^{\downarrow}_A (x), 
  \lambda \varphi(x, \cdot)^{\downarrow}(p) \big\}.
\end{equation}
Then there exists $a > 0$ such that
\begin{equation} \label{NessFPCond}
  \varphi(\cdot, p)^{\downarrow}_A (x) \le - a \quad
  \forall (x, p) \in \big\{ (y, q) \in (C \times P)_{\inf} \mid \varphi(y, \cdot)^{\downarrow}(q) \ge 0 \big\}.
\end{equation}
\end{theorem}

\begin{proof}
Arguing by reductio ad absurdum, suppose that for any $a > 0$ there exists $(x, p) \in (C \times P)_{\inf}$ such that 
$\varphi(x, \cdot)^{\downarrow}(p) \ge 0$, and $\varphi(\cdot, p)^{\downarrow}_A (x) > - a$. 

Then, in particular, there exists a sequence $\{ (x_n, p_n) \} \subset (C \times P)_{\inf}$ such that 
$\varphi(x_n, \cdot)^{\downarrow}(p_n) \ge 0$, $\varphi(\cdot, p_n)^{\downarrow}_A (x_n) < 0$ for any 
$n \in \mathbb{N}$, and $\varphi(\cdot, p_n)^{\downarrow}_A (x_n) \to 0$ as $n \to \infty$. Observe that $x_n$ is a
limit point of the set $A$ for any $n \in \mathbb{N}$, since otherwise 
$\varphi(\cdot, p_n)^{\downarrow}_A (x_n) = + \infty$.

Suppose, at first, that there is only a finite number of different points in the sequence $\{ x_n \}$. Then there exists
a point $x_0 \in \{ x_n \}_{n \in \mathbb{N}}$ and a subsequence $\{ p_{n_k} \}$ such that 
$\varphi(\cdot, p_{n_k})^{\downarrow}_A (x_0) \to 0$ as $n \to \infty$. If $x_0 \in \Omega$, then define 
$f(x) = d(x, x_0)$. Otherwise, there exists $r > 0$ such that $B(x_0, r) \cap \Omega = \emptyset$ due to the fact that
$\Omega$ is closed. Then define $f(x) = \min\{ d(x, x_0) - r, 0 \}$. Note that in both cases $f$ attains a global
minimum on the set $\Omega$, and is Lipschitz continuous on $X$.

Applying inequality (\ref{RSD_SumEstimate}), and Lemma~\ref{Lemma_SimpleCases} one obtains that
\begin{multline*}
  (F_{\lambda})^{\downarrow}_{A \times P} (x_0, p_{n_k}) \ge 
  \kappa \min\Big\{ 0, f^{\downarrow}_A (x_0) + \lambda \varphi(\cdot, p_{n_k})^{\downarrow}_A (x_0),
  \lambda \varphi(x_0, \cdot)^{\downarrow} (p_{n_k}) \Big\} \ge \\
  \ge \kappa \min\big\{ 1 + \lambda \varphi(\cdot, p_{n_k})^{\downarrow}_A (x_0), 0 \big\}
\end{multline*}
for any $\lambda \ge 0$ and $k \in \mathbb{N}$. From the fact that $\varphi(\cdot, p_{n_k})^{\downarrow}_A (x_0) \to 0$
as $k \to \infty$ it follows that for any $\lambda \ge 0$ there exists $k_0 \in \mathbb{N}$ such that for any 
$k \ge k_0$ one has $(F_{\lambda})^{\downarrow}_{A \times P} (x_0, p_{n_k}) \ge 0$, which contradicts the assumption
that $F_{\lambda}$ is feasibility-preserving on $C$ for any Lipschitz continuous function $f$.

Suppose, now, that there is an infinite number of different points in the sequence $\{ x_n \}$. Without loss of
generality we can suppose that $x_n \ne x_k$ for all $n, k \in \mathbb{N}$ such that $n \ne k$. Denote
$M_1 = \{ x_n \}_{n \in \mathbb{N}} \cap \Omega$ and $M_2 = \{ x_n \}_{n \in \mathbb{N}} \setminus \Omega$. Clearly, one
of these sets is infinite. 

Let $M_1$ be infinite. Then it can be identified with a subsequence $\{ x_{n_k} \} \subset \Omega$ of the sequence 
$\{ x_n \}$. For any $k \in \mathbb{N}$ define $r_k = \inf_{m > k} d( x_{n_k}, x_{n_m} )$.
If for some $k \in \mathbb{N}$ one has $r_k = 0$, then there exists a subsequence $\{ x_{n_{k_l}} \}$ with $k_l > k$
for any $l \in \mathbb{N}$ converging to the point $x_{n_k}$. Recall that all points in the original sequence 
$\{ x_n \}$ are distinct. Since a sequence in a metric space cannot converge to more than one point,
$\inf_{r > l} d( x_{n_{k_l}}, x_{n_{k_r}} ) > 0$ for all $l \in \mathbb{N}$. Therefore replacing, if necessary, the
sequence $\{ x_{n_k} \}$ by the subsequence $\{ x_{n_{k_l}} \}$ one can suppose that $r_k > 0$ for 
all $k \in \mathbb{N}$.

Define $\theta_1 = r_1$ and $\theta_k = \min\{ \theta_{k - 1}, r_k \}$ for any $k > 1$.
Thus, $\{ \theta_k \}$ is a non-increasing sequence such that $\theta_k \le r_k$ for all $k \in \mathbb{N}$. Denote
$$
  f(x) = \sum_{k = 1}^{\infty} f_k(x), \quad f_k(x) = \min\left\{ d(x, x_{n_k}) - \frac{\theta_k}{3}, 0 \right\}.
$$
Observe that the functions $f_k$, $k \in \mathbb{N}$, have disjoint supports. Consequently, the function $f$ is
correctly defined. Furthermore, it is easy to verify that $f$ attains a global minimum on $\Omega$ at the point
$x_{n_1}$.

Let us show that $f$ is Lipschitz continuous on $X$. Indeed, let $x, y \in X$ be arbitrary. Note that 
there exist $m, l \in \mathbb{N}$ such that $f(x) = f_m(x)$ and $f(y) = f_l(y)$ by virtue of the fact that the
functions $f_k$ have disjoint supports. Hence and from the fact that all functions $f_k$, $k \in \mathbb{N}$, are
Lipschitz continuous on $X$ with a Lipschitz constant $L_k \le 1$ it follows that
\begin{multline*}
  |f(x) - f(y)| = |f_m(x) + f_l(x) - f_m(y) - f_l(y)| \le \\
  \le |f_m(x) - f_m(y)| + |f_l(x) - f_l(y)| \le 2 d(x, y).
\end{multline*}
in the case $m \ne l$, and
$$
  |f(x) - f(y)| = |f_m(x) - f_m(y)| \le d(x, y)
$$
in the case $m = l$. Thus, $f$ is Lipschitz continuous on $X$, which implies that the penalty function 
$F_{\lambda} = f +  \lambda \varphi$ is feasibility-preserving on the set $C$. On the other hand,
taking into account the fact that $f(x) = f_k(x) = d(x, x_{n_k}) - \theta_k/3$ in a neighbourhood of the point
$x_{n_k}$, and applying inequality (\ref{RSD_SumEstimate}), and Lemma~\ref{Lemma_SimpleCases} one obtains that
\begin{multline*}
  (F_{\lambda})^{\downarrow}_{A \times P} (x_{n_k}, p_{n_k}) \ge 
  \kappa \min\Big\{ 0, f^{\downarrow}_A (x_{n_k}) + \lambda \varphi(\cdot, p_{n_k})^{\downarrow}_A (x_{n_k}),
  \lambda \varphi(x_{n_k}, \cdot)^{\downarrow} (p_{n_k}) \Big\} \ge \\
  \ge \kappa \min\big\{ 1 + \lambda \varphi(\cdot, p_{n_k})^{\downarrow}_A (x_{n_k}), 0 \big\}
\end{multline*}
for any $\lambda \ge 0$ and $k \in \mathbb{N}$. Recall that by construction one has 
$\varphi(\cdot, p_n)^{\downarrow}_A (x_n) \to 0$ as $n \to \infty$. Therefore for any $\lambda \ge 0$ there exists
$k_0 \in \mathbb{N}$ such that for any $k \ge k_0$ one has 
$(F_{\lambda})^{\downarrow}_{A \times P} (x_{n_k}, p_{n_k}) \ge 0$, which contradicts the fact that $F_{\lambda}$ is
feasibility-preserving on the set $C$.

Suppose, finally, that the set $M_1 = \{ x_n \}_{n \in \mathbb{N}} \cap \Omega$ is finite, but the set 
$M_2 = \{ x_n \}_{n \in \mathbb{N}} \setminus \Omega$ is infinite. Then one can identify $M_2$ with a subsequence 
$\{ x_{n_k} \} \subset A \setminus \Omega$ of the sequence $\{ x_n \}$. Repeating, if necessary, the same argument as in
the case when the set $M_1$ is infinite, one can suppose that $r_k = \inf_{m > k} d( x_{n_k}, x_{n_m} ) > 0$ for any 
$k \in \mathbb{N}$. Taking into account the fact that the set $\Omega$ is closed, one gets that for any 
$k \in \mathbb{N}$ there exists $s_k > 0$ such that $B(x_{n_k}, s_k) \cap \Omega = \emptyset$. Denote 
$\theta_k = \min\{ \theta_{k-1}, r_k, s_k \}$, and define
$$
  f(x) = \sum_{k = 1}^{\infty} f_k(x), \quad f_k(x) = \min\left\{ d(x, x_{n_k}) - \frac{\theta_k}{3}, 0 \right\}.
$$
Observe that the functions $f_k$, $k \in \mathbb{N}$, have disjoint supports that do not intersect with the set
$\Omega$. Therefore the function $f$ is correctly defined, and attains a global minimum on the set $\Omega$ 
(in fact, $f(x) \equiv 0$ on $\Omega$). 

Arguing in the same way as in the case when the set $M_1$ is infinite, one can show that the function $f$ is Lipschitz
continuous on $X$, and for any $\lambda \ge 0$ there exists $k_0 \in \mathbb{N}$ such that for any $k \ge k_0$ one has
$(F_{\lambda})^{\downarrow}_{A \times P} (x_{n_k}, p_{n_k}) \ge 0$, which contradicts the fact that $F_{\lambda}$ is
feasibility-preserving on the set $C$.	 
\end{proof}

\begin{remark}
Recall that throughout this article we suppose that the exists a globally optimal solution of the problem
$(\mathcal{P})$,
i.e. we suppose that the function $f$ attains a global minimum on the set $\Omega$, since otherwise the definition of
exact penalty function is meaningless. That is why, when we construct a function $f$ in the proof of the theorem above,
we must ensure that it attains a global minimum on $\Omega$.
\end{remark}

Let us show that condition (\ref{NessFPCond}) is not only necessary but also sufficient for the penalty function
$F_{\lambda}$ to be feasibility-preserving.

\begin{proposition} \label{Prp_SuffCondFeasPres}
Let $C \subset A$ be a nonempty set, and the function $f$ be Lipschitz continuous on an open set $V$ containing 
the set $C$. Suppose that there exists $a > 0$ such that
$$
  \varphi(\cdot, p)^{\downarrow}_A (x) \le - a \quad
  \forall (x, p) \in \big\{ (y, q) \in (C \times P)_{\inf} \mid \varphi(y, \cdot)^{\downarrow}(q) \ge 0 \big\}.
$$
Then the penalty function $F_{\lambda}$ is feasibility-preserving on the set $C$. Moreover, 
$\lambda_{fp}(C) \le L / a$, where $L$ is a Lipschitz constant of $f$ on $V$.
\end{proposition}

\begin{proof}
Let $(x, p) \in (C \times P)_{\inf}$ be arbitrary. If $\varphi(x, \cdot)^{\downarrow}(p) < 0$, then with the use of
Lemma~\ref{Lemma_PartialRSD} one gets that for any $\lambda > 0$ the following inequalities hold true
$$
  (F_{\lambda})^{\downarrow}_{A \times P} (x, p) \le F_{\lambda}(x, \cdot)^{\downarrow}(p) =
  \lambda \varphi(x, \cdot)^{\downarrow}(p) < 0.
$$
Suppose, now, that $\varphi(x, \cdot)^{\downarrow}(p) \ge 0$. Then by the assumption of the proposition
there exists a sequence $\{ x_n \} \subset A$ converging to the point $x$ such that
$$
  \varphi(\cdot, p)^{\downarrow}_A (x) = \lim_{n \to \infty} \frac{\varphi(x_n, p) - \varphi(x, p)}{d(x_n, x)} \le - a.
$$
Since $x \in C$, one can suppose that $\{ x_n \} \subset V$. Hence $f(x_n) - f(x) \le L d(x_n, x)$ for all 
$n \in \mathbb{N}$. Consequently, for any $\lambda > 0$ one gets
\begin{multline*}
  (F_{\lambda})^{\downarrow}_{A \times P} (x, p) \le F_{\lambda}(\cdot, p)^{\downarrow}_A(x) \le
  \liminf_{n \to \infty} \frac{F_{\lambda}(x_n, p) - F_{\lambda}(x, p)}{d(x_n, x)} \le \\
  \le L + \lambda \lim_{n \to \infty} \frac{\varphi(x_n, p) - \varphi(x, p)}{d(x_n, x)} \le
  L - \lambda a.
\end{multline*}
Therefore for any $\lambda > L / a$ one has $(F_{\lambda})^{\downarrow}_{A \times P} (x, p) < 0$, which implies
that the penalty function $F_{\lambda}$ is feasibility-preserving on the set $C$ and $\lambda_{fp}(C) \le L / a$.
 
\end{proof}

\subsection{Zero duality gap}

Sometimes, it might be difficult to verify whether a given parametric penalty function is exact (or
feasibility-preserving). In such cases, one can try to characterize a quality of the chosen penalty function via
different notions, the most important of which is the \emph{zero duality gap property}. Recall that the zero duality
gap property is said to hold true for the penalty function $F_{\lambda}(x, p) = f(x) + \lambda \varphi(x, p)$ if
$$
  \sup_{\lambda \ge 0} \inf_{(x, p) \in A \times P} F_{\lambda}(x, p) = \inf_{x \in \Omega} f(x).
$$
Note that if the penalty function $F_{\lambda}$ is exact, then the zero duality gap property obviously holds for
$F_{\lambda}$. However, as it easy to see, the converse is not true.

In this subsection, we obtain necessary and sufficient conditions for the zero duality gap property to hold true. These
conditions are similar to the ones for a standard (non-parametric) penalty function 
(see, e.g., \cite{RubinovYang,RubinovHuangYang}), and are expressed in terms of the lower semicontinuity of the optimal
value function of a perturbed optimization problem. 

\begin{remark} \label{Rmrk_ZeroDualityGap_ImageSpaceAnalysis}
It should be noted that although the fact that the zero duality gap property is equivalent to the lower semicontinuity
of the optimal value function (perturbation function) is well-known in constrained optimization, the validity of this
statement in the case of parametric penalty functions does not follow from any existing results. For instance, it does
not follow from the general results from the image space analysis (see~\cite{ZhuLi} and references therein). In
particular, it does not follow from Theorems~3.3 and 3.4 in \cite{ZhuLi} due to the fact that assumptions $\mathcal{A}$,
$\mathcal{B}$ and $\mathcal{C}$ from \cite{ZhuLi} need not be satisfed for a parametric penalty function. Note that
singular penalty functions (see Section~\ref{Section_SingPenFunc}) do not satisfy assumptions $\mathcal{B}$ and
$\mathcal{C}$, while smoothing penalty functions (see Section~\ref{Section_SmoothingPenFunc}) do not satisfy assumption
$\mathcal{C}$ from \cite{ZhuLi}.
\end{remark}

Consider the following perturbation of the problem $(\mathcal{P})$:
$$
  \min \quad f(x) \quad \text{subject to} \quad \inf_{p \in P} \varphi(x, p) \le \eta, \quad x \in A
  \eqno{(\mathcal{P}_{\eta})},
$$
where $\eta \ge 0$ is a perturbation parameter. For any $\eta \ge 0$ denote 
$$
  \Omega(\eta) = \big\{ x \in A \mid \inf_{p \in P} \varphi(x, p) \le \eta \big\}, \quad
  \beta(\eta) = \inf_{x \in \Omega(\eta)} f(x).
$$
Thus, $\Omega(\eta)$ is the set of feasible points of the problem ($\mathcal{P}_{\eta}$), while $\beta(\eta)$ is
the \emph{perturbation function} (or the \textit{optimal value function}) of this problem. Note that 
$\beta(0) = f^* := \min_{x \in \Omega} f(x)$, and $\beta \colon \mathbb{R}_+ \to \overline{\mathbb{R}}$ is a
non-increasing function.

\begin{theorem} \label{Thrm_DualityGap}
Let the penalty function $F_{\lambda}$ be bounded below on $A \times P$ for some $\lambda \ge 0$. Then 
the following equality holds true
$$
  \sup_{\lambda \ge 0} \inf_{(x, p) \in A \times P} F_{\lambda}(x, p) =
  \lim_{\eta \to +0} \beta(\eta).
$$
\end{theorem}

\begin{proof}
Denote
\begin{equation} \label{InfOfPenFunc}
  h(\lambda) = \inf_{(x, p) \in A \times P} F_{\lambda}(x, p) \quad \forall \lambda \ge 0.
\end{equation}
Observe that $h(\lambda) \le f^*$ for all $\lambda \ge 0$ due to the fact that $F_{\lambda}(x, p_0) = f(x)$ 
for any $x \in \Omega$.

Choose an arbitrary $\eta > 0$. From Lemma~\ref{Lmm_Reduction} it follows that there exists $\lambda_0 \ge 0$ such that
$$
  F_{\lambda}(x, p) \ge f^* \quad \forall (x, p) \notin \Omega_{\eta} \quad \forall \lambda \ge \lambda_0.
$$
Therefore for all $\lambda \ge \lambda_0$ and for any $\varepsilon > 0$ there exists $(x, p) \in \Omega_{\eta}$ such
that
$$
  f^* \ge h(\lambda) \ge F_{\lambda}(x, p) - \varepsilon \ge f(x) - \varepsilon
$$
(if $h(\lambda) = f^*$, then one can choose $(x, p) = (x^*, p_0) \in \Omega_{\eta}$, where $x^*$ is a globally optimal
solution of the problem $(\mathcal{P})$). Note that $\varphi(x, p) < \eta$, since $(x, p) \in \Omega_{\eta}$, which
implies that $x$ belongs to the set of feasible points $\Omega(\eta)$ of the problem ($\mathcal{P}_{\eta}$).
Consequently, $f(x) \ge \beta(\eta)$, which yields
$$
  h(\lambda) \ge f(x) - \varepsilon \ge \beta(\eta) - \varepsilon \quad \forall \lambda \ge \lambda_0.
$$
Taking the supremum over all $\lambda \ge 0$, and taking into account the fact that $\varepsilon > 0$ is arbitrary one
obtains that
$$
  \sup_{\lambda \ge 0} h(\lambda) \ge \beta(\eta),
$$
which yields
\begin{equation} \label{LimValPenFunc_LowerEstim}
  \sup_{\lambda \ge 0} h(\lambda) \ge \limsup_{\eta \to +0} \beta(\eta)
\end{equation}
due to the fact that $\eta > 0$ was chosen arbitrarily.

To prove the reverse inequality, choose a decreasing sequence $\{ \eta_k \} \subset (0, +\infty)$ such that 
$\eta_k \to 0$ as $k \to \infty$, and denote $\lambda_k = 1 / \sqrt{\eta_k}$. By the definition of the perturbation
function $\beta$, for any $k \in \mathbb{N}$ there exists $x_k \in A$ such that
$$
  f(x_k) \le \beta(\eta_k) + \sqrt{\eta_k}, \quad \inf_{p \in P} \varphi(x_k, p) \le \eta_k.
$$
Clearly, for any $k \in \mathbb{N}$ there exists $p_k \in P$ such that $\varphi(x_k, p_k) \le 2 \eta_k$.
Consequently, one has
$$
  h(\lambda_k) \le F_{\lambda_k}(x_k, p_k) = f(x_k) + \lambda_k \varphi(x_k, p_k) \le
  \beta(\eta_k) + \sqrt{\eta_k} + 2 \lambda_k \eta_k = \beta(\eta_k) + 3 \sqrt{\eta_k}.
$$
Passing to the limit inferior as $k \to \infty$ one gets
$$
  \liminf_{k \to \infty} h(\lambda_k) \le \liminf_{k \to \infty} \beta(\eta_k).
$$
Note that the function $h(\lambda)$ is non-decreasing due to the fact that the penalty function $F_{\lambda}$ is
non-decreasing in $\lambda$ (see~(\ref{InfOfPenFunc})). Hence 
$\sup_{\lambda \ge 0} h(\lambda) = \lim_{\lambda \to \infty} h(\lambda)$, which implies that
$$
  \sup_{\lambda \ge 0} h(\lambda) = \lim_{k \to \infty} h(\lambda_k) \le \liminf_{k \to \infty} \beta(\eta_k),
$$
due to the fact that $\lambda_k = 1 / \sqrt{\eta_k}$ and $\eta_k \to +0$ as $k \to \infty$. Taking into account the fact
that the sequence $\{ \eta_k \}$ was chosen arbitrarily one obtains that
\begin{equation} \label{LimValPenFunc_UpperEstim}
  \sup_{\lambda \ge 0} h(\lambda) \le \liminf_{\eta \to +0} \beta(\eta)
\end{equation}
Combining (\ref{LimValPenFunc_LowerEstim}) and (\ref{LimValPenFunc_UpperEstim}) one obtains the desired result.	 
\end{proof}

As a simple corollary to the theorem above we obtain a characterization of the zero duality gap property in terms of the
perturbation function $\beta$. In order to formulate this result note that if the zero duality gap property holds for
$F_{\lambda}$, then $F_{\lambda}$ is bounded below on $A \times P$ for some $\lambda \ge 0$ by the fact that 
$f^* > - \infty$.

\begin{theorem} \label{Thrm_ZeroDualityGapCharacterization}
The zero duality gap property holds true for the penalty function $F_{\lambda}$ if and only if $F_{\lambda}$ is bounded
below on $A \times P$ for some $\lambda \ge 0$, and the perturbation function $\beta$ is lower semicontinuous 
at the origin.
\end{theorem}

\begin{proof}
As it was mentioned above, $\beta(0) = f^* > -\infty$. Hence and from Theorem~\ref{Thrm_DualityGap} it follows that the
zero duality gap property holds true for $F_{\lambda}$ iff $F_{\lambda}$ is bounded below on $A \times P$ for some
$\lambda \ge 0$ and $\lim_{\eta \to +0} \beta(\eta) = \beta(0)$, i.e. iff $F_{\lambda}$ is bounded below on 
$A \times P$ for some $\lambda \ge 0$, and $\beta$ is continuous at the origin. It remains to note that $\beta$ is
continuous at the origin iff it is lower semicontinuous at the origin due to the fact that $\beta$ is a non-increasing
function.	 
\end{proof}

It is interesting to note that necessary and sufficient conditions for the penalty function $F_{\lambda}$ to be exact
can also be expressed in terms of the behaviour of the perturbation function $\beta$.

\begin{theorem}
The parametric penalty $F_{\lambda}$ is exact if and only if $F_{\lambda}$ is bounded below on $A \times P$ for
some $\lambda$, and the perturbation function $\beta$ is calm from below at the origin, i.e.
$$
  \liminf_{\eta \to +0} \frac{\beta(\eta) - \beta(0)}{\eta} > - \infty.
$$
\end{theorem}

\begin{proof}
The validity of the theorem follows directly from Corollary~\ref{Crlr_ReductionToStandExPen}, and \cite{Dolgopolik},
Theorem~3.24.	 
\end{proof}

Various sufficient conditions for a perturbation function of a constrained optimization problem to be lower
semicontinuous at the origin are well-known in parametric optimization (see, e.g., \cite{RubinovYang}, Section~3.1.6).
Below, we provide a modification of one of these conditions that takes into account the particular structure of the
perturbed problem ($\mathcal{P}_{\eta}$).

\begin{proposition} \label{Prp_ZeroDualityGapSuffCond}
Let $A$ be closed, $f$ be l.s.c. on $\Omega$ and $\varphi$ be l.s.c. on $A \times \{ p_0 \}$. Suppose that there
exists $\eta_0 > 0$ such that the set $\{ x \in \Omega(\eta_0) \mid f(x) < f^* \}$ is relatively compact. Suppose also
that there exists a non-decreasing function $\omega \colon \mathbb{R}_+ \to \mathbb{R}_+$ such that $\omega(t) = 0$ iff 
$t = 0$ and
\begin{equation} \label{PertFunc_LowerEstimPenTerm}
  \varphi(x, p) \ge \omega(d(p, p_0)) \quad \forall p \in P.
\end{equation}
Then the perturbation function $\beta(\eta)$ is l.s.c. at the origin.
\end{proposition}

\begin{proof}
Fix an arbitrary $\varepsilon > 0$. Our aim is to find $\overline{\eta} > 0$ such that 
$\beta(\eta) \ge \beta(0) - \varepsilon$ for all $\eta \in (0, \overline{\eta})$.

By definition, for any $x \in \Omega$ one has $f(x) \ge f^*$. Applying the lower semicontinuity of the function $f$ on
the set $\Omega$ one obtains that for any $x \in \Omega$ there exists a neighbourhood $U(x)$ of $x$ such that
$f(x) \ge f^* - \varepsilon$ for all $x \in U(x)$. Denote
$$
  V = \bigcup_{x \in \Omega} U(x), \quad C = \cl\big\{ x \in \Omega(\eta_0) \mid f(x) < f^* \big\}
$$
and $K = C \setminus V$, where ``cl'' stands for the closure operator. By the assumption of the theorem the set $C$ is
compact, while the set $V$ is open by construction. Consequently, the set $K$ is compact. Note also that $\Omega \subset
V$, and $C \subset A$ due to the fact that $A$ is closed, which implies that $K \subset A$ and 
$K \cap \Omega = \emptyset$. 

The function $\varphi(\cdot, p_0)$ is l.s.c. on the set $A$, since $\varphi$ is l.s.c. on the set $A \times \{ p_0 \}$.
Hence and from the facts that the set $K$ is compact, and $\varphi(x, p_0) > 0$ for any $x \notin \Omega$ one
obtains that there exists $\delta > 0$ such that $\varphi(x, p_0) \ge \delta$ for all $x \in K$. Applying the lower
semicontinuity of $\varphi$ on $A \times \{ p_0 \}$, and the compactness of the set $K$ again, one
can easily verify that there exists $r > 0$ such that
$$
  \varphi(x, p) \ge \frac{\delta}{2} \quad \forall (x, p) \in K \times B(p_0, r).
$$
On the other hand, if $x \in K$, but $p \notin B(p_0, r)$, then taking into account
(\ref{PertFunc_LowerEstimPenTerm}) one gets $\varphi(x, p) \ge \omega(r) > 0$, which yields
\begin{equation} \label{LowerEstimOfInfPenTerm}
  \inf_{p \in P} \varphi(x, p) \ge \min\left\{ \frac{\delta}{2}, \omega(r) \right\} \quad \forall x \in K.
\end{equation}
Denote $\overline{\eta} = \min\{ \eta_0, \delta / 2, \omega(r) \}$, and choose arbitrary $\eta < \overline{\eta}$ and 
$x \in \Omega(\eta)$. Observe that $x \in \Omega(\eta_0)$ due to the fact that $\eta < \overline{\eta} \le \eta_0$.
Furthermore, if $f(x) < f^*$, then $x \in C \cap V$, since otherwise $x \in K$, which is impossible due to
(\ref{LowerEstimOfInfPenTerm}) and the choice of $\overline{\eta}$. Hence and from the definition of the set $V$ it
follows that
$$
  f(x) \ge f^* - \varepsilon \quad \forall x \in \Omega(\eta) \quad \forall \eta \in (0, \overline{\eta}).
$$
Taking the infimum over all $x \in \Omega(\eta)$ one gets
$$
  \beta(\eta) = \inf_{x \in \Omega(\eta)} f(x) \ge f^* - \varepsilon = \beta(0) - \varepsilon \quad
  \forall \eta \in (0, \overline{\eta}).
$$
Thus, $\beta$ is l.s.c. at the origin.	 
\end{proof}

\begin{remark}
Note that if $X$ is a finite dimensional normed space, then, as it is well-known, the set 
$\{ x \in \Omega(\eta) \mid f(x) < f^* \}$ is relatively compact iff it is bounded. In particular, in this case it is
sufficient to suppose that either the set $\Omega(\eta)$ is bounded for some $\eta > 0$ or the function $f$ is coercive,
i.e. $f(x) \to + \infty$ as $\| x \| \to \infty$.
\end{remark}

\subsection{Minimizing sequences}

Let us discuss a computational aspect of the parametric penalty functions method. In practice, penalty functions can be
applied in the following way. One chooses an initial value $\lambda_1$ of the penalty parameter, and then finds
an unconstrained minimum of the penalty function. Then one increases the value of the penalty parameter, and minimizes
the penalty function again using the point obtained on the first step as the initial guess. This process is repeated
until a ``good'' approximation of a solution of the original problem is found. From the theoretical point of view, one
can look at this process as a procedure that generates a sequence $\{ (x_n, p_n) \}$ such that
\begin{equation} \label{MinimizingSequenceDef}
  F_{\lambda_n}(x_n, p_n) \le \inf_{(x, p) \in A \times P} F_{\lambda_n}(x, p) + \varepsilon,
\end{equation}
where $\lambda_n \to + \infty$, and $\varepsilon > 0$ is a given tolerance. Below, we study the important question of
how the sequence $\{ (x_n, p_n) \}$ behaves as $n \to \infty$. 

It should be noted that although there are many general results on convergence of penaly/barrier methods (see, e.g.,
\cite{FiaccoMcCormic,AuslenderCominettiHaddou,BenTal,Auslender}), none of them can be applied in the case of parametric
penalty functions. Nevertheless, the proofs of the theorems below are very similar to the proof of convergence of
penalty/barrier methods.

Let an increasing unbounded sequence $\{ \lambda_n \} \subset (0, + \infty)$ be fixed. The theorem below sharpens and
significantly generalizes Theorem~5.2 and Corollary~5.1 from \cite{WangMaZhou}.

\begin{theorem} \label{Thrm_MinimizingSequences}
Let a sequence $\{ (x_n, p_n) \} \subset A \times P$ satisfy the inequality
$$
  F_{\lambda_n}(x_n, p_n) \le \inf_{(x, p) \in A \times P} F_{\lambda_n}(x, p) + \varepsilon \quad
  \forall n \in \mathbb{N}
$$
for some $\varepsilon > 0$. Suppose also that the function $F_{\lambda}$ is bounded below on $A \times P$ for some
$\lambda \ge 0$. Then the following statements hold true:
\begin{enumerate}
\item{$\varphi(x_n, p_n) \to 0$ as $n \to \infty$;
\label{Item_PhiConv}
}

\item{if there exists a non-decreasing function $\omega \colon \mathbb{R}_+ \to \mathbb{R}_+$ such that $\omega(t) = 0$
iff $t = 0$ and $\varphi(x, p) \ge \omega(d(p, p_0))$ for all $p \in P$, then $p_n \to p_0$ as $n \to \infty$;
}

\item{if $A$ is closed, and $\varphi$ is l.s.c. on $A \times P$, then a cluster point of the sequence 
$\{ (x_n, p_n) \}$ (if exists) belongs to the set $\Omega \times \{ p_0 \}$.
}

\item{the following inequalities hold true
$$
  \lim_{\eta \to +0} \beta(\eta) \le \liminf_{n \to \infty} f(x_n) \le
  \limsup_{n \to \infty} f(x_n) \le \lim_{\eta \to +0} \beta(\eta) + \varepsilon,
$$
where $\beta$ is the perturbation function of the problem ($\mathcal{P}_{\eta}$);
\label{Item_FConv}
}

\item{if the zero duality gap property holds true, then 
$f^* \le \liminf_{n \to \infty} f(x_n) \le \limsup_{n \to \infty} f(x_n) \le f^* + \varepsilon$;
\label{Item_FConv_ZDG}
}
\end{enumerate}
\end{theorem}

\begin{proof}
\ref{Item_PhiConv}. Fix an arbitrary $\sigma \ge 0$. By the assumption of the theorem there exists 
$\mu \ge 0$ such that $c := \inf_{(x, p) \in A \times P} F_{\mu}(x, p) > - \infty$. Consequently, for any 
$(x, p) \in A \times P$ such that $\varphi(x, p) \ge \sigma$ one has
$$
  F_{\lambda_n}(x, p) = F_{\mu}(x, p) + (\lambda_n - \mu) \varphi(x, p) \ge c + (\lambda_n - \mu) \sigma.
$$
Applying the fact that the sequence $\{ \lambda_n \}$ is increasing and unbounded one obtains that for any sufficiently
large $n \in \mathbb{N}$ and for any $(x, p)$ such that $\varphi(x, p) \ge \sigma$ the following inequalities hold
true
$$
  \lambda_n > \mu + \frac{f^* + \varepsilon - c}{\sigma}, \quad
  F_{\lambda_n}(x, p) \ge c + (\lambda_n - \mu) \sigma > f^* + \varepsilon.
$$
On the other hand, taking into account the fact that $F_{\lambda}(x, p_0) = f(x)$ for all $x \in \Omega$ and 
$\lambda \ge 0$ one gets that for any $n \in \mathbb{N}$
$$
  F_{\lambda_n}(x_n, p_n) \le \inf_{(x, p) \in A \times P} F_{\lambda_n}(x, p) + \varepsilon \le
  \inf_{x \in \Omega} F_{\lambda_n}(x, p_0) + \varepsilon = f^* + \varepsilon.
$$
Therefore for any sufficiently large $n \in \mathbb{N}$ one has $\varphi(x_n, p_n) < \sigma$, which implies that
$\varphi(x_n, p_n) \to 0$ as $n \to \infty$ due to the fact that $\sigma > 0$ was chosen arbitrarily.

Note that the second and the third assertions of the theorem are direct corollaries to the first one.

\ref{Item_FConv}. Choose an arbitrary $\eta > 0$. From the first part of the proof it follows that there exists 
$n_0 \in \mathbb{N}$ such that $\varphi(x_n, p_n) < \eta$ for all $n \ge n_0$. Hence $x_n \in \Omega(\eta)$ for any 
$n \ge n_0$, where $\Omega(\eta)$ is the set of feasible points of the problem ($\mathcal{P}_{\eta}$). Consequently, one
has $f(x_n) \ge \beta(\eta)$ for any $n \ge n_0$, which yields 
$$
  \liminf_{n \to \infty} f(x_n) \ge \lim_{\eta \to +0} \beta(\eta).
$$
On the other hand, applying Theorem~\ref{Thrm_DualityGap}, and taking into account the facts that the penalty function
$F_{\lambda}$ is non-decreasing in $\lambda$, and $f(x) \le F_{\lambda}(x, p)$ for any $(x, p) \in A \times P$ and
$\lambda \ge 0$ one obtains
\begin{multline*}
  \limsup_{n \to \infty} f(x_n) \le \limsup_{n \to \infty} F_{\lambda_n}(x_n, p_n) \le 
  \limsup_{n \to \infty} \inf_{(x, p) \in A \times P} F_{\lambda_n}(x, p) + \varepsilon = \\
  = \sup_{\lambda \ge 0} \inf_{(x, p) \in A \times P} F_{\lambda_n}(x, p) + \varepsilon =
  \lim_{\eta \to +0} \beta(\eta) + \varepsilon,
\end{multline*}
that completes the proof of the fourth assertion of the theorem.

The validity of the fifth assertion of the theorem follows directly from the fourth assertion,
Theorem~\ref{Thrm_DualityGap} and the definition of the zero duality gap property.	 
\end{proof}

\begin{corollary} \label{Crlr_MinimizingSequences}
Let a sequence $\{ (x_n, p_n) \} \subset A \times P$ satisfy the inequality
$$
  F_{\lambda_n}(x_n, p_n) \le \inf_{(x, p) \in A \times P} F_{\lambda_n}(x, p) + \varepsilon_n, \quad 
  \forall n \in \mathbb{N}
$$
for some sequence $\{ \varepsilon_n \} \subset (0, + \infty)$ such that $\varepsilon_n \to 0$ as 
$n \to \infty$. Suppose also that the zero duality gap property holds true. Then $\varphi(x_n, p_n) \to 0$ and
$f(x_n) \to f^*$ as $n \to \infty$. Moreover, if $A$ is closed, $f$ is l.s.c. on $A$, and $\varphi$ is l.s.c. on 
$A \times P$, then any cluster point of the sequence $\{ (x_n, p_n \}$ (if exists) has the form $(x^*, p_0)$, where
$x^*$ is a globally optimal solution of the problem $(\mathcal{P})$.
\end{corollary}

Note that the assumptions on the penalty function $F_{\lambda}$ in the corollary above are satisfied when the penalty
function $F_{\lambda}$ is exact. However, in the case when $F_{\lambda}$ is exact, there is no need to choose 
an unbounded sequence $\{ \lambda_n \}$, since it is sufficient to choose any $\lambda > \lambda^*(f, \varphi)$.

\begin{proposition} \label{Prp_MinimizingSeq_ExactPenFunc}
Let the penalty function $F_{\lambda}$ be exact, and let a sequence $\{ (x_n, p_n) \} \subset A \times P$ be such that
$$
  \lim_{n \to \infty} F_{\lambda_0}(x_n, p_n) = \inf_{(x, p) \in A \times P} F_{\lambda_0} (x, p)
$$
for some $\lambda_0 > \lambda^*(f, \varphi)$. Then $\varphi(x_n, p_n) \to 0$ and $f(x_n) \to f^*$ as $n \to \infty$.
\end{proposition}

\begin{proof}
Choose an arbitrary $\varepsilon > 0$. By the definition of the sequence $\{ (x_n, p_n) \}$ there exists 
$n_0 \in \mathbb{N}$ such that
\begin{equation} \label{ExPenFunc_MinimizSeq}
  F_{\lambda_0}(x_n, p_n) < \inf_{(x, p) \in A \times P} F_{\lambda_0} (x, p) + 
  \varepsilon ( \lambda_0 - \lambda^*(f, \varphi) ) \quad \forall n \ge n_0.
\end{equation}
The exactness of the penalty function $F_{\lambda}$ implies that
$$
  \inf_{(x, p) \in A \times P} F_{\lambda_0} (x, p) = 
  \inf_{(x, p) \in A \times P} F_{\lambda^*(f, \varphi)} (x, p) = f^*.
$$
Consequently, for any $(x, p) \in A \times P$ such that $\varphi(x, p) \ge \varepsilon$ one has
$$
  F_{\lambda_0}(x, p) = F_{\lambda^*(f, \varphi)}(x, p) + (\lambda_0 - \lambda^*(f, \varphi)) \varphi(x, p) \ge
  f^* + (\lambda_0 - \lambda^*(f, \varphi)) \varepsilon.
$$
Hence and from (\ref{ExPenFunc_MinimizSeq}) it follows that $\varphi(x_n, p_n) < \varepsilon$ for any $n \ge n_0$,
which implies that $\varphi(x_n, p_n) \to 0$ as $n \to \infty$.

Choose arbitrary $\eta > 0$. From the first part of the proof it follows that there exists $n_0 \in \mathbb{N}$ such
that for any $n \ge n_0$ one has $\varphi(x_n, p_n) < \eta$. Consequently, $x_n \in \Omega(\eta)$ for any 
$n \ge n_0$, which implies that $f(x_n) \ge \beta(\eta)$ for all $n \ge n_0$. Therefore
$$
  \liminf_{n \to \infty} f(x_n) \ge \lim_{\eta \to +0} \beta(\eta) = f^*,
$$
where the last equality follows from Theorem~\ref{Thrm_DualityGap}, and the fact that the zero duality gap property
holds true due to the exactness of the penalty function $F_{\lambda}$. On the other hand, taking into account the fact
that $f(x) \le F_{\lambda}(x, p)$ for any $\lambda \ge 0$ and $(x, p) \in A \times P$ one obtains that
$$
  \limsup_{n \to \infty} f(x_n) \le \lim_{n \to \infty} F_{\lambda_0}(x_n, p_n) = 
  \inf_{(x, p) \in A \times P} F_{\lambda_0}(x, p) = f^*,
$$
which implies the desired result.	 
\end{proof}

Let us also obtains an interesting property of cluster points of a minimizing sequence constructed with the use of a
parametric penalty function (i.e. a sequence satisfying (\ref{MinimizingSequenceDef})) in the case when this penalty
function is not exact.

\begin{proposition} \label{Prp_ConvergenceToNonExactPoint}
Let $A$ be closed, $f$ be l.s.c. on $A$ and $\varphi$ be l.s.c. on $A \times P$. Suppose also that the zero
duality gap property holds true for the penalty function $F_{\lambda}$, and $F_{\lambda}$ is not exact. Let a sequence
$\{ (x_n, p_n) \} \subset A \times P$ satisfy the inequality
\begin{equation} \label{NonExactPenFunc_SubOptSol}
  F_{\lambda_n}(x_n, p_n) < f^* \quad \forall n \in \mathbb{N}.
\end{equation}
Then any cluster point of the sequence $\{ (x_n, p_n) \}$ (if exists) has the form $(x^*, p_0)$, where $x^*$ is a
globally optimal solution of the problem $(\mathcal{P})$ such that the penalty function $F_{\lambda}$ is not exact at
$x^*$.
\end{proposition}

\begin{proof}
Note that $\inf_{(x, p) \in A \times P} F_{\lambda_n} (x, p) < f^*$ for all $n \in \mathbb{N}$ by virtue of
Proposition~\ref{Prp_EquivDefExPen}, and the fact that $F_{\lambda}$ is not exact. Therefore there exists 
a sequence $\{ (x_n, p_n) \}$ satisfying (\ref{NonExactPenFunc_SubOptSol}).

Denote $\varepsilon_n = f^* - \inf_{(x, p) \in A \times P} F_{\lambda_n} (x, p) > 0$ for any $n \in \mathbb{N}$. Then
$$
  F_{\lambda_n}(x_n, p_n) \le \inf_{(x, p) \in A \times P} F_{\lambda_n} (x, p) + \varepsilon_n \quad
  \forall n \in \mathbb{N}.
$$
Taking into account the fact that the zero duality gap property holds true one obtains that $\varepsilon_n \to 0$ as 
$n \to \infty$. Hence and from Corollary~\ref{Crlr_MinimizingSequences} one gets that any cluster point of the sequence
$\{ (x_n, p_n) \}$ (if exists) has the form $(x^*, p_0)$, where $x^*$ is a globally optimal solution of the problem
$(\mathcal{P})$.

Let $(x^*, p_0)$ be a cluster point of the sequence $\{ (x_n, p_n) \}$. Let us show that the penalty function
$F_{\lambda}$ is not exact at the point $x^*$. Arguing by reductio ad absurdum, suppose that $F_{\lambda}$ is exact at
$x^*$, and fix an arbitrary $\lambda_0 > \lambda^*(x^*)$. Then there exists $r > 0$ such that
$$
  F_{\lambda_0}(x, p) \ge F_{\lambda_0}(x^*, p_0) = f(x^*) = f^* \quad 
  \forall (x, p) \in \big( U(x^*, r) \cap A \big) \times U(p_0, r).
$$
Consequently, applying the facts that $\lambda_n \to \infty$ as $n \to \infty$, and $F_{\lambda}$ is non-decreasing with
respect to $\lambda$ one obtains that there exists $n_0 \in \mathbb{N}$ such that for any $n \ge n_0$ one has 
$\lambda_n \ge \lambda_0$ and
$$
  F_{\lambda_n}(x, p) \ge F_{\lambda_n}(x^*, p_0) = f^* \quad 
  \forall (x, p) \in \big( U(x^*, r) \cap A \big) \times U(p_0, r).
$$
Therefore taking into account the definition of the sequence $\{ (x_n, p_n) \}$ one gets that 
$(x_n, p_n) \notin U(x^*, r) \times U(p_0, r)$ for any $n \ge n_0$, which contradicts the fact that $(x^*, p_0)$ is a
cluster point of the sequence $\{ (x_n, p_n) \}$. Thus, $F_{\lambda}$ is not exact at the point $x^*$.	 
\end{proof}

Let us consider a simple application of the proposition above.

\begin{example}
Let $g_{\lambda}$ be the $\ell_1$ penalty function for the nonlinear programming problem
\begin{equation} \label{MathProgProblem_l1PenFunc}
  \min f(x) \quad \text{subject to} \quad a_i(x) = 0, \quad i \in I, \quad
  b_j(x) \le 0, \quad j \in J,
\end{equation}
where $I = \{ 1, \ldots, m \}$ and $J = \{1, \ldots, l\}$, i.e. let
$$
  g_{\lambda}(x) = f(x) + \lambda \Big( \sum_{i = 1}^m |a_i(x)| + \sum_{j = 1}^m \max\{ 0, b_j(x) \} \Big).
$$
Here $f, a_i, b_j \colon \mathbb{R}^d \to \mathbb{R}$ are continuously differentiable. Suppose that there exists
$\lambda_0 \ge 0$ such that the function $g_{\lambda_0}$ is coercive.

As it is well known (see, e.g., \cite{HanMangasarian}), if the Mangasarian-Fromowitz constraint qualification
(MFCQ) holds true at a locally optimal solution $x^*$ of the problem \eqref{MathProgProblem_l1PenFunc}, then the penalty
function $g_{\lambda}$ is exact at $x^*$. Hence and from Theorem~\ref{Thrm_ExactnessNessSuffCond_NonDegeneracy} it
follows that the penalty function $g_{\lambda}$ is exact, provided MFCQ holds true at every global minimum of 
the problem~\eqref{MathProgProblem_l1PenFunc}.

Let, now, $\{ \lambda_n \}$  be an increasing unbounded sequence, and let a sequence $\{ x_n \} \subset \mathbb{R}^d$
be such that $g_{\lambda_n}(x_n) \le \inf_{x \in \mathbb{R}^d} g_{\lambda_n}(x) + \varepsilon_n$ for any 
$n \in \mathbb{N}$ with sufficiently small $\varepsilon_n > 0$. Then from
Proposition~\ref{Prp_ConvergenceToNonExactPoint} it follows that
either the $\ell_1$ penalty function $g_{\lambda}$ is exact, and $x_n$ is a globally optimal solution of 
the problem~\eqref{MathProgProblem_l1PenFunc} for some $n \in \mathbb{N}$, or MFCQ fails to hold true at cluster points
of the sequence $\{ x_n \}$ that exist due to out assumption that $g_{\lambda}$ is coercive for sufficiently large
$\lambda$. 

Thus, roughly speaking, if the $\ell_1$ penalty function is not exact, then a minimizing sequence constructed
with the use of this penalty function converges to a globally optimal solution of the
problem~\eqref{MathProgProblem_l1PenFunc} at which a constraint qualification fails to hold true. In contrast, a
minimizing sequence constructed with the use of the standard (quadratic) smooth penalty function can converge to any
globally optimal solution of the problem~\eqref{MathProgProblem_l1PenFunc}.
\end{example}

\section{Singular penalty functions}
\label{Section_SingPenFunc}

In this section, we apply the general theory of parametric penalty functions developed above to the study of a class of
smooth exact penalty functions introduced by Huyer and Neumaier in \cite{HuyerNeumaier} (see also 
\cite{WangMaZhou,Dolgopolik_OptLet,Dolgopolik_OptLet2}). We refer to penalty functions from this class as
\textit{singular penalty functions}. In the terminology of this paper, the class of singular penalty functions
consists of parametric penalty function with the set of parameters $P$ equal to $\mathbb{R}_+$. We obtain new simple
sufficient conditions for the local exactness of singular penalty functions, and demonstrate how the theory developed in
the previous section can be utilized in order to derive new results on these penalty functions. The results of this
section sharpen and generalize the main results of the paper \cite{WangMaZhou}.

\subsection{Local exactness of singular penalty functions}

Let the set $M \subset X$ have the form $M = \{ x \in X \mid 0 \in G(x) \}$, where 
$G \colon X \rightrightarrows Y$ is a set-valued mapping with closed values, and $Y$ is a normed space. In other
words, in this section we suppose that the problem $(\mathcal{P})$ has the form
$$
  \min f(x) \quad \text{subject to} \quad 0 \in G(x), \quad x \in A.
$$
Note that for the problem above one has $\Omega = G^{-1}(0) \cap A$.

Let also $P = \mathbb{R}_+$ and $p_0 = 0$. Choose an arbitrary $w \in Y$, and non-decreasing functions 
$\phi \colon [0, + \infty] \to [0, + \infty]$ and $\omega \colon \mathbb{R}_+ \to [0, + \infty]$
such that $\phi(t) = 0$ iff $t = 0$ and $\omega(t) = 0$ iff $t = 0$. Being inspired by the ideas of
\cite{HuyerNeumaier,WangMaZhou}, introduce the penalty term $\varphi(x, p)$ for the problem $(\mathcal{P})$ as follows
\begin{equation} \label{PenaltyTermOneDim}
  \varphi(x, p) = \begin{cases}
    0, & \text{if } x \in \Omega, p = 0, \\
    + \infty, & \text{if } x \notin \Omega, p = 0, \\
    p^{-1} \phi(d(0, G(x) - p w)^2) + \omega(p), & \text{if } p > 0,
  \end{cases}
\end{equation}
where
$$
  d(0, G(x) - p w) = \begin{cases}
    \inf_{y \in G(x)} \| y - p w \|, & \text{if } G(x) \ne \emptyset, \\
    + \infty, & \text{if } G(x) = \emptyset,
  \end{cases}
$$
and define the parametric penalty function $F_{\lambda}(x, p) = f(x) + \lambda \varphi(x, p)$. Hereafter, the penalty
function $F_{\lambda}$ is referred to as a \textit{singular penalty function}.

\begin{remark}
Note that one needs to include $+\infty$ into the domain of the function $\phi$ in order to allow the set $G(x)$ to be
empty for some $x \in X$.
\end{remark}

Our aim is to obtain simple sufficient conditions for the penalty function $F_{\lambda}$ to be (locally or globally)
exact or feasibility-preserving. We start with sufficient conditions for local exactness. 

Recall that a set-valued mapping $Q \colon X \rightrightarrows Y$ is called \textit{metrically subregular} with
constant $\tau > 0$ at a point $(\overline{x}, \overline{y}) \in X \times Y$ with
$\overline{y} \in Q(\overline{x})$ if there exists a neighbourhood $U$ of $x$ such that
$$
  d(\overline{y}, Q(x)) \ge \tau d(x, Q^{-1}(\overline{y})) \quad \forall x \in U.
$$
We say that the set-valued mapping $Q$ is \textit{metrically subregular with respect to a set} $C \subseteq X$ with
constant $\tau > 0$ at a point $(\overline{x}, \overline{y}) \in C \times Y$ with $\overline{y} \in Q(\overline{x})$ if
the restriction of the mapping $Q$ to the set $C$ is metrically subregular with constrant $\tau > 0$ at 
$(\overline{x}, \overline{y})$ or, equivalently, if there exists a neighbourhood $U$ of $x$ such that
$$
  d(\overline{y}, Q(x)) \ge \tau d(x, Q^{-1}(\overline{y}) \cap C) \quad \forall x \in U \cap C.
$$
For various necessary and sufficient conditions for metric subregularity see \cite{Kruger} and references therein.

The following theorem extends Theorem~3 from \cite{Dolgopolik_OptLet} to the case of arbitrary $w \in Y$.

\begin{theorem} \label{Thrm_LocalExactness}
Let $x^* \in \Omega$ be a locally optimal solution of the problem $(\mathcal{P})$, $f$ be Lipschitz continuous near
$x^*$,
and $G$ be metrically subregular with respect to the set $A$ at $(x^*, 0)$. Suppose also that there exist 
$t_0, \phi_0, \omega_0 > 0$ such that
\begin{equation} \label{PhiOmegaLowerEstimate}
  \phi(t) \ge \phi_0 t, \quad \omega(t) \ge \omega_0 t \quad \forall t \in [0, t_0].
\end{equation}
Then the penalty function $F_{\lambda}$ is exact at $(x^*, 0)$, and
\begin{equation} \label{ExPenParamEstimOneDim}
  \lambda^*(x^*) \le \frac{L (\sqrt{\phi_0} \| w \| + \sqrt{\omega_0 + \phi_0 \| w \|})}{2 \omega_0 \sqrt{\phi_0} \tau},
\end{equation}
where $L$ is a Lipschitz constant of $f$ near $x^*$, and $\tau > 0$ is the constant of metric subregularity of $G$ with
respect to the set $A$ at $(x^*, 0)$.
\end{theorem}

\begin{proof}
From the facts that $f$ is Lipschitz continuous in a neighbourhood of $x^*$, and $x^*$ is a locally optimal solution of
the problem $(\mathcal{P})$ it follows that there exists $r_1 > 0$ such that
\begin{equation} \label{LowerLipEstimObjFunc}
  f(x) - f(x^*) \ge - L d(x, \Omega) \quad \forall x \in B(x^*, r_1),
\end{equation}
where $L > 0$ is a Lipschitz constant of $f$ in a neighbourhood of $x^*$ (see, e.g., \cite{Dolgopolik},
Proposition~2.7). Let us also obtain a lower estimate of the penalty term $\varphi$.

Applying the well-known inequality $\| u - v \| \ge | \| u \| - \| v \| |$, and taking into account the fact that 
the mapping $G$ is metrically subregular with respect to the set $A$ at $(x^*, 0)$ one obtains that there exist 
$\tau > 0$ and $r_2 \in (0, r_1)$ such that
$$
  d(0, G(x) - p w) \ge d(0, G(x)) - p \| w \| \ge \tau d(x, \Omega) - p \| w \|
$$
for any $x \in B(x^*, r_2) \cap A$ and $p \ge 0$. Hence for any $x \in B(x^*, r_2) \cap A$ and $p \ge 0$ such that 
$\tau d(x, \Omega) \ge p \| w \|$ one has
$$
  d(0, G(x) - p w)^2 \ge (\tau d(x, \Omega) - p \| w \|)^2
$$
Note that since $x^* \in \Omega$, for any $x \in B(x^*, \sqrt{t_0} / 2 \tau)$ and 
$p \in (0, \widehat{p})$ one has $(\tau d(x, \Omega) - p \| w \|)^2 \le t_0$, where 
$\widehat{p} = \sqrt{t_0} / 2 \| w \|$ in the case $w \ne 0$, and $\widehat{p} = + \infty$ otherwise. Therefore taking
into account (\ref{PhiOmegaLowerEstimate}) and the fact that the function $\phi$ is non-decreasing one gets that
\begin{multline*}
  \phi( d(0, G(x) - p w)^2 ) \ge \phi( (\tau d(x, \Omega) - p \| w \|)^2 ) \ge \\
  \ge \phi_0 \big( \tau^2 d(x, \Omega)^2 - 2 p \| w \| \tau d(x, \Omega) + p^2 \| w \|^2 \big).
\end{multline*}
for any $x \in B(x^*, \overline{r}) \cap A$ and $p \in (0, \widehat{p})$ such that 
$\tau d(x, \Omega) \ge p \| w \|$, where $\overline{r} = \min\{ \sqrt{t_0} / 2 \tau, r_2 \}$. Hence and from
(\ref{LowerLipEstimObjFunc}) it follows that for all $\lambda \ge 0$ and for any 
$x \in B(x^*, \overline{r}) \cap A$ and $p \in (0, \widehat{p})$ such that 
$\tau d(x, \Omega) \ge p \| w \|$ one has
\begin{multline*}
  F_{\lambda}(x, p) - F_{\lambda}(x^*, 0) = f(x) - f(x^*) + 
  \lambda\left( \frac{1}{p} \phi( d(0, G(x) - p w)^2 ) + \omega(p) \right) \ge \\
  \ge - L d(x, \Omega) + \frac{\lambda}{p} \phi_0 \tau^2 d(x, \Omega)^2 - 
  2 \lambda \phi_0 \| w \| \tau d(x, \Omega) +   \lambda \phi_0 \| w \|^2 p + \lambda \omega(p).
\end{multline*}
Minimizing the right-hand side of the last inequality with respect to $d(x, \Omega)$ one obtains
$$
  F_{\lambda}(x, p) - F_{\lambda}(x^*, 0) \ge - \left( \frac{L^2}{4 \lambda \phi_0 \tau^2} +
  \frac{L \| w \|}{\tau} \right) p + \lambda \omega(p).
$$
Consequently, applying (\ref{PhiOmegaLowerEstimate}) one gets that
\begin{equation} \label{PartlyExactOneDimParam}
  F_{\lambda}(x, p) - F_{\lambda}(x^*, 0) \ge
  \left( - \frac{L^2}{4 \lambda \phi_0 \tau^2} - \frac{L \| w \|}{\tau} + \lambda \omega_0 \right) p \ge 0
\end{equation}
for any $x \in B(x^*, \overline{r}) \cap A$, $p \in (0, \overline{p})$ and $\lambda \ge \overline{\lambda}$ such
that $\tau d(x, \Omega) \ge p \| w \|$, where
$$
  \overline{p} = \min\left\{ \widehat{p}, t_0 \right\}, \quad
  \overline{\lambda} = 
  \frac{L (\sqrt{\phi_0} \| w \| + \sqrt{\omega_0 + \phi_0 \| w \|^2})}{2 \omega_0 \sqrt{\phi_0} \tau}
$$
($\overline{\lambda}$ is the smallest positive $\lambda$ for which inequality (\ref{PartlyExactOneDimParam}) is
satisfied for any $p > 0$). 

On the other hand, if $x \in B(x^*, \overline{r}) \cap A$ and $p \in (0, \overline{p})$ are such that 
$\tau d(x, \Omega) < p \| w \|$ (note that in this case $\| w \| > 0$), then applying (\ref{PhiOmegaLowerEstimate}), and
taking into account the fact that the function $\phi$ is nonnegative one obtains that
\begin{multline*}
  F_{\lambda}(x, p) - F_{\lambda}(x^*, 0) = f(x) - f(x^*) + \frac{\lambda}{p} \phi( d(0, G(x) - p w)^2 ) + 
  \lambda \omega(p) \ge \\
  \ge - L d(x, \Omega) + \lambda \omega_0 p \ge 
  \left( - L + \lambda \frac{\tau \omega_0}{\| w \|} \right) d(x, \Omega) \ge 0
\end{multline*}
for any $\lambda \ge \overline{\lambda} \ge L \| w \| / \tau \omega_0$. Furthermore, if $p = 0$, then
$F_{\lambda}(x, 0) = + \infty \ge F_{\lambda}(x^*, 0)$ for any $x \notin \Omega$, and
$F_{\lambda}(x, 0) = f(x) \ge f(x^*) = F_{\lambda}(x^*, 0)$ for any $x \in B(x^*, \delta) \cap \Omega$, provided
$\delta > 0$ is sufficiently small, due to the fact that $x^*$ is a locally optimal solution of the problem
$(\mathcal{P})$. Therefore for any $\lambda \ge \overline{\lambda}$ one has
$$
  F_{\lambda}(x, p) \ge F_{\lambda}(x^*, 0) \quad 
  \forall (x, p) \in \big( B(x^*, \min\{ \overline{r}, \delta \}) \cap A \big) \times [0, \overline{p}).
$$
Thus, the penalty function $F_{\lambda}$ is exact at $x^*$, and inequality (\ref{ExPenParamEstimOneDim}) is valid.
\end{proof}

\subsection{Some properties of singular penalty functions}

In order to apply the general theorems on the global exactness of parametric penalty functions
(Theorems~\ref{Th_GlobalExGeneralCase} and \ref{Th_GlobExFiniteDim}) to the penalty function $F_{\lambda}$ introduced
above, we need to be able to check the assumptions of these theorems, including the lower semicontinuity of the penalty
term $\varphi$, and the boundedness of the set $\Omega_{\delta}$. The following auxiliary results are useful for
verifying these assumptions.

\begin{lemma} \label{Lmm_LSC_Criterion}
Let $\phi$ and $\omega$ be l.s.c. on $\mathbb{R}_+$, and the function $(x, p) \to d(0, G(x) - p w)$ be l.s.c. on 
$X \times \mathbb{R}_+$. Then the penalty term $\varphi(x, p)$ is l.s.c. on $X \times \mathbb{R}_+$.
\end{lemma}

\begin{proof}
For any $\varepsilon > 0$ introduce the function
$$
  \varphi_{\varepsilon}(x, p) = \frac{1}{p + \varepsilon} \phi( d(0, G(x) - p w)^2 ) + \omega(p) \quad
  \forall (x, p) \in X \times \mathbb{R}_+.
$$
With the use of the facts that the function $(x, p) \to d(0, G(x) - p w)$ is l.s.c., and the function $\phi$ is
l.s.c. and non-decreasing, one can verify that the function $(x, p) \to \phi( d(0, G(x) - p w)^2 )$ is l.s.c., which,
as it is easy to see, implies that the function $(x, p) \to \phi( d(0, G(x) - p w)^2 ) / (p + \varepsilon)$ is l.s.c.
as well. Hence the function $\varphi_{\varepsilon}$ is l.s.c. as the sum of two l.s.c. functions. Note that
$$
  \varphi(x, p) = \sup_{\varepsilon > 0} \varphi_{\varepsilon}(x, p) \quad 
  \forall (x, p) \in X \times \mathbb{R}_+.
$$
Therefore $\varphi$ is l.s.c. as the supremum of a family of l.s.c. functions.	 
\end{proof}

\begin{remark}
Let us note that by \cite{BorweinZhu}, Theorem~5.1.20, the function $(x, p) \to d(0, G(x) - p w)$ is l.s.c., provided 
the multifunction $G(x)$ is outer semicontinuous (upper semicontinuous; cf. the terminologies in \cite{BorweinZhu} and
\cite{RockWets}).
\end{remark}

\begin{lemma} \label{Lmm_UpperEstimPenTerm}
Let there exist $t_0, \Phi_0, \Omega_0 > 0$ such that
$$
  \phi(t) \le \Phi_0 t, \quad \omega(t) \le \Omega_0 t \quad \forall t \in [0, t_0].
$$
Then there exists $\delta_0 > 0$ such that for any $x \in X$ with $d(0, G(x)) < \delta_0$ one has
\begin{equation} \label{InfimumPenTerm_UpperEstimate}
  \inf_{p \ge 0} \varphi(x, p) \le C_0 d(0, G(x)), \quad
  C_0 := 2 \big( \sqrt{\Phi_0^2 \| w \|^2 + \Phi_0 \Omega_0} + \Phi_0 \| w \| \big).
\end{equation}
\end{lemma}

\begin{proof}
Choose an arbitrary $\delta \in (0, \sqrt{t_0} / 2)$. Clearly, one can suppose that $t_0 \| w \| < \sqrt{t_0} / 2$.
Then for any $x \in X$ such that $d(0, G(x)) < \delta$ and for any $p \in (0, t_0)$ one has
$$
  d(0, G(x) - p w) \le d(0, G(x)) + p \| w \| < \delta + p_0 \| w \| < \sqrt{t_0}.
$$
Therefore for any such $x$ and $p$ one has
\begin{multline*}
  \varphi(x, p) = \frac{1}{p} \phi( d(0, G(x) - p w)^2 ) + \omega(p) \le 
  \frac{\Phi_0}{p} d(0, G(x) - p w)^2 + \Omega_0 p \le \\
  \le \frac{\Phi_0}{p} (d(0, G(x))^2 + 2 p \| w \| d(0, G(x)) + p^2 \| w \|^2) + \Omega_0 p.
\end{multline*}
Fix an arbitrary $x \in X$ such that $d(0, G(x)) < \delta$. If $d(0, G(x)) = 0$, then
$$
  \inf_{p \ge 0} \varphi(x, p) \le \inf_{p \in (0, t_0)} \varphi(x, p) \le
  \inf_{p \in (0, t_0)} \big( \Phi_0 \| w \|^2 + \Omega_0 \big) p = 0,
$$
which implies (\ref{InfimumPenTerm_UpperEstimate}). If $d(0, G(x)) > 0$, then denote
$$
  h(p) = \frac{\Phi_0}{p} (d(0, G(x))^2 + 2 p \| w \| d(0, G(x)) + p^2 \| w \|^2) + \Omega_0 p.
$$
Observe that the function $h(p)$ is continuously differentiable on $(0, + \infty)$, and $h(p) \to +\infty$ as 
$p \to +0$ or $p \to + \infty$. Therefore the function $h(p)$ attains a global minimum on $(0, + \infty)$ at a
point $p^* > 0$ that must satisfy the equality $h'(p^*) = 0$. Solving the equation $h'(p) = 0$ one obtains that
$$
  p^* = \frac{\sqrt{\Phi_0} d(0, G(x))}{ \sqrt{\Phi_0 \| w \|^2 + \Omega_0}} <
  \frac{\sqrt{\Phi_0} \delta}{ \sqrt{\Phi_0 \| w \|^2 + \Omega_0}}.
$$
Choosing $\delta \in (0, \sqrt{t_0} / 2)$ sufficiently small one obtains that $p^* < t_0$. Therefore for
any $x \in X$ such that $0 < d(0, G(x)) < \delta$ one has
\begin{multline*}
  \inf_{p \ge 0} \varphi(x, p) \le \inf_{p \in (0, t_0)} \varphi(x, p) \le \inf_{p \in (0, t_0)} h(p) =
  h(p^*) = \\
  = 2 ( \sqrt{\Phi_0^2 \| w \|^2 + \Phi_0 \Omega_0} + \Phi_0 \| w \| ) d(0, G(x)),
\end{multline*}
that completes the proof.	 
\end{proof}

\begin{lemma} \label{Lmm_LowerEstimPenTerm}
Let there exist $t_0, \phi_0, \omega_0 > 0$ such that
$$
  \phi(t) \ge \phi_0 t, \quad \omega(t) \ge \omega_0 t \quad \forall t \in [0, t_0].
$$
Then for any $(x, p) \in X \times \mathbb{R}_+$ such that
$\varphi(x, p) < \delta_0 := \min\big\{ \omega_0, \omega_0 t_0, \phi_0 t_0 \big\}$
one has
$$
  \varphi(x, p) \ge c_0 d(0, G(x)), \quad 
  c_0 := 
  \min\left\{ \frac{\omega_0}{\| w \|}, 2 \sqrt{\phi_0 \omega_0 + \phi_0^2 \| w \|^2} - 2 \phi_0 \| w \| \right\},
$$
where by definition $\omega_0 / \| w \| = + \infty$ in the case $w = 0$.
\end{lemma}

\begin{proof}
Fix $(x, p) \in X \times \mathbb{R}_+$ such that $\varphi(x, p) < \delta_0$. If $p = 0$, then $\varphi(x, 0) = 0$ in the
case $d(0, G(x)) = 0$, and $\varphi(x) = + \infty$ otherwise, which implies the validity of the inequality 
$\varphi(x, p) \ge c_0 d(0, G(x))$ for any $c_0 \ge 0$.

Suppose, now, that $p > 0$. Then
$$
  \varphi(x, p) = \frac{1}{p} \phi( d(0, G(x) - p w)^2 ) + \omega(p) < \delta_0.
$$
Hence $p^{-1} \phi( d(0, G(x) - p w)^2 ) < \delta_0$ and $\omega(p) < \delta_0$. From the fact that the function
$\omega$ is non-decreasing it follows that for any $s \ge t_0$ one has 
$\omega(s) \ge \omega(t_0) \ge \omega_0 t_0$. Therefore $p < t_0$ due to the definition of $\delta_0$, which implies
that $\omega_0 p \le \omega(p) <\delta_0$, and $p < 1$ by the fact that $\delta_0 < \omega_0$. 
Hence $\phi( d(0, G(x) - p w)^2 ) < \delta_0$. Applying the inequality
$\phi(t) \ge \phi_0 t$, and taking into account the facts that $\delta_0 < \phi_0 t_0$ and the function $\phi$ is
non-decreasing one can easily verify that $d(0, G(x) - p w)^2 < t_0$. Consequently, one has
$$
  \varphi(x, p) = \frac{1}{p} \phi( d(0, G(x) - p w)^2 ) + \omega(p) \ge 
  \frac{\phi_0}{p} d(0, G(x) - p w)^2 + \omega_0 p.
$$
If $ w \ne 0$ and $d(0, G(x)) < p \| w \|$, then from the last equality it follows that
\begin{equation} \label{LowerEstimOfPenTermNearOmega}
  \varphi(x, p) \ge \frac{\omega_0}{\| w \|} d(0, G(x)).
\end{equation}
On the other hand, if $d(0, G(x)) \ge p \| w \|$, then applying the inequality 
$d(0, G(x) - p w) \ge d(0, G(x)) - p \| w \|$ one obtains that
$$
  d(0, G(x) - p w)^2 \ge \big( d(0, G(x)) - p \| w \| \big)^2 = 
  d(0, G(x))^2 - 2 p \| w \| d(0, G(x)) + p^2 \| w \|^2,
$$
which yields
\begin{equation} \label{AuxilInequal}
  \varphi(x, p) \ge \frac{\phi_0}{p} d(0, G(x))^2 - 2 \phi_0 \| w \| d(0, G(x)) + 
  (\phi_0 \| w \|^2 + \omega_0) p.
\end{equation}
For any $p > 0$ define
$$
  \sigma(p) = \frac{\phi_0}{p} d(0, G(x))^2 - 
  2 \phi_0 \| w \| d(0, G(x)) + (\phi_0 \| w \|^2 + \omega_0) p.
$$
The function $\sigma$ is nonnegative, continuously differentiable on $(0, + \infty)$, and $\sigma(p) \to + \infty$ as
$p \to + 0$ or $p \to + \infty$. Therefore $\sigma$ attains a global minimum on $(0, + \infty)$ at a point
$p^*$ that must satisfy the equality $\sigma'(p^*) = 0$. Solving the equation $\sigma'(p) = 0$ one gets
$$
  p^* = \frac{\sqrt{\phi_0} d(0, G(x))}{\sqrt{\omega_0 + \phi_0 \| w \|^2}}.
$$
Hence and from (\ref{AuxilInequal}) it follows that
$$
  \varphi(x, p) \ge \sigma(p) \ge \min_{q > 0} \sigma(q) = \sigma(p^*) = 
  2 ( \sqrt{\phi_0 \omega_0 + \phi_0^2 \| w \|^2} - \phi_0 \| w \| ) d(0, G(x)
$$
Combining (\ref{LowerEstimOfPenTermNearOmega}) with the inequality above one obtains the desired result.	 
\end{proof}

\begin{lemma} \label{Lmm_Representation}
Let there exist $t_0, \phi_0, \omega_0 > 0$ such that
$$
  \phi(t) \ge \phi_0 t, \quad \omega(t) \ge \omega_0 t \quad \forall t \in [0, t_0].
$$
Then for any $\eta \in (0, \delta_0)$ one has
$$
  \Omega(\eta) = \big\{ x \in A \mid \inf_{p \in \mathbb{R}_+} \varphi(x, p) \le \eta \big\} \subseteq
  \big\{ x \in A \mid d(0, G(x)) \le \eta / c_0 \big\},
$$
where $\delta_0$ and $c_0$ are the same as in Lemma~\ref{Lmm_LowerEstimPenTerm}. Furthermore, if the penalty function
$F_{\lambda}$ is bounded below on $A \times \mathbb{R}_+$ for some $\lambda \ge 0$, then there exists 
$\lambda_0 \ge 0$ such that for any $\lambda \ge \lambda_0$ one has
$$
  \big\{ x \in A \mid \inf_{p \in \mathbb{R}_+} F_{\lambda}(x, p) < f^* \big\} \subseteq
  \big\{ x \in A \mid f(x) + \lambda c_0 d(0, G(x)) < f^* \big\}.
$$
\end{lemma}

\begin{proof}
Let $x \in \Omega(\eta)$ with some $\eta \in (0, \delta_0)$. Then for any $\varepsilon \in (0, \delta_0 - \eta)$ there
exists $p \ge 0$ such that $\varphi(x, p) < \eta + \varepsilon$. Therefore by Lemma~\ref{Lmm_LowerEstimPenTerm} one
has
$$
  \eta + \varepsilon > \varphi(x, p) \ge c_0 d(0, G(x)),
$$
which implies $d(0, G(x)) < \eta / c_0$ by the fact that $\varepsilon$ was chosen arbitrarily.

Suppose, additionally, that the penalty function $F_{\lambda}$ is bounded below on $A \times \mathbb{R}_+$ for
some $\lambda \ge 0$. Then applying Lemma~\ref{Lmm_Reduction} one obtains that there exists $\lambda_0 \ge 0$ such that
for any $\lambda \ge \lambda_0$ one has
\begin{equation} \label{PenFuncOutsideOmega}
  F_{\lambda}(x, p) \ge f^* \quad \forall (x, p) \notin \Omega_{\delta_0}.
\end{equation}
Choose arbitrary $\lambda \ge \lambda_0$ and $x \in A$ such that 
$\inf_{p \in \mathbb{R}_+} F_{\lambda}(x, p) < f^*$. Clearly, there exists $p \ge 0$ such that 
$F_{\lambda}(x, p) < f^*$. With the use of (\ref{PenFuncOutsideOmega}) one gets that $(x, p) \in \Omega_{\delta_0}$
or, equivalently, $\varphi(x, p) < \delta_0$. Hence applying Lemma~\ref{Lmm_LowerEstimPenTerm} one obtains that
$$
  f^* > F_{\lambda}(x, p) = f(x) + \lambda \varphi(x, p) \ge f(x) + \lambda c_0 d(0, G(x)),
$$
which completes the proof.	 
\end{proof}

\subsection{Global exactness of singular penalty functions}

Now, we can apply Theorems~\ref{Th_GlobalExGeneralCase} and \ref{Th_GlobExFiniteDim} in order to obtain new
sufficient conditions for the singular penalty function $F_{\lambda}$ to be globally exact. At first, we consider the
finite dimensional case and apply the localization principle.

\begin{theorem}
Let $X$ be a finite dimensional normed space, and the set $A$ be closed. Suppose that the following assumptions are
satisfied:
\begin{enumerate}
\item{$f$ is l.s.c. on $A$, and Lipschitz continuous near each of globally optimal solutions of 
the problem $(\mathcal{P})$;
}

\item{$G$ is metrically subregular with respect to the set $A$ at $(x^*, 0)$ for every globally optimal solution $x^*$
of
the problem $(\mathcal{P})$;
}

\item{the mapping $(x, p) \to d(0, G(x) - p w)$ is l.s.c. on $A \times \mathbb{R}_+$ (in particular, one can suppose
that $G$ is outer semicontinuous on $A$);
}

\item{the functions $\phi$ and $\omega$ are l.s.c., and there exist $\phi_0, \omega_0, t_0 > 0$ such that
$\phi(t) \ge \phi_0 t$ and $\omega(t) \ge \omega_0 t$ for any $t \in (0, t_0)$;
}

\item{either there exists $\delta > 0$ such that the set $\{ x \in A \mid d(0, G(x)) < \delta \}$ is bounded or there
exists $\mu \ge 0$ such that the set $\{ x \in A \mid f(x) + \mu d(0, G(x)) < f^* \}$ is bounded.
}
\end{enumerate}
Then the penalty function $F_{\lambda}$ is exact if and only if it is bounded below on $A \times \mathbb{R}_+$ for
some $\lambda \ge 0$.
\end{theorem}

\begin{proof}
From Theorem~\ref{Thrm_LocalExactness} it follows that the penalty function $F_{\lambda}$ is exact at every globally
optimal solution of the problem $(\mathcal{P})$. Applying Lemma~\ref{Lmm_LSC_Criterion} one gets that the penalty
term $\varphi$ is l.s.c. on $A \times \mathbb{R}_+$, while taking into account Lemma~\ref{Lmm_Representation} one
obtains that either the projection of the set $\Omega_{\eta}$ onto $X$ is bounded for sufficiently small $\eta > 0$
(note that the projection of $\Omega_{\eta}$ onto $X$ is contained in $\Omega(\eta)$) or the set
$\{ x \in A \mid \inf_{p \ge 0} F_{\mu}(x, p) < f^* \}$ is bounded for sufficiently large $\mu$. Hence with the
use of Theorem~\ref{Th_GlobExFiniteDim} one obtains the desired result.	 
\end{proof}

One can apply Theorem~\ref{Th_GlobalExGeneralCase} in order to obtain sufficient conditions for the singular penalty
function $F_{\lambda}$ under consideration to be exact in the general case. However, instead of applying
Theorem~\ref{Th_GlobalExGeneralCase} directly, we prove a useful auxiliary result about an intimate relation between
the exactness of the parametric penalty function $F_{\lambda}$ and the exactness of the standard penalty function
$h_{\lambda}(x) = f(x) + \lambda d(0, G(x))$, and then apply Theorem~\ref{Th_GlobalExGeneralCase} to the function
$h_{\lambda}$. This approach allows one to obtain stronger and, at the same time, simpler sufficient conditions for the
penalty function $F_{\lambda}$ to be globally exact than a direct application of Theorem~\ref{Th_GlobalExGeneralCase}.
Moreover, a theorem about a connection between the exactness of penalty functions $F_{\lambda}$ and $h_{\lambda}$ (the
first step in the approach that we use) is a very useful and important result on its own, since it allows one to reduce
the study of the singular penalty function $F_{\lambda}$ to the study of the standard penalty function $h_{\lambda}$.

The following result sharpens Theorems 7 and 9 from \cite{Dolgopolik_OptLet2}.

\begin{theorem} \label{Thrm_ReductionToStandPenFunc}
Let there exist $\phi_0, \Phi_0, \omega_0, \Omega_0, t_0 > 0$ such that
\begin{equation} \label{RHS_Derivatives_Inequal}
  \phi_0 t \le \phi(t) \le \Phi_0 t, \quad
  \omega_0 t \le \omega(t) \le \Omega_0 t  \quad \forall t \in [0, t_0]
\end{equation}
(in particular, one can suppose that there exist the right-hand side derivatives $\phi'_+(0)$ and $\omega'_+(0)$ of the
functions $\phi$ and $\omega$ at the origin such that $\phi'_+(0) > 0$ and $\omega'_+(0) > 0$). Then the penalty
function $F_{\lambda}$ is exact on $\Omega_{\delta}$ for some $\delta > 0$ if and only if the penalty function
$h_{\lambda}$ is exact on the set $\{ x \in A \mid d(0, G(x)) < \theta \}$ for some $\theta > 0$.
\end{theorem}

\begin{proof}
Let the penalty function $F_{\lambda}$ be exact on $\Omega_{\delta}$ for some $\delta > 0$. By definition, there exists 
$\lambda \ge 0$ such that $F_{\lambda}(x, p) \ge f^*$ for all $(x, p) \in \Omega_{\delta}$. Hence applying
Proposition~\ref{Prp_EquivDefExPen} one can verify that the penalty function 
$\Psi_{\lambda}(x, p) = f(x) + \lambda \psi_{\delta}(\varphi(x, p))$, where
\begin{equation} \label{BarrierTrasformDef}
  \psi_{\delta}(t) = \begin{cases}
    \dfrac{t}{\delta - t}, & \text{if } 0 \le t < \delta, \\
    + \infty, & \text{if } t \ge \delta.
  \end{cases}
\end{equation}
is exact, and $\lambda^*(f, \psi_{\delta} \circ \varphi) \le \delta \cdot \lambda^*( \Omega_{\delta} )$ (see
Remark~\ref{Rmrk_BarrierTermPenFunc}). Consequently, with the use of the fact that the function $\psi_{\delta}$ is
non-decreasing and continuous on its effective domain one obtains that
$$
  \inf_{p \ge 0} \Psi_{\lambda}(x, p) = f(x) + \lambda \psi_{\delta}\left( \inf_{p \ge 0} \varphi(x, p) \right) \ge f^*
  \quad \forall x \in A \quad \forall \lambda \ge \lambda^*(f, \psi_{\delta} \circ \varphi).
$$
By Lemma~\ref{Lmm_UpperEstimPenTerm} there exist $C_0 > 0$ and $\delta_0 > 0$ such that 
$\inf_{p \le 0} \varphi(x, p) \le C_0 d( 0, G(x) )$ for any $x \in X$ with $d(0, G(x)) < \delta_0$. Therefore taking
into account the facts that the function $\psi_{\delta}$ is non-decreasing and $\psi_{\delta}(t) \le 2 t / \delta$ for
any $t \in [0, \delta / 2]$ one gets that
$$
  f^* \le f(x) + \lambda \psi_{\delta}\left( \inf_{p \ge 0} \varphi(x, p) \right) \le
  f(x) + \lambda \psi_{\delta}\Big( C_0 d(0, G(x)) \Big) \le h_{\lambda C_1}(x),
$$
for any $x \in A$ with $d(0, G(x)) < \min\{ \delta_0, \delta / 2 C_0 \}$ and 
$\lambda \ge \lambda^*(f, \psi_{\delta} \circ \varphi)$, where $C_1 = 2 C_0 / \delta$. Thus, the penalty function
$h_{\lambda}$ is exact on the set $\{ x \in A \mid d(0, G(x)) < \theta \}$ with 
$\theta = \min\{ \delta_0, \delta / 2 C_0 \}$

Let, now, the penalty function $h_{\lambda}$ be exact on the set $\{ x \in A \mid d(0, G(x)) < \theta \}$ for some 
$\theta > 0$. Then the penalty function $b_{\lambda}(x) = f(x) + \lambda \psi_{\theta}(d(0, G(x))$ is exact. By
Lemma~\ref{Lmm_LowerEstimPenTerm} there exist $c_0 > 0$ and $\delta_0 > 0$ such that 
$\varphi(x, p) \ge c_0 d(0, G(x))$ for any $(x, p) \in A \times P$ with $\varphi(x, p) < \delta_0$. Hence taking into
account the facts that the function $\psi_{\theta}$ is non-decreasing and $\psi_{\theta}(t) \le 2 t / \theta$ for any
$t \in [0, \theta / 2]$ one obtains that
$$
  f^* \le b_{\lambda}(x) \le f(x) + \lambda \psi_{\theta}\left( \frac{1}{c_0} \varphi(x, p) \right) \le
  f(x) + \frac{2 \lambda}{\theta c_0} \varphi(x, p)
$$
for any sufficiently large $\lambda \ge 0$ and for all $(x, p) \in A \times P$ such that 
$\varphi(x, p) < \min\{ \delta_0, \theta c_0 / 2 \}$. Thus, the penalty function $F_{\lambda}$ is exact on the set
$\Omega_{\delta}$ with $\delta = \min\{ \delta_0, \theta c_0 / 2 \}$. 
\end{proof}

\begin{corollary}
Let there exist $\phi_0, \Phi_0, \omega_0, \Omega_0, t_0 > 0$ such that inequalities \eqref{RHS_Derivatives_Inequal} are
valid. Suppose also that there exists $\lambda_0 \ge 0$ such that the function $F_{\lambda_0}$ is bounded
below on $A \times \mathbb{R}_+$, and the function $h_{\lambda_0}$ is bounded below on $A$. Then the penalty
function $F_{\lambda}$ is exact if and only if the penalty function $h_{\lambda}$ is exact.
\end{corollary}

Now, we can utilize Theorem~\ref{Th_GlobalExGeneralCase} in order to derive sufficient conditions for the parametric
penalty function $F_{\lambda}$ to be exact in the general case.

\begin{theorem}
Let $X$ be a complete metric space, $A$ be a closed set, and the functions $f$ and $d(0, G(x))$ be l.s.c. on $A$ (in
particular, one can suppose that $G$ is outer semicontinuous on $A$). Suppose also that there exist $\delta > 0$ and 
$\lambda_0 \ge 0$ such that the following assumptions are satisfied:
\begin{enumerate}
\item{$f$ is Lipschitz continuous on the set
$$
  K(\delta, \lambda_0) := \big\{ x \in A \mid d(0, G(x)) < \delta, f(x) + \lambda_0 d(0, G(x)) < f^* \big\},
$$
and u.s.c. at every point of this set;
}

\item{there exists $a > 0$ such that $d(0, G(\cdot))^{\downarrow}_A(x) < - a$ for any $x \in K(\delta, \lambda_0)$;
}

\item{there exist $\phi_0, \Phi_0, \omega_0, \Omega_0, t_0 > 0$ such that inequalities \eqref{RHS_Derivatives_Inequal}
hold true;
}
\end{enumerate}
Then the penalty function $F_{\lambda}$ is exact if and only if it is bounded below on $A \times \mathbb{R}_+$ for
some $\lambda \ge 0$.
\end{theorem}

\begin{proof}
Applying Theorem~\ref{Th_GlobalExGeneralCase} to the penalty function 
$b_{\lambda}(x) = f(x) + \lambda \psi_{\delta}(d(0, G(x))$ (see~\eqref{BarrierTrasformDef}) one obtains that
$b_{\lambda}$ is exact. Therefore, as it is easy to see, the penalty function $h_{\lambda}$ is exact on the set 
$\{ x \in A \mid d(0, G(x)) < \theta \}$ for any $\theta < \delta$ (cf.~Remark~\ref{Rmrk_BarrierTermPenFunc}). Hence
applying Theorem~\ref{Thrm_ReductionToStandPenFunc} one obtains that the parametric penalty function $F_{\lambda}$ is
exact on the set $\Omega_{\tau}$ for some $\tau > 0$, which with the use of Proposition~\ref{Prp_EquivDefExPen} implies
the desired result.
\end{proof}

\begin{remark}
Note that the second assumption of the theorem above, in essence, means that the set-valued mapping $G$ is metrically
regular with respect to the set $A$ on the set $K(\delta, \lambda_0) \times \{ 0 \}$ (see~\cite{Ioffe,Kruger}).
\end{remark}

Let us also obtain sufficient conditions for the penalty function $F_{\lambda}$ under consideration to be
feasibility-preserving.

\begin{theorem} \label{Thrm_SingPenFunc_FeasPreserv}
Let $C \subset A$ be a nonempty set, and $f$ be Lipschitz continuous on an open set containing the set $C$. Suppose also
that the following assumptions are valid:
\begin{enumerate}
\item{the functions $\phi$ and $\omega$ are convex and continuously differentiable on $\dom \varphi$ and $\dom \omega$,
respectively, and $\phi'(0) > 0$ and $\omega'(0) > 0$;}

\item{there exists $a > 0$ such that $d(0, G(\cdot) - p w)^{\downarrow}_A (x) \le - a$ for any 
$(x, p) \in (C \setminus \Omega) \times \dom \omega$ such that $d(0, G(x) - p w) > 0$ and
$\phi(d(0, G(x) - p w)^2) < + \infty$.
}
\end{enumerate}
Then the penalty function $F_{\lambda}$ is feasibility-preserving on the set $C$.
\end{theorem}

\begin{proof}
Let us show that under the assumptions of the theorem there exists $c > 0$ such that
$$
  \varphi^{\downarrow}_{A \times \mathbb{R}_+}(x, p) < - c \quad 
  \forall (x, p) \in (C \times \mathbb{R}_+)_{\inf}.
$$
Then applying Proposition~\ref{Prp_SuffCondFeasPres} one obtains the required result.

Let $(x, p) \in (C \times \mathbb{R}_+)_{\inf}$ be arbitrary. Note that if $p = 0$, then $x \notin \Omega$, which
implies $\varphi(x, p) = + \infty$. Therefore $p > 0$. Furthermore, one has $\omega(p) < + \infty$ and 
$\phi(d(0, G(x) - p w)^2) < + \infty$ by the definition of the set $(C \times \mathbb{R}_+)_{\inf}$.

Suppose, at first, that $d(0, G(x) - p w) = 0$. Then by the second part of Lemma~\ref{Lemma_SimpleCases} the function 
$g(q) = d(0, G(x) - q w)^2$ is differentiable at the point $p$ and $g'(p) = 0$. Consequently, the function 
$q \to \varphi(x, q) = q^{-1} \phi(g(q)) + \omega(q)$ is differentiable at the point $p$ and
$$
  \frac{\partial \varphi}{\partial q}(x, p) = - \frac{1}{p^2} \phi(g(p)) + \frac{1}{p} \phi'(g(p)) g'(p) +
  \omega'(p) = \omega'(p)
$$
by the fact that $\phi(0) = g'(p) =  0$. Note that the function $\omega'(p)$ is non-decreasing due to the convexity of
$\omega$. Therefore
\begin{equation} \label{RSD_PenTerm_UpperEstimate_1}
  \varphi_{A \times \mathbb{R}_+}^{\downarrow}(x, p) \le \varphi(x, \cdot)^{\downarrow}_{\mathbb{R}_+}(p) =
  - \left| \frac{\partial \varphi}{\partial q}(x, p) \right| = - | \omega'(p) | \le - |\omega'(0)| < 0.
\end{equation}
Suppose, now, that $d(0, G(x) - p w) > 0$, and choose an arbitrary $\sigma > 0$. If
$$
  \frac{1}{p} \phi'( d(0, G(x) - p w)^2 ) d(0, G(x) - p w) \ge \sigma,
$$
then applying Lemma~\ref{Lemma_Superpos} one obtains that
\begin{equation} \label{RSD_PenTerm_UpperEstimate_2}
  \varphi(\cdot, p)^{\downarrow}_A (x) = \frac{2}{p} \phi'( d(0, G(x) - p w)^2 ) d(0, G(x) - p w)
  d(0, G(\cdot) - p w)^{\downarrow}_A(x) \le - \sigma a < 0.
\end{equation}
On the other hand, if
\begin{equation} \label{RSD_Estimate_TrickyInequal}
  \frac{1}{p} \phi'( d(0, G(x) - p w)^2 ) d(0, G(x) - p w) < \sigma,
\end{equation}
then introduce the functions $g_0(q) = d(0, G(x) - q w)$ and $h(q) = q^{-1} \phi( g_0(q)^2 )$. Clearly, the
function $g_0$ is Lipschitz continuous on $\mathbb{R}$ with a Lipschitz constant $L \le \| w \|$. Hence 
$g_0^{\uparrow}(p) \le \| w \|$. Therefore applying (\ref{RSD_Estimate_TrickyInequal}), and
Lemmas~\ref{Lemma_RSD_Product} and \ref{Lemma_Superpos} one obtains that
\begin{multline*}
  h^{\uparrow}(p) \le \frac{1}{p^2} \phi\big( (g_0(p))^2 \big) + 
  \frac{1}{p} \phi'(d(0, G(x) - p w)^2) d(0, G(x) - p w) g_0^{\uparrow}(p) \le \\
  \le \frac{1}{p^2} \phi\big( (g_0(p))^2 \big) + \sigma \| w \|.
\end{multline*}
Recall that the function $\phi$ is convex, which implies that its derivative is a non-decreasing function.
Consequently, with the use of (\ref{RSD_Estimate_TrickyInequal}) one gets that
\begin{multline*}
  \phi( d(0, G(x) - p w)^2 ) \le \phi' ( d(0, G(x) - p w)^2 ) d(0, G(x) - p w)^2 \le \\
  \le \sigma p d(0, G(x) - p w) \le \frac{\sigma^2 p^2}{\phi'(d(0, G(x) - p w)^2)} \le 
  \frac{\sigma^2 p^2}{\phi'(0)}.
\end{multline*}
Hence one has that $h^{\uparrow}(p) \le \sigma^2 / \phi'(0) + \sigma \| w \|$. Choosing $\sigma > 0$ sufficiently small
one gets that $h^{\uparrow}(p) \le \omega'(0) / 2$. Therefore applying Lemma~\ref{Lmm_SDD_SumEstimate} one obtains that
\begin{multline}  \label{RSD_PenTerm_UpperEstimate_3}
  \varphi^{\downarrow}_{A \times \mathbb{R}_+}(x, p) \le \varphi(x, \cdot)^{\downarrow}_{\mathbb{R}_+}(p) =
  (h + \omega)^{\downarrow}(p) \le \\
  \le h^{\uparrow}(p) + \omega^{\downarrow}(p) \le 
  \frac{\omega'(0)}{2} - | \omega'(p) | \le - \frac{\omega'(0)}{2}
\end{multline}
by virtue of the convexity of the function $\omega$. Combining (\ref{RSD_PenTerm_UpperEstimate_1}),
(\ref{RSD_PenTerm_UpperEstimate_2}) and (\ref{RSD_PenTerm_UpperEstimate_3}) one gets
$$
  \varphi^{\downarrow}_{A \times \mathbb{R}_+}(x, p) \le 
  - \min\left\{ \frac{\omega'(0)}{2}, \sigma a \right\} < 0 
  \quad \forall (x, p) \in (C \times \mathbb{R}_+)_{\inf},
$$
that completes the proof.	 
\end{proof}

\begin{remark}
Let $X$ be a normed space, $Y = \mathbb{R}^m \times \mathbb{R}^l$, and let the space $Y$ be equipped with an arbitrary
norm. Suppose also that the multifunction $G$ has the form
$$
  G(x) = \Big( \prod_{i = 1}^m \{ f_i(x) \} \Big) \times \Big( \prod_{j = 1}^l [g_j(x), + \infty) \Big), 
$$
where $\prod$ stands for the Cartesian product, and the functions $f_i, g_j \colon X \to \mathbb{R}$ are continuously
Fr\'echet differentiable. Thus, the inclusion $0 \in G(x)$ is equivalent to the following system of equations and
inequalities
\begin{equation} \label{EqualInequalConstraints}
  f_i(x) = 0, \quad i \in \{ 1, \ldots, m \}, \quad g_j(x) \le 0, \quad j \in \{ 1, \ldots, l \}.
\end{equation}
Let $x \in X$ and $p > 0$ be such that $d(0, G(x) - p w) > 0$. Then one can check that if MFCQ holds
true for the system (\ref{EqualInequalConstraints}) at the point $x$, then 
$d(0, G(\cdot) - p w)^{\downarrow}(x) < 0$. Furthermore, one can verify (cf.~\cite{HanMangasarian}, Theorem 2.2, and
\cite{DiPilloGrippo}, Lemma 3.1) that in this case there exist $r > 0$ and $a > 0$ such that
$$
  d(0, G(\cdot) - q w)^{\downarrow}(y) < - a  \quad 
  \forall (y, q) \in U(x, r) \times U(p, r) \colon d(0, G(y) - q w) > 0.
$$
Therefore, if MFCQ holds true for the system (\ref{EqualInequalConstraints}) at every point of a compact set 
$C \subset X$, then for any $t_0 > 0$ there exists $a > 0$ such that
$d(0, G(\cdot) - p w)^{\downarrow}(x) < - a$ for any $(x, p) \in (C \setminus \Omega) \times [0, t_0]$ such that
$d(0, G(x) - p w) > 0$. Hence assumption~2 of the previous theorem is satisfied provided that one of the set
$\dom \phi$ or $\dom \omega$ is bounded.

Usually, if in Theorem~\ref{Thrm_SingPenFunc_FeasPreserv} the sets $C$ and $\dom \phi \cap \dom \omega$ are relatively
compact, then assumption~2 of this theorem is satisfied provided that a constraint qualification holds true at every
point of the set $C \setminus \Omega$.
\end{remark}

\subsection{Zero duality gap and singular penalty functions}

Let us obtain two simple sufficient conditions for the zero duality gap property for the
penalty function $F_{\lambda}$ to hold true, and show how Theorem~\ref{Thrm_MinimizingSequences} about minimizing
sequences of a penalty function transforms in the case under consideration.

\begin{theorem}
Let there exist $\phi_0, \Phi_0, \omega_0, \Omega_0, t_0 > 0$ such that
$$
  \phi_0 t \le \phi(t) \le \Phi_0 t, \quad \omega_0 t \le \omega(t) \le \Omega_0 t 
  \quad \forall t \in [0, t_0].
$$
Then the zero duality gap property holds true for the penalty function $F_{\lambda}$ if and only if $F_{\lambda}$ is
bounded below for some $\lambda \ge 0$, and the perturbation function
$$
  \gamma(\eta) := \inf_{x \in K_{\eta}} f(x), \quad \eta \ge 0
$$
is l.s.c. at the origin, where $K_{\eta} = \{ x \in A \mid d(0, G(x)) \le \eta \}$.
\end{theorem}

\begin{proof}
Define, as above, the perturbation function $\beta(\eta) := \inf_{x \in \Omega(\eta)} f(x)$ for any
$\eta \ge 0$. Note that $\beta(0) = \gamma(0) = f^*$. By Lemma~\ref{Lmm_Representation}, there exist 
$\delta_1 \ge 0$ and $c_0 > 0$ such that for any $\eta \in (0, \delta_1)$ one has 
$\beta(\eta) \ge \gamma( \eta / c_0 )$. On the other hand, from Lemma~\ref{Lmm_UpperEstimPenTerm} it follows that there
exists $\delta_2 > 0$ and $C_0 > 0$ such that for any 
$\eta \in (0, \delta_2)$ one has $\{ x \in A \mid d(0, G(x)) \le \eta \} \subseteq \Omega( C_0 \eta )$, which implies
$\gamma(\eta) \ge \beta( C_0 \eta )$. Therefore, as it is easy to verify, the perturbation function $\gamma$ is l.s.c.
at the origin iff the perturbation function $\beta$ is l.s.c. at the origin. It remains to apply
Theorem~\ref{Thrm_ZeroDualityGapCharacterization}.	 
\end{proof}

\begin{theorem}
Let $A$ be closed, $f$ be l.s.c. on $\Omega$, $\phi$ and $\omega$ be l.s.c. on $\mathbb{R}_+$, and the mapping 
$(x, p) \to d(0, G(x) - p w)$ be l.s.c. on $A \times \mathbb{R}_+$ (in particular, one can suppose that $G$ is outer
semicontinuous on $A$). Suppose also that there exist $t_0, \phi_0, \omega_0 > 0$ such that
$$
  \phi(t) \ge \phi_0 t, \quad \omega(t) \ge \omega_0 t \quad \forall t \in [0, t_0],
$$
and there exists $\eta > 0$ such that the set $\{ x \in A \mid d(0, G(x)) < \eta, f(x) < f^* \}$ is relatively compact.
Then the zero duality gap property for the penalty function $F_{\lambda}$ holds true.
\end{theorem}

\begin{proof}
By Lemma~\ref{Lmm_LSC_Criterion} the penalty term $\varphi(x, p)$ is l.s.c. on $A \times \mathbb{R}_+$. Applying
Lemma~\ref{Lmm_Representation} one obtains that there exists $c_0 > 0$ such that 
$$
  \{ x \in \Omega( c_0 \eta ) \mid f(x) < f^* \} \subset \{ x \in A \mid d(0, G(x) < \eta, f(x) < f^* \},
$$
provided $\eta > 0$ is sufficiently small. Consequently, applying Proposition~\ref{Prp_ZeroDualityGapSuffCond}
one gets that the zero duality gap property holds true for $F_{\lambda}$.	 
\end{proof}

\begin{theorem}
Let $\{ \lambda_n \} \subset (0, + \infty)$ be an increasing unbounded sequence. Let also a sequence 
$\{ (x_n, p_n) \} \subset A \times [0, +\infty)$ be such that
$$
  F_{\lambda_n}(x_n, p_n) \le 
  \inf_{(x, p) \in A \times \mathbb{R}_+} F_{\lambda_n}(x, p) + \varepsilon
$$
for some $\varepsilon > 0$. Then $p_n \to 0$ and $d(0, G(x_n)) \to 0$ as $n \to \infty$. Moreover, if,
additionally, the zero duality gap property holds true for the penalty function $F_{\lambda}$, then 
$f^* \le \liminf_{n \to \infty} f(x_n) \le \limsup_{n \to \infty} f(x_n) \le f^* + \varepsilon$.
\end{theorem}

\begin{proof}
From Theorem~\ref{Thrm_MinimizingSequences} it follows that $p_n \to 0$ and $\varphi(x_n, p_n) \to 0$ as 
$n \to \infty$, and $f^* \le \liminf_{n \to \infty} f(x_n) \le \limsup_{n \to \infty} f(x_n) \le f^* + \varepsilon$, if
the zero duality gap property holds true. Thus, it remains to show that $d(0, G(x_n)) \to 0$ as $n \to \infty$.
Arguing by reductio ad absurdum, suppose that this is not true. Then there exist $\sigma > 0$ and a
subsequence $\{ x_{n_k} \}$ such that $d(0, G(x_{n_k})) > \sigma$ for any $k \in \mathbb{N}$.

Note that if $p_{n_k} = 0$ for some $k \in \mathbb{N}$, then $d(0, G(x_{n_k})) = 0$, since otherwise 
$\varphi(x_{n_k}, p_{n_k}) = + \infty$, which is impossible. Consequently, $p_{n_k} > 0$ for all $k \in \mathbb{N}$.
Since $p_n \to 0$ as $n \to \infty$, for any sufficiently large $k$ one has 
$d(0, G(x_{n_k}) - p_{n_k} w) > \sigma / 2$. Thererefore taking into account the fact that the function $\phi$ is
non-decreasing one obtains that for any $k$ large enough the following inequality holds true
$$
  \varphi(x_{n_k}, p_{n_k}) = \frac{1}{p_{n_k}} \phi( d(0, G(x_{n_k}) - p_{n_k} w)^2 ) + \omega(p_{n_k}) \ge
  \frac{1}{p_{n_k}} \phi(\sigma / 2).
$$
Consequently, $\varphi(x_{n_k}, p_{n_k}) \to + \infty$ as $k \to \infty$, which contradicts the fact that
$\varphi(x_n, p_n) \to 0$ as $n \to \infty$. Thus, $d(0, G(x_n)) \to 0$ as $n \to \infty$.	 
\end{proof}

\section{Smoothing approximations of nonsmooth penalty functions}
\label{Section_SmoothingPenFunc}

In this section, we discuss a general approach to the construction of exact parametric penalty functions based on 
the use of smoothing approximations of standard exact penalty functions, and study some properties of approximations of
penalty functions.

Let $g_{\lambda}(x) = f(x) + \lambda \phi(x)$ be a standard penalty function for the problem $(\mathcal{P})$, i.e.
let $\phi \colon X \to [0, + \infty]$ be a given function such that $\phi(x) = 0$ iff $x \in M$. Let also
$P$ be a metric space, and $p_0 \in P$ be fixed.

\begin{definition}
A function $\Phi \colon X \times P \to [0, + \infty]$ is called an \textit{approximation} of the
penalty term $\phi$, if $\Phi(\cdot, p_0) = \phi(\cdot)$, and $\Phi(\cdot, p) \to \phi(\cdot)$ as $p \to p_0$ uniformly
on the set $A$. An approximation $\Phi$ of $\phi$ is called an \textit{upper approximation} of the penalty term $\phi$
if $\Phi(x, p) \ge \phi(x)$ for all $(x, p) \in X \times P$.
\end{definition}

Approximations of the penalty term $\phi$ can be utilized to construct parametric penalty functions for the
problem $(\mathcal{P})$. Indeed, let $\Phi$ be an approximation of the penalty term $\phi$. Choose a
non-decreasing function $\omega \colon \mathbb{R}_+ \to [0, + \infty]$ such that $\omega(t) = 0$ iff $t = 0$, and define
$\varphi(x, p) = \Phi(x, p) + \omega(d(p_0, p))$ for any $(x, p) \in X \times P$. Clearly, $\varphi(x, p) = 0$ iff
$x \in \Omega$, and $p = p_0$. Therefore the function
\begin{equation} \label{ParamPenFuncViaUpperApprox}
  F_{\lambda}(x, p) = f(x) + \lambda \varphi(x, p) = f(x) + \lambda \big( \Phi(x, p) + \omega(d(p_0, p)) \big)
\end{equation}
is a parametric penalty function for the problem $(\mathcal{P})$.

Note that if $\Phi$ is an upper approximation of the penalty term $\phi$, then
$$
  \inf_{p \in P} \varphi(x, p) = \inf_{p \in P} \big( \Phi(x, p) + \omega(d(p_0, p)) \big) = 
  \phi(x) \quad \forall x \in X.
$$
Therefore, in this case, the parametric penalty function $F_{\lambda}$ is exact if and only if the penalty function
$g_{\lambda}$ is exact by virtue Corollary~\ref{Crlr_ReductionToStandExPen}. Thus, an approach based on upper
approximations of the penalty term $\phi$ provides one with a simple method for constructing exact parametric penalty
functions.

In the theory of smoothing approximations of exact penalty functions (see, e.g.,
\cite{Pinar,WuBaiYang,MengDangYang,Liu,LiuzziLucidi,XuMengSunShen,Lian,XuMengSunHuangShen} and
references therein), one defines $P = \mathbb{R}_+$, and chooses an approximation $\Phi$ of the penalty term $\phi$ such
that the function $\Phi(\cdot, p)$ is smooth for any $p \in (0, + \infty)$. Then one chooses a decreasing sequence 
$\{ p_n \} \subset (0, + \infty)$ such that $p_n \to 0$ as $n \to \infty$, and uses any method of smooth optimization in
order to find a point of global minimum $x_n$ of the function $f(\cdot) + \lambda \Phi(\cdot, p_n)$. Under some
additional assumptions one can show that the sequence $\{ x_n \}$ converges to a globally optimal solution of the
problem
$(\mathcal{P})$. Thus, smoothing approximations allow one to avoid the minimization problem for the
\textit{nondifferentiable} exact penalty function $g_{\lambda}$. On the other hand, a smoothing approximation of an
exact penalty function often enjoys some good properties that standard smooth penalty functions do note have 
(see \cite{Pinar,WuBaiYang,MengDangYang,Liu,LiuzziLucidi,XuMengSunShen,Lian,XuMengSunHuangShen}).

From the viewpoint of the theory of exact parametric penalty functions, one can consider the method of smoothing
approximation of exact penalty functions as follows. One defines the parametric penalty function $F_{\lambda}$ of
the form (\ref{ParamPenFuncViaUpperApprox}), and then applies a ``coordinate'' descent method (with $x$ and $p$ being
two ``coordinates'') to find a point of global minimum of this function. The ``coordinate'' descent method is used in
order to utilize the smoothness of the function $\Phi$ in $x$, while the term $\omega$ is not used because it does not
affect the optimization process.

Let us study how a minimizing sequence of the parametric penalty function $F_{\lambda}$ constructed via the
``coordinate'' descent method behaves as $n \to \infty$. The propositions below unify and significantly generalize many
theorems on convergence of minimization methods based on the use of smoothing approximations of exact penalty functions
(cf.~Proposition~4 in \cite{Pinar}, Theorem~ 2.1 and 3.1 in \cite{MengDangYang}, Theorem~2.1 in \cite{Liu}, etc.).

\begin{proposition} \label{Prp_CoordDescentApproxPenFunc}
Let the penalty function $g_{\lambda}$ be bounded below on $A$, $\Phi$ be an approximation of the penalty term
$\phi$, and $F_{\lambda}(x, p) = f(x) + \lambda \Phi(x, p)$. Let a sequence $\{ p_n \} \subset P$ converge to the
point $p_0$, and let a sequence $\{ x_n \} \subset A$ satisfy the inequality
$$
  F_{\lambda}(x_n, p_n) \le \inf_{x \in A} F_{\lambda}(x, p_n) + \varepsilon.
$$
for some $\varepsilon \ge 0$ and $\lambda > 0$. Then
\begin{equation} \label{CoordDescentApproxPenFunc_Ineq}
  \inf_{x \in A} g_{\lambda}(x) \le \liminf_{n \to \infty} g_{\lambda}(x_n) \le
  \limsup_{n \to \infty} g_{\lambda} (x_n) \le \inf_{x \in A} g_{\lambda}(x) + \varepsilon.
\end{equation}
If, additionally, $f$ and $\phi$ are l.s.c. on $A$, and the set $A$ is closed, then any cluster point $x^*$ of the
sequence $\{ x_n \}$ (if exists) belongs to the set $A$, and satisfies the inequality 
$g_{\lambda}(x^*) \le \inf_{x \in A} g_{\lambda}(x) + \varepsilon$.
\end{proposition}

\begin{proof}
By the definition of approximation, $\Phi(\cdot, p) \to \phi(\cdot)$ as $p \to p_0$ uniformly on the set $A$. Therefore
for any $\sigma > 0$ there exists $n_0 \in \mathbb{N}$ such that $| \Phi(x, p_n) - \phi(x) | < \sigma$ for all 
$n \ge n_0$ and $x \in A$. Consequently, for any $n \ge n_0$ the function $F_{\lambda}(\cdot, p_n)$ is
bounded below on $A$, and
$$
  \big| \inf_{x \in A} F_{\lambda}(x, p_n) - \inf_{x \in A} g_{\lambda}(x) \big| \le \sigma
$$
Hence and from the definition of the sequence $\{ x_n \}$ it follows that
$$
  \inf_{x \in A} g_{\lambda}(x) - \sigma \le F_{\lambda}(x_n, p_n) \le 
  \inf_{x \in A} g_{\lambda}(x) + \varepsilon + \sigma \quad \forall n \ge n_0.
$$
Then applying the fact that $|F_{\lambda}(x_n, p_n) - g_{\lambda}(x_n)| = | \Phi(x_n, p_n) - \phi(x_n) | < \sigma$
for all $n \ge n_0$ one obtains that inequalities (\ref{CoordDescentApproxPenFunc_Ineq}) are valid by virtue of the
fact that $\sigma > 0$ is arbitrary.

If, additionally, the functions $f$ and $\phi$ are l.s.c. on $A$, and the set $A$ is closed, then taking into account
inequalities (\ref{CoordDescentApproxPenFunc_Ineq}) one obtains that any cluster point $x^*$ of the
sequence $\{ x_n \}$ satisfies the inequality $g_{\lambda}(x^*) \le \inf_{x \in A} g_{\lambda}(x) + \varepsilon$.
\end{proof}

\begin{proposition}
Let the penalty function $g_{\lambda}$ be exact, $\Phi$ be an approximation of the penalty term $\phi$, and
$F_{\lambda}(x, p) = f(x) + \lambda \Phi(x, p)$. Let a sequence $\{ p_n \} \subset P$ converge to the
point $p_0$, and let a sequence $\{ x_n \} \subset A$ satisfy the inequality
$$
  F_{\lambda}(x_n, p_n) \le \inf_{x \in A} F_{\lambda}(x, p_n) + \varepsilon_n.
$$
for some $\lambda > \lambda^*(f, \phi)$, and a decreasing sequence $\{ \varepsilon_n \} \subset (0, + \infty)$ such
that $\varepsilon_n \to 0$ as $n \to \infty$.  Then $\phi(x_n) \to 0$ as $n \to \infty$, and 
$\lim_{n \to \infty} f(x_n) = f^*$. If, additionally, $f$ and $\phi$ are l.s.c. on $A$, and the set $A$ is closed, then
any cluster point $x^*$ of the sequence $\{ x_n \}$ (if exists) is a globally optimal solution of the problem
($\mathcal{P}$).
\end{proposition}

\begin{proof}
With the use of Proposition~\ref{Prp_CoordDescentApproxPenFunc} one can easily obtain that
$$
  \lim_{n \to \infty} g_{\lambda}(x_n) = \inf_{x \in A} g_{\lambda}(x).
$$
Then applying Proposition~\ref{Prp_MinimizingSeq_ExactPenFunc} one gets the required result.	 
\end{proof}

Let $\Phi$ be an approximation of the penalty term $\phi$. Denote
$$
  e(p) = \sup_{x \in A} |\Phi(x, p) - \phi(x)|
$$
By the definition of approximation one has $e(p) \to 0$ as $p \to p_0$. One can use the function $e(\cdot)$ in order to
understand how well the parametric penalty function $F_{\lambda}(x, p) = f(x) + \Phi(x, p)$ approximates the
standard penalty function $g_{\lambda}$.

The following proposition and its corollary unify many results on smoothing approximations of exact penalty function
(see, e.g., \cite{Pinar}, Propositions~1 and 2, \cite{MengDangYang}, Theorems~2.2, 2.3, 3.2, 3.3, 
\cite{Liu}, Theorems~2.2, 2.3 and 3.1, etc.)

\begin{proposition} \label{Prp_ApproxPenFuncEstimate}
Let the penalty function $g_{\lambda_0}$ be bounded below on the set $A$. Then for any $\lambda \ge \lambda_0$ and
for all $p \in \dom e$ one has
\begin{equation} \label{InfEstimateApproxPenFunc}
  \big| \inf_{x \in A} F_{\lambda}(x, p) - \inf_{x \in A} g_{\lambda}(x) \big| \le \lambda e(p).
\end{equation}
Furthermore, if the penalty function $g_{\lambda}$ is exact and $\lambda > \lambda^*(f, \phi)$, then for any 
$p \in \dom e$ such that the function $F_{\lambda}(\cdot, p)$ attains a global minimum on the set $A$, and for any 
$x_p \in \argmin_{x \in A} F_{\lambda}(x, p)$ one has
$$
  | f^* - f(x_p) | \le \lambda \big( e(p) + \Phi(x_p, p) \big).
$$
\end{proposition}

\begin{proof}
Fix arbitrary $\lambda \ge \lambda_0$ and $p \in \dom e$. For any $x \in A$ one has
\begin{equation} \label{UniformEstimateApproxPenFunc}
  |F_{\lambda}(x, p) - g_{\lambda}(x, p)| = \lambda | \Phi(x, p) - \phi(x) | \le \lambda e(p).
\end{equation}
Consequently, taking into account the fact that the penalty function $g_{\lambda}$ is bounded below on the set $A$
one obtains that the function $F_{\lambda}(\cdot, p)$ is bounded below on the set $A$ as well.

By the definition of infimum, for any $\varepsilon > 0$ there exists $x_{\varepsilon}$ such that
$$
  F_{\lambda}(x_{\varepsilon}, p) \le \inf_{x \in A} F_{\lambda}(x, p) + \varepsilon.
$$
Hence and from \eqref{UniformEstimateApproxPenFunc} one obtains that
$$
  \lambda e(p) \ge g_{\lambda}(x_{\varepsilon}) - F_{\lambda}(x_{\varepsilon}, p) \ge
  \inf_{x \in A} g_{\lambda}(x) - \inf_{x \in A} F_{\lambda}(x, p) - \varepsilon.
$$
Passing to the limit as $\varepsilon \to + 0$ one gets
$$
  \inf_{x \in A} g_{\lambda}(x) - \inf_{x \in A} F_{\lambda}(x, p) \le \lambda e(p). 
$$
Arguing in a similar way one can easily verify that the inequality
$$
  \inf_{x \in A} F_{\lambda}(x, p) - \inf_{x \in A} g_{\lambda}(x) \le \lambda e(p)
$$
is valid, which implies that \eqref{InfEstimateApproxPenFunc} holds true.

If the penalty function $g_{\lambda}$ is exact and $\lambda > \lambda^*(f, \phi)$, then 
$\inf_{x \in A} g_{\lambda}(x) = f^*$ by virtue of Proposition~\ref{Prp_EquivDefExPen}. Therefore applying
inequaliaty \eqref{InfEstimateApproxPenFunc} one obtains that for any $p \in \dom e$ such that the function
$F_{\lambda}(\cdot, p)$ attains a global minimum on the set $A$, and for any 
$x_p \in \argmin_{x \in A} F_{\lambda}(x, p)$ one has $| f^* - F_{\lambda}(x_p, p) | \le \lambda e(p)$.
Consequently, taking into account the fact that the function $\Phi$ is nonnegative one gets that the following
inequalities hold true
$$
  \lambda e(p) \ge | f^* - F_{\lambda}(x_p, p) | = |f^* - f(x_p) - \lambda\Phi(x_p, p)| \ge
  |f^* - f(x_p)| - \lambda \Phi(x_p, p),
$$
which implies the desired result.	 
\end{proof}

\begin{proposition}
Suppose that the penalty function $g_{\lambda_0}$ is bounded below on the set $A$ for some $\lambda_0 \ge 0$. Let a
sequence $\{ p_n \} \subset P$ converge to $p_0$, and let $\{ \lambda_n \} \subset (\lambda_0, + \infty)$ be
an increasing unbounded sequence such that $\lambda_n e(p_n) \to 0$ as $n \to \infty$. Suppose, finally, that 
a sequence $\{ x_n \} \subset A$ satisfies the inequality
$$
  F_{\lambda_n}(x_n, p_n) \le \inf_{x \in A} F_{\lambda_n}(x, p_n) + \varepsilon
$$
for any sufficiently large $n \in \mathbb{N}$, and for some $\varepsilon > 0$. Then $\phi(x_n) \to 0$ as $n \to \infty$.
Moreover, if the functions $f$ and $\phi$ are l.s.c. on $A$, and the set $A$ is closed, then any cluster point $x^*$ of
the sequence $\{ x_n \}$ is a feasible point of the problem $(\mathcal{P})$, and $f^* \le f(x^*) \le f^* + \varepsilon$.
\end{proposition}

\begin{proof}
Taking into account the facts that the function $g_{\lambda_0}$ is bounded below on $A$, and $e(p) \to 0$ as 
$p \to \infty$ one obtains that the function $F_{\lambda_n}(\cdot, p_n)$ is bounded below for any sufficiently
large $n \in \mathbb{N}$. Thus, the sequence $\{ x_n \}$ is correctly defined.

Recall that $g_{\lambda}$ is a penalty function for the problem $(\mathcal{P})$. Therefore for any $\lambda \ge 0$ one
has $\inf_{x \in A} g_{\lambda}(x) \le f^*$. Hence and from Proposition~\ref{Prp_ApproxPenFuncEstimate} it follows that
\begin{equation} \label{UpperEstimApproxPenFuncMinSeq}
  F_{\lambda_n}(x_n, p_n) \le \inf_{x \in A} F_{\lambda_n}(x, p_n) + \varepsilon
  \le \inf_{x \in A} g_{\lambda_n}(x) + \varepsilon + \lambda_n e(p_n) \le f^* + \varepsilon + \lambda_n e(p_n)
\end{equation}
for any sufficiently large $n$, which implies that the sequence $\{ F_{\lambda_n}(x_n, p_n) \}$ is bounded above
due to the fact that $\lambda_n e(p_n) \to 0$ as $n \to \infty$.

Observe that for any $n \in \mathbb{N}$ one has
\begin{multline*}
  F_{\lambda_n}(x_n, p_n) = F_{\lambda_0}(x_n, p_n) + (\lambda_n - \lambda_0) \Phi(x_n, p_n) \ge \\
  \ge g_{\lambda_0}(x_n) - \lambda_0 e(p_n) + (\lambda_n - \lambda_0) \Phi(x_n, p_n).
\end{multline*}
Consequently, taking into account the facts that the function $g_{\lambda_0}$ is bounded below on the set $A$, and
$e(p_n) \to 0$ as $n \to \infty$ one obtains that $\Phi(x_n, p_n) \to 0$ as $n \to \infty$, since otherwise
$\limsup_{n \to \infty} F_{\lambda_n}(x_n, p_n) = + \infty$, which is impossible. Therefore applying the inequality
$|\Phi(x_n, p_n) - \phi(x_n)| \le e(p_n)$ one gets that $\phi(x_n) \to 0$ as $n \to \infty$.

Suppose, now, that the functions $f$ and $\phi$ are l.s.c. on $A$, and the set $A$ is closed. Then applying the fact
that $\phi(x_n) \to 0$ as $n \to \infty$ one obtains that any cluster point
$x^*$ of the sequence $\{ x_n \}$ must satisfy the equality $\phi(x^*) = 0$, which implies that $x^*$ is a feasible
point of the problem $(\mathcal{P})$. Hence $f(x^*) \ge f^*$. On the other hand, with the use of
(\ref{UpperEstimApproxPenFuncMinSeq}), and the fact that the function $\Phi$ is nonnegative one gets that 
$f(x_n) \le f^* + \varepsilon + \lambda_n e(p_n)$. Passing to the limit as $n \to \infty$ one obtains the desired
result.	 
\end{proof}

In the end of this section, we provide two examples of smoothing approximations of exact penalty functions, and
exact parametric penalty functions constructed with the use of these smoothing approximations. Let, for the sake of
simplicity, $X = \mathbb{R}^d$, and suppose that the problem $(\mathcal{P})$ has the form
\begin{equation} \label{InequalConstrOptimization}
  \min f(x) \quad \text{subject to} \quad g_i(x) \le 0, \quad i \in \{ 1, \ldots, m \},
\end{equation}
where $g_i \colon \mathbb{R}^d \to \mathbb{R}$ are given functions. Let also $P = \mathbb{R}_+$ and $p_0 = 0$.

\begin{example} \cite{Liu}
Let $\phi$ be the $\ell_1$ penalty term for the problem (\ref{InequalConstrOptimization}), i.e. let
$$
  \phi(x) = \sum_{i = 1}^m \max\{ 0, g_i(x) \}.
$$
For any $p > 0$ define
$$
  \theta(t, p) = \begin{cases}
    \dfrac{1}{2} p e^{t/p}, & \text{if } t \le 0, \\
    t + \dfrac{1}{2} p e^{- t/p}, & \text{if } t > 0,
  \end{cases}
$$
and define $\theta(t, 0) = \max\{ 0, t \}$. Set
$$
  \Phi(x, p) = \sum_{i = 1}^m \theta(g_i(x), p) \quad \forall x \in \mathbb{R}^n.
$$
Note that $0 \le \theta(t, p) - \max\{ 0, t \} \le p/2$ for any $t \in \mathbb{R}$. Therefore the function $\Phi$ is an
upper approximation of the penalty term $\phi$, and $e(p) = p / 2$ for any $p \ge 0$. 

Thus, the parametric penalty function $F_{\lambda}(x, p) = f(x) + \lambda \big( \Phi(x, p) + p \big)$ is exact if and
only if the $\ell_1$ penalty function $h_{\lambda}(x) = f(x) + \lambda \phi(x)$ is exact. Furthermore, observe that if
the functions $f$ and $g_i$ are twice continuously differentiable on $\mathbb{R}^d$, then the function
$F_{\lambda}(\cdot, p)$ is twice continuously on $\mathbb{R}^n$ for any $p > 0$.
\end{example}

\begin{example} \cite{LiuzziLucidi}
Let $\phi$ be the $\ell_{\infty}$ penalty term for the problem (\ref{InequalConstrOptimization}), i.e. let
$$
  \phi(x) = \max\left\{ 0, g_1(x), \ldots, g_m(x) \right\}.
$$
For any $p > 0$ define
$$
  \Phi(x, p) = p \ln \Big( \sum_{i = 1}^m \exp\big( g_i(x) / p \big) \Big)
$$
(see, e.g., \cite{Xu}), and set $\Phi(x, 0) = \phi(x)$. Note that for any $p > 0$ and $x \in \mathbb{R}^d$ one has
$$
  \Phi(x, p) = \phi(x) + p \ln \Big( \sum_{i = 1}^m \exp\big( (g_i(x) - \phi(x)) / p \big) \Big) \le \phi(x) + p \ln m,
$$
if $\phi(x) \ne 0$, and $0 \le \Phi(x, p) \le p \ln m$, otherwise. Hence
$$
  \phi(x) \le \Phi(x, p) \le \phi(x) + p \ln m \quad \forall (x, p) \in \mathbb{R}^d \times \mathbb{R}_+.
$$
Therefore the function $\Phi$ is an upper approximation of the penalty term $\phi$, and $e(p) \le p \ln m$.

Thus, the parametric penalty function $F_{\lambda}(x, p) = f(x) + \lambda \big( \Phi(x, p) + p \big)$ is exact if and
only if the $\ell_{\infty}$ penalty function $h_{\lambda}(x) = f(x) + \lambda \phi(x)$ is exact. Moreover, note
that if the functions $f$ and $g_i$ are twice continuously differentiable on $\mathbb{R}^d$, then the function
$F_{\lambda}$ is twice continuously differentiable on $\mathbb{R}^d \times (0, + \infty)$.
\end{example}

\bibliographystyle{abbrv}  
\bibliography{ParamPenFunc}

\end{document}